\documentclass{amsart}

\usepackage[english]{babel}
\usepackage{amssymb}
\usepackage{mathtools}
\usepackage{enumitem}
\usepackage{mathrsfs}
\usepackage[breaklinks]{hyperref}
\usepackage{xcolor}
\usepackage[normalem]{ulem}
\usepackage[bbgreekl]{mathbbol} 
\usepackage{comment}
\usepackage{fullpage}
\usepackage[paper]{knowledge}
\knowledgeconfigure{notion, quotation}

\newcommand{\N}{\mathbb{N}}
\newcommand{\R}{\mathbb{R}}
\newcommand{\tot}{\textrm{total}}

\renewcommand{\vec}[1]{{\bf #1 } }
\newcommand{\citep}[1]{\cite{#1}}
\newcommand{\sys}[1]{\left \{ \begin{aligned} #1 \end{aligned} \right. }
\newcommand{\saut}{{\color{white} h} \\}

\newcommand{\ini}{\mathrm{in}}

\newcommand{\sw}{\mathrm{sw}}
\newcommand{\coo}[1]{\begin{pmatrix} #1 \end{pmatrix}}

\knowledgenewrobustcmd{\etab}{\cmdkl{\underline{\eta}}}
\knowledgenewrobustcmd{\vb}{\cmdkl{\underline{v}}}
\knowledgenewrobustcmd{\Vb}{\cmdkl{\overline{V}}}
\knowledgenewrobustcmd{\Ub}{\cmdkl{ { \bf \underline{U}}}}
\knowledgenewrobustcmd{\Pb}{\cmdkl{\underline{P}}}
\newcommand{\rhob}{\underline{\rho}}
\newcommand{\rhotot}{\rhob + \epsilon \delta \rho}
\knowledgenewrobustcmd{\E}{\cmdkl{\mathcal{E}}}
\knowledgenewrobustcmd{\Elow}{\cmdkl{\mathcal{E}_0}}
\knowledgenewrobustcmd{\Ediff}{\cmdkl{\mathcal{F}}}
\knowledgenewrobustcmd{\Edifflow}{\cmdkl{\mathcal{F}_0}}
\newcommand{\hb}{\underline{h}}

\knowledgenewrobustcmd{\A}{\cmdkl{A_{\mu}}(\partial \eta)}
\newcommand{\Mb}{\underline{\text{M}}}
\newcommand{\nablamu}{\nabla_{x,r}^{\mu}}
\knowledgenewrobustcmd{\jac}{\cmdkl{J}}
\knowledgenewcommand{\gradphi}[1][]{\cmdkl{\nabla}_{#1}^{\cmdkl{\varphi}}}
\knowledgenewcommand{\partialphi}[1]{\cmdkl{\partial}_{#1}^{\cmdkl{\varphi}}}

\knowledgenewcommand{\vort}{\cmdkl{\boldsymbol{\omega}_{\mu}}}
\newcommand{\point}[2]{{#1}^{#2}}
\newcommand{\pointal}[2]{{#1}^{(#2)}}
\newcommand{\pointgen}[1]{\dot{#1}}

\newcommand{\pointf}{\pointgen{\vec{f}}}
\newcommand{\taylorgen}{\underline{\taylor}}

\newcommand{\vortsw}{\omega_{\sw}}
\newcommand{\Vortsw}{\boldsymbol{\omega}_{\sw}}
\newcommand{\Vdiff}{\tilde{V}}
\newcommand{\wdiff}{\tilde{w}}
\newcommand{\etadiff}{\tilde{\eta}}
\newcommand{\Pdiff}{P}
\newcommand{\vortdiff}{\tilde{\boldsymbol{\omega}}_{\mu}}

\newcommand{\Vdiffs}{\tilde{V}^{(s)}}
\newcommand{\wdiffs}{\tilde{w}^{(s)}}
\newcommand{\etadiffs}{\tilde{\eta}^{(s)}}
\newcommand{\Pdiffs}{P^{(s)}}

\newcommand{\reg}[1]{\mathcal{J}_{\iota_{#1}}}
\knowledgenewrobustcmd{\commal}[3]{\cmdkl{[#1,#2]^{\mathrm{al}} #3 }}

\knowledgenewrobustcmd{\F}{\cmdkl{\vec{F}}}
\knowledgenewrobustcmd{\Fun}{\cmdkl{F_1}}
\knowledgenewrobustcmd{\Fdeux}{\cmdkl{f_2}}
\knowledgenewrobustcmd{\Ftrois}{\cmdkl{f_3}}
\knowledgenewrobustcmd{\Fquatre}{\cmdkl{f_4}}
\newcommand{\Fcinq}{\vec{F_5}}

\knowledgenewrobustcmd{\Rzero}{\cmdkl{r_0}}
\knowledgenewrobustcmd{\Run}{\cmdkl{R_1}}
\knowledgenewrobustcmd{\Rdeux}{\cmdkl{R_2}}
\knowledgenewrobustcmd{\Rdeuxetdemi}{\cmdkl{R_{6}}}
\knowledgenewrobustcmd{\Rtrois}{\cmdkl{R_3}}
\knowledgenewrobustcmd{\Rquatre}{\cmdkl{R_4}}
\knowledgenewrobustcmd{\Rcinq}{\cmdkl{R_5}}
\knowledgenewrobustcmd{\Rsix}{\cmdkl{\vec{R}_7}}

\newcommand{\diff}{\dot{\Lambda}^{s}}
\knowledgenewcommand{\taylor}{\cmdkl{\mathrm{a}}}

\newtheorem{theorem}{Theorem}

\newtheorem{proposition}{Proposition}[section]
\newtheorem{lemma}[proposition]{Lemma}

\newtheorem{corollary_alpha}[theorem]{Corollary}

\theoremstyle{definition}

\theoremstyle{remark}
\newtheorem{remark}[proposition]{Remark}

\numberwithin{equation}{section}

\newcounter{hyp}

\newenvironment{hyp}{\begin{equation} \stepcounter{hyp} \tag{H.\thehyp}}{\end{equation}}

\AtBeginDocument{
   \def\MR#1{}
}

\begin{document}
\title[Free-surface Euler equations with density]{Free-surface Euler equations with density variations, and shallow-water limit}
\date{\today}
\author{Théo Fradin}
\address{Univ. Bordeaux, CNRS, Bordeaux INP, IMB, UMR 5251, F-33400 Talence, France }
\email{theo.fradin@math.u-bordeaux.fr}
\subjclass[2020]{35Q35,35F61,35R35,76B03}
\keywords{Euler equations, free-surface, topography, vorticity, density variations}
\begin{abstract}
In this paper we study the well-posedness in Sobolev spaces of the incompressible Euler equations in an infinite strip delimited from below by a non-flat bottom and from above by a free-surface. We allow the presence of vorticity and density variations, and in these regards the present system is an extension of the well-studied water waves equations. When the bottom is flat and with no density variations (but when the flow is not necessarily irrotational), our study provides an alternative proof of the already known large-time well-posedness results for the water waves equations.
Our main contribution is that we allow for the presence of density variations, while also keeping track of the dependency in the shallow water parameter. This allows us to justify the convergence from the free-surface Euler equations towards the non-linear shallow water equations in this setting. Using an already established large time existence result for these latter equations, we also prove the existence of the free-surface Euler equations on a logarithmic time-scale, in a suitable regime.
\end{abstract}
\maketitle
\section{Introduction}
\subsection{Physical setting}
Let $\Omega_t$ be the domain of the fluid, delimited from below by an non-dimensional  non-flat bottom $-1 + \beta b$ where $b : \R^d \to \R$ is a smooth function such that $-1 + \beta b$ is negative and bounded from below, and $\beta \geq 0$. Here, $d \in \{1,2\}$ denotes the horizontal dimension. The domain of the fluid is also delimited from above by a free surface $\epsilon \eta_0$ of size $\epsilon$ not necessarily small, namely
\begin{equation}
\label{eqn:def:Omegat}\Omega_t := \{(x,z) \in \R^d \times \R, -1+\beta b(x) < z < \epsilon \eta_0 \}.
\end{equation}
We denote by $\epsilon\vec{U} := (\epsilon V,\epsilon w)^T$ the total velocity of the fluid, with $V$ its horizontal component and $w$ its vertical one, both of size $\epsilon$. We consider a fluid with a density given by $\rho_{\tot}(t,x,z) := \rhob + \epsilon \delta \rho(t,x,z)$ with $\rhob \in \R^+$ a constant reference density and $\delta \geq 0$. Although the size of the velocity and free-surface perturbations is $\epsilon$, we consider a density perturbation of size $\epsilon \delta $. Both $\epsilon$ and $\delta$ can be taken large in our main result (see however Remark~\ref{rk:ctes}). The incompressible Euler equations are completed with the mass conservation equation, and in non-dimensional form read
\begin{equation}
\label{eqn:euler_euleriennes}
\sys{ \partial_t V + \epsilon \vec{U} \cdot \nabla_{x,z} V + \frac{1}{\rhob + \epsilon \delta \rho}\nabla P + \frac{g\rhob}{\rhob + \epsilon \delta \rho}  \nabla \eta_0 &= 0, \\
		\mu \left( \partial_t w + \epsilon \vec{U} \cdot \nabla_{x,z} w \right) + \frac{1}{\rhob + \epsilon \delta \rho}\partial_z P + g \delta \frac{\rho}{\rhob + \epsilon \delta \rho}&= 0, \\
	\partial_t \rho + \epsilon \vec{U}\cdot \nabla_{x,z} \rho &=
		0,\\
		\nabla \cdot V + \partial_z w &= 0, } \qquad \text{ in } \Omega_t;
\end{equation}
here, $P$ denotes the perturbation of the hydrostatic pressure at equilibrium, given by $g \rhob(\epsilon \eta_0-z)$. Namely, if $P_{\tot}$ denotes the total pressure of the fluid, $P_{\tot} = g(z-\epsilon \eta_0) + \epsilon P$. See Subsection~\ref{subsection:pression} for the precise definition of $P$ from the unknowns $V,w,\rho$ and $\eta_0$. The parameter $\mu \in (0,1]$ is the shallow water parameter, and is defined as the square of the ratio between characteristic vertical and horizontal scales. In many applications, it is small with respect to $1$, therefore in this paper we first conduct a study on~\eqref{eqn:euler_euleriennes} that holds uniformly in $\mu$, so that we can then examine the limit $\mu \to 0$; see for instance~\cite{Vallis2017,Duchene2022a,Lannes20} for an account on numerous shallow water models, widely used to study oceanic flows.\\
The set of equations~\eqref{eqn:euler_euleriennes} is completed with boundary conditions
\begin{equation}
\label{eqn:euler_euleriennes:bc}
\sys{ P_{|z=\epsilon \eta_0} &= 0, \\
	\partial_t \eta_0 + \epsilon V_{|z=\epsilon \eta_0} \cdot \nabla \eta_0 - w_{|z=\epsilon \eta_0} &= 0, \\
	\vec{N_b} \cdot \vec{U}_{|z=-1+\beta b} &= 0,}\qquad \text{in } [0,T]\times \R^d,
\end{equation}
where $\vec{N_b} := ( - \beta \nabla^T b, 1)^T$ is the (upward) normal direction at the bottom. The first boundary condition in~\eqref{eqn:euler_euleriennes:bc} is the continuity of the pressure at the surface with the atmospheric pressure, assumed to be constant and renormalized at $0$. The second boundary condition is the kinematic equation on the free-surface and the last one is the impermeability of the bottom.\\
In view of studying the initial value problem associated with these equations, we add the following initial data, that must satisfy the incompressibility condition in~\eqref{eqn:euler_euleriennes} and the first and third boundary conditions in~\eqref{eqn:euler_euleriennes:bc}:
\begin{equation}
\label{eqn:euler_euleriennes:ci}
\sys{\vec{U}_{|t=0} &= \vec{U}_{\ini} \qquad & \text{in } \Omega_t,\\
\rho_{|t=0} &= \rho_{\ini} \qquad & \text{in } \Omega_t,\\
(\eta_0)_{|t=0} &= (\eta_0)_{\ini} \qquad & \text{in } \R^d. }
\end{equation}

The aim of this paper is to study the initial boundary value problem on the system~\eqref{eqn:euler_euleriennes} with boundary conditions~\eqref{eqn:euler_euleriennes:bc} and initial data~\eqref{eqn:euler_euleriennes:ci} and obtain results that are uniform in $\mu$ (see the statement of Theorem~\ref{thm:euler_phi} in Subsection~\ref{subsection:first_result} below). This result then allows us to justify the persistence of the non-linear shallow water equation~\eqref{eqn:nlsw} in this setting, see the statement of Theorem~\ref{thm:convergence:svt} in Subsection~\ref{subsection:second_result}.\\

Let us now define the non-dimensionalized vorticity
$$ \vort := \begin{pmatrix} \frac{1}{\sqrt{\mu}} \partial_z V^{\perp} - \sqrt{\mu} \nabla^{\perp} w \\ \nabla^{\perp} \cdot V\end{pmatrix}.$$
Taking the $(d+1)$-dimensional curl of the equations on $V$ and $w$ in~\eqref{eqn:euler_euleriennes} and dividing the first component by $\sqrt{\mu}$, we get 
\begin{equation}
\label{eqn:euler_euleriennes:vort}
\partial_t \vort + \epsilon \vec{U} \cdot \nabla_{x,z} \vort + \epsilon\vort^{[x]} \cdot \nabla \begin{pmatrix} V \\ \sqrt{\mu} w \end{pmatrix} + \epsilon \vort^{[z]} \partial_z \begin{pmatrix} \frac{1}{\sqrt{\mu}} V \\ w \end{pmatrix} = \frac{\delta}{\sqrt{\mu}} \vec{F} \qquad \text{in } \Omega_t.
\end{equation}
Here, $\vort^{[x]}$ and $\vort^{[z]}$ denote the horizontal and vertical components of the vorticity $\vort$. When $d=1$,~\eqref{eqn:euler_euleriennes:vort} classically boils down to a scalar transport equation, although we do not use this particular structure in this study so that our result holds both when $d=1$ and $d=2$. The source term $\vec{F}$ stems from the density variations, namely\\
\begin{equation}
\label{eqn:def:vort:source}
\vec{F}:= \begin{pmatrix} \nabla^{\perp}\left( \frac{\rho}{\rhob + \epsilon \delta \rho} \right) + \epsilon \partial_z \left( \frac{\rho}{\rhob + \epsilon \delta \rho} \right) \nabla^{\perp} \left( P - g \eta_0 \right) -  \epsilon \nabla_{x}^{\perp} \left( \frac{\rho}{\rhob + \epsilon \rho } \right) \partial_z P\\
\epsilon \sqrt{\mu} \nabla^{\perp} \left( \frac{\rho}{\rhob + \epsilon \delta \rho} \right) \cdot \nabla (P - g \eta_0)
\end{pmatrix}.
\end{equation}
\begin{figure}
\includegraphics[scale=1]{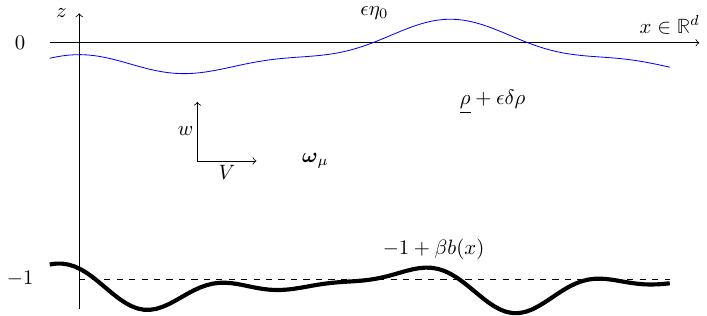}
\caption{Setting in Eulerian coordinates.}\label{fig:eul}
\end{figure}Note that if there are no density variations ($\delta = 0$) and if the initial vorticity is zero,~\eqref{eqn:euler_euleriennes:vort} yields that the vorticity remains zero. In this so-called irrotational case, see the pioneer works~\cite{Wu1997},~\cite{Lannes05},~\cite{AlazardBurqZuily2014}. In this setting,~\eqref{eqn:euler_euleriennes} can be reduced to a system of equations on the evolution of the free surface $\eta_0$ and a trace of the velocity potential at the surface, the Zakharov/Craig-Sulem formulation of the water waves equations, see~\cite{Lannes2013} for a comprehensive study of the derivation and well-posedness of the water waves equations as well as several associated asymptotic regimes.\\
In the presence of vorticity but without density variations, see the works~\cite{Kukavica2018} and~\cite{WangZhang2015}, where the authors study the free-surface Euler equations with minimal regularity assumptions on the initial data, although without tracking the dependency in $\mu$. The author in~\cite{Depoyferre19} derives energy estimates for the Euler equations~\eqref{eqn:euler_euleriennes} with a non-flat bottom. He even treats the case of emerging bottom, that is enabling $1 - \beta b(x) + \epsilon \eta_0 < 0$ for some $x \in \R^d$, improving the results in~\cite{ShatahZeng08a,ShatahZeng08,ShatahZeng11}. See also~\cite{ZhangZhang2008}. However, these works do not keep track of the dependency of their estimates with respect to $\mu$. \\
In~\cite{Castro2014a}, the authors derive a generalization of the Zakharov/Craig-Sulem formulation in the presence of vorticity variations, alongside an equation (namely,~\eqref{eqn:euler_euleriennes:vort} with $\delta = 0$) on the evolution of the vorticity, in the flat-bottom case. This formulation is extended to the case of a non-flat bottom and with Coriolis forcing in~\cite{Melinand2017}, where the energy estimates on the subsequent system are performed. These two results keep track of the dependency in $\mu$, and even perform the shallow water limit $\mu \to 0$. See also the survey~\cite{Lannes20}. In~\cite{MasmoudiRoussetSun2023}, the incompressible limit for the compressible free-surface Navier-Stokes equations is considered. The authors in~\cite{MasmoudiRousset2012} study the free-surface Navier-Stokes equations (that is, adding viscosity in~\eqref{eqn:euler_euleriennes}), with no bottom boundary. Although the presence of viscosity changes several parts of the analysis, the limit of vanishing viscosity is studied. 

In the presence of density variations, several studies have been conducted in the whole space~\cite{Danchin2010} or in more general domains~\cite{Shigeharu1999}. In the case of a strip with a density stratification, with or without a free-surface, see~\cite{Desjardins2019},~\cite{Duchene2022} and~\cite{Fradin2024} for the well-posedness theory. To the best of our knowledge, the present study is the first on the free-surface case with a non-flat bottom and density variations, keeping track of the dependency of the shallow water parameter $\mu$. To the best of our knowledge, it is also the first result on the persistence of the non-linear shallow water equations as an approximate model for the free-surface Euler equations with weak density variations.\\

Our motivation for performing the analysis directly on the free-surface Euler equations rather than on a Zakharov/Craig-Sulem type formulation as in ~\cite{Castro2014a} is three-fold. First, and most importantly, extending the Zakharov/Craig-Sulem formulation from~\cite{Castro2014a} in the presence of density variations is not straightforward. Then, we believe that some parts of the analysis are more standard, namely the scheme of existence (see Subsection~\ref{subsection:scheme}) as well as the quasilinear structure (see Subsections~\ref{subsection:quasilin} and~\ref{subsection:quasilin:gen}). Eventually, this allows us to compare the analysis with the one performed on similar systems: namely the Navier-Stokes equations (see for instance~\cite{MasmoudiRousset2012}) or the stratified Euler equations (see~\cite{Desjardins2019},~\cite{Duchene2022} and~\cite{Fradin2024}). 
\subsection{First result: well-posedness of~\eqref{eqn:euler_euleriennes}}
\label{subsection:first_result}
The main result of this paper is a local well-posedness result on the system~\eqref{eqn:euler_euleriennes}, stated in Theorem~\ref{thm:euler_phi}. We state here a shorter version, for the sake of clarity. Let us first state the main assumptions sufficient for the following theorem to hold. Let $d \in \{1,2\}$ be the horizontal dimension, $s_0\in \N$ with $s_0>\frac{d+1}{2}$, and $s \in \N$ with $s \geq s_0+2$ be the regularity of the solution considered. Let $M> 0$, we assume
\begin{hyp}
\label{hyp:b}
\Vert b \Vert_{H^{s+\frac32}} \leq M.
\end{hyp}Let $M_{\ini} >0$ and $(V_{\ini},w_{\ini}, \rho_{\ini}, (\eta_0)_{\ini}) \in (H^{s})^{d+2}(\Omega_0) \times H^s(\R^d)$ satisfying the following assumptions. First, we make the non-cavitation assumption, namely
\begin{hyp}
\label{hyp:density}
\rhob + \epsilon \delta \rho_{\ini} \geq c_* > 0.
\end{hyp}We also assume that the depth of the fluid is bounded from below
\begin{hyp}
\label{hyp:non_cavitation:euleriennes}
1 - \beta b + \epsilon (\eta_0)_{\ini} \geq c_*,
\end{hyp}for some $c_* > 0$, as well as the positivity of the Rayleigh-Taylor coefficient
\begin{hyp}
\label{hyp:taylor:intro}
\taylor \geq c_* > 0,
\end{hyp}defined as
\begin{equation}
\label{eqn:taylor:def:intro}
\taylor := g\rhob - \epsilon (\partial_z P)_{|z=\epsilon \eta_0}.
\end{equation}Here, $P$ is defined from the initial data through~\eqref{eqn:pression}.
These assumptions are classical in this setting, see for instance~\cite{Castro2014a}. Note however the results in~\cite{Depoyferre19} which hold without~\eqref{hyp:non_cavitation:euleriennes}, and a discussion about~\eqref{hyp:taylor:intro} for the Navier-Stokes equation in the vanishing viscosity limit in~\cite{MasmoudiRousset2012}.\\

We furthermore assume
\begin{hyp}
\label{hyp:Mini:euleriennes}
\Vert(V_{\ini},\sqrt{\mu} w_{\ini},\sqrt{\mu}\rho_{\ini})\Vert_{H^{s}(\Omega_0)} + \vert (\eta_0)_{\ini}\vert_{H^{s}(\R^d)} \leq M_{\ini};
\end{hyp}the Sobolev spaces $H^s(\R^d)$ as well as $H^s(\Omega_t)$ are defined in~\eqref{eqn:def:sobolev} and~\eqref{eqn:def:Hsomegat}. We also assume that the initial (non-dimensional) vorticity $\vort_{\ini}$ defined from $(V_{\ini},w_{\ini})$ through~\eqref{eqn:vort:def} satisfies
\begin{hyp}
\label{hyp:vort:ini}
\Vert \vort_{\ini} \Vert_{H^{s}(\Omega_0)} \leq M_{\ini}.
\end{hyp}
Recall the notation $a \vee b := \max(a,b)$. We can now make a rough statement of Theorem~\ref{thm:euler_phi}.
\theoremstyle{plain}
\newtheorem*{thm_temp}{Theorem~\ref{thm:euler_phi} {\it (Rough statement)}}
\begin{thm_temp}
\label{cor:ew:euleriennes}
Let $\mu\in (0,1],$ $\epsilon,\beta,\delta \in [0,1]$. Let $(V_{\ini},w_{\ini}, \rho_{\ini}(\eta_0)_{\ini}) \in (H^{s})^{d+2}(\Omega_0) \times H^s(\R^d)$ satisfying~\eqref{hyp:b},~\eqref{hyp:density},~\eqref{hyp:non_cavitation:euleriennes},~\eqref{hyp:taylor:intro},~\eqref{hyp:Mini:euleriennes},~\eqref{hyp:vort:ini}, the first and third boundary conditions in~\eqref{eqn:euler_euleriennes:bc} and the incompressibility condition in~\eqref{eqn:euler_euleriennes}. Then there exists $T> 0$ independent of $\beta$,$\epsilon$,$\mu$ and $\delta$ such that there exists a unique solution $(V,w,\rho,\eta_0) \in C^{0}([0,\frac{T}{\epsilon \vee \beta\vee\frac{\delta}{\mu}}],H^{s}(\Omega_t)^{d+2} \times H^s(\R^d))$ to~\eqref{eqn:euler_euleriennes} satisfying the boundary conditions~\eqref{eqn:euler_euleriennes:bc} and the initial condition~\eqref{eqn:euler_euleriennes:ci}.
\end{thm_temp}
\begin{remark}
\label{rk:ctes}
The constant $T$ in the previous theorem is independent of $\beta,\epsilon,\mu,\delta$. However, it may depend on the constants defined in the assumptions, namely $M,c_*,M_{\ini}$, as well as $d,s_0$ and $s$.\\
Note that the time of existence in the previous theorem can be made independent of $\mu$ by assuming small density variations, i.e. $\delta \leq \mu$.
\end{remark}
\begin{remark}
\label{rk:reg}
The minimal Sobolev regularity for Theorem~\ref{thm:euler_phi} to hold is $s_0+2$. In the case of the Euler equations without a free boundary, the minimal regularity obtained from the energy method is $s_0+1$. In the present study, the non-linear interactions between the free surface and the other unknowns yield the minimal regularity $s_0+2$. See however the works~\cite{ShatahZeng08a},~\cite{ShatahZeng08},~\cite{ShatahZeng11},~\cite{WangZhang2015},~\cite{Kukavica2018}, ~\cite{Depoyferre19} that manage to lower the minimal regularity, though with a time of existence that is not independent of the shallow-water parameter $\mu$.
\end{remark}
\begin{remark}
\label{rk:columnar}
In the irrotational case $\vort = 0$, in the shallow-water limit $\mu = 0$, one gets the columnar motion property $\partial_z V=0$. Then, the non-linear shallow water equations are an approximate model of the water waves equations, with precision $O(\mu)$, see~\cite{Lannes2013} and references within.\\
The setting in which the initial non-dimensional vorticity is of size $1$ (see~\eqref{hyp:vort:ini}) is the setting studied in~\cite{Castro2014a} and~\cite{Melinand2017}. It roughly states that $\partial_z V \sim \sqrt{\mu}$, and as such, it is a relaxation on the columnar motion assumption $\partial_z V = 0$ in the shallow water setting (i.e. $\mu = 0$). We use the vorticity equation~\eqref{eqn:euler_euleriennes:vort} in this study to show that the assumption~\eqref{hyp:vort:ini} is propagated for for positive times, at least in the regime $\delta \leq \mu$ of small density variations, and on the time-scale considered in Theorem~\ref{thm:euler_phi}. This point is strongly used to control some terms in the analysis: namely, the vertical advection term in~\eqref{eqn:euler_euleriennes} is
\begin{equation}
\label{eqn:vertical_advection}
 \epsilon \sqrt{\mu} w \frac{1}{\sqrt{\mu}} \partial_z V,
 \end{equation}
where we have a uniform control on $\sqrt{\mu} w$, but not on $w$ (see~\eqref{eqn:energy:def}). The singularity $1/\sqrt{\mu}$ that appears is dealt with in the following analysis by using $\partial_z V \sim \sqrt{\mu}$ from~\eqref{hyp:vort:ini}, see more precisely the proof of Proposition~\ref{prop:restes}.\\
In the setting $\vort_{\ini} = O(1)$, the non-linear shallow water equations are an approximate model of the free-surface Euler equations with precision $O(\sqrt{\mu})$. This slower rate of convergence was already observed in~\cite{Castro2014a} and is due to the presence of vorticity. Theorem~\ref{thm:convergence:svt} shows that it extends to the case of small density variations.\\
In the context of stratified fluids (that is, when the reference density $\rhob$ is not assumed to be constant but strictly decreasing with respect to $z$), the equation on the vorticity does not appear to help propagate the condition $\partial_z V \sim \sqrt{\mu}$ from~\eqref{hyp:vort:ini}. This is because the density variations at equilibrium cannot be taken small with respect to $\mu$, as opposed to the present case with $\delta \leq \mu$. Because the condition $\partial_z V \sim \sqrt{\mu}$ cannot be propagated in the stratified case,~\cite{Desjardins2019} and~\cite{Fradin2024} do not assume this condition to hold, and the singularity in~\eqref{eqn:vertical_advection} is dealt with the additional assumption $\epsilon \leq \sqrt{\mu}$. Although the present analysis and the one for the stratified case share a lot of similarities, this is a major discrepancy.
\end{remark}
\begin{remark}
Solving the water waves equations in the Zakharov/Craig-Sulem formulation is an initial value problem on the horizontal space $\R^d$ (see~\cite{Lannes2013} and references within). On the other hand, both the present study on~\eqref{eqn:euler_euleriennes} or~\cite{Castro2014a} on the free-surface euler equations in the Zakharov/Craig-Sulem formulation treat initial boundary value problems, in the free-surface domain $\Omega_t$, as the evolution of the vorticity does not reduce to an equation defined on the free surface. The approximation scheme to construct a solution is then more involved, see for instance~\cite{BenzoniGavage2006} for an global picture on hyperbolic initial boundary value problems. The strategy used in~\cite{Castro2014a} relies on a suitable vertical regularization, despite the boundaries. The proof of Theorem~\ref{thm:euler_phi} is slightly different, as we use semi-Lagrangian coordinates to write a reformulation of the Euler equations without vertical advection, see~\eqref{eqn:euler:slag}. Both strategies are adapted to the case where the flow is tangent to the boundaries.  
\end{remark}
\begin{remark}
In the absence of topographic variations ($\beta = 0$) and no density variations ($\delta = 0$), the time of existence in Theorem~\ref{thm:euler_phi} becomes $\frac{T}{\epsilon}$, which is sometimes called a large time, because it grows as the size $\epsilon$ of the deviation from the equilibrium decreases. With large topographic variations, the same time of existence $\frac{T}{\epsilon \vee \beta}$ as in Theorem~\ref{thm:euler_phi} holds for the water waves equations, see for instance~\cite{Lannes2013}. Whether a large-time existence result (i.e. on $[0,\frac{T}{\epsilon}]$) holds in the presence of large topographic variations is an open question, both for the water waves equations or for the present system~\eqref{eqn:euler_euleriennes}, see also Remark~\ref{rk:tps_long:vort}. However, in the case of the water waves equations with surface tension and large topographic variations, the large time existence holds, see~\cite{Mesognon17}. Such a large time existence result also holds for the non-linear shallow water equations with large topographic variations, see~\cite{BreschMetivier10}.
\end{remark}
\subsection{Second result: shallow water limit}
\label{subsection:second_result}
As the time of existence in Theorem~\ref{thm:euler_phi} is bounded from below uniformly with respect to $\mu$ in the case of small density variations ($\delta \leq \mu$), it defines a family of solutions $(V,w,\rho,\eta_0)_{\mu>0}$ in $C^{0}([0,\frac{T}{\epsilon \vee \beta\vee \frac{\delta}{\mu}}[,H^{s}(\Omega_t)^{d+2} \times H^s(\R^d))$, parameterized by $\mu > 0$. A natural question is then to study its limit as $\mu \to 0$. The candidate for the limit is the solution of the  non-linear shallow water (or Saint-Venant) equations, that read
\begin{equation}
\label{eqn:nlsw}
\sys{\partial_t \eta_{\sw} + \nabla \cdot \left( (1 - \beta b + \epsilon \eta_{\sw}) V_{\sw} \right) & = 0,\\
	\partial_t V_{\sw} + \epsilon V_{\sw} \cdot \nabla V_{\sw} + g\nabla \eta_{\sw} &=0, \\
	} \qquad \text{ in } \R^d;
\end{equation} 
here, $V_{\sw}$ actually denotes the vertically averaged horizontal velocity, and $\eta_{\sw}$ is the function whose graph is the free-surface. In particular,~\eqref{eqn:nlsw} does not take into account density variations. See~\cite{Lannes2013} and~\cite{Duchene2022a} and references within for an account on this system, and~\cite{Lannes20} where the vorticity is considered.  We now make a rough statement of the convergence result, as the precise statement of Theorem~\ref{thm:convergence:svt} requires the introduction of some notations.
\theoremstyle{plain}
\newtheorem*{thm_temp2}{Theorem~\ref{thm:convergence:svt}{\it (Rough statement)}}
\begin{thm_temp2}
\label{cor:cvce:rough}
If the solution $(V,w,\rho,\eta_0)$ of~\eqref{eqn:euler_euleriennes} for $\mu > 0$ given by Theorem~\ref{thm:euler_phi} and the solution $(V_{\sw},\eta_{\sw})$ are close initially (more precisely, the norm of their difference in some functional space is small with respect to $\mu$; see~\eqref{hyp:E0}), then they remain close on the time interval $[0,\frac{T}{\epsilon \vee \beta\vee\frac{\delta}{\mu}}]$ given by Theorem~\ref{thm:euler_phi}.
\end{thm_temp2}
To the best of our knowledge, this is the first justification of the non-linear shallow water equations, in the presence of topography and density variations, from the free-surface Euler equations.
\begin{remark}
The convergence from the water waves equation towards the non-linear shallow water equations was already established in the irrotational case with topography (see~\cite{Lannes2013} and references within), in the flat-bottom case with vorticity (see~\cite{Castro2014a}) and in the case with vorticity and topographic variations (as well as Coriolis forcing) in~\cite{Melinand2017}.
\end{remark}
\begin{remark}
The present result proves that the non-linear shallow water equations~\eqref{eqn:nlsw} are an approximate model for the free-surface Euler equations~\eqref{eqn:euler_euleriennes} with precision $O(\sqrt{\mu})$; see Remark~\ref{rk:columnar} for a note on the precision. It takes place in the regime where the initial non-dimensional vorticity is of size $O(1)$, which can be seen as a relaxation of the columnar motion property ($\partial_z V \neq 0$). It would be of interest to prove a well-posedness result in the same fashion as Theorem~\ref{thm:euler_phi} outside of the case $\partial_z V \sim \sqrt{\mu}$. Then the solution of~\eqref{eqn:euler_euleriennes} is expected to be described at first order in $\mu$ by the solution of the hydrostatic Euler equations (see for instance~\cite{Masmoudi2012} for the flat-surface, flat-bottom case in dimension 2, and~\cite{Lannes20} for a different viewpoint).
\end{remark}
\begin{remark}
\label{rk:tps_long}
With large topographic variations ($\beta = 1$) and small density variations ($\delta \leq \mu$), Theorem~\ref{thm:euler_phi} only proves the existence of the solution to~\eqref{eqn:euler_euleriennes} on $[0,T]$, and on this time interval, it remains close to the solution of the non-linear shallow water system~\eqref{eqn:nlsw}. However, it was shown in~\cite{BreschMetivier10} that, despite large topographic variations, the solution of~\eqref{eqn:nlsw} exists on the time interval $[0,\frac{T}{\epsilon}]$. This might suggest that it is possible to prove the large time existence (i.e. on a time interval $[0,\frac{T}{\epsilon}]$) of the solution of the water waves equations and of~\eqref{eqn:euler_euleriennes}. Nonetheless, these questions remain, to the best of our knowledge, open.
\end{remark}
Set $\beta = 1$ and $\delta = 0$ for the sake of clarity. As stated in Remark~\ref{rk:tps_long}, the existence time for the solution of~\eqref{eqn:euler_euleriennes} is $T$, whereas it is $\frac{T}{\epsilon}$ for~\eqref{eqn:nlsw} according to~\cite{BreschMetivier10}, and thus is much larger in the shallow water case. We use the latter result in subsection~\ref{subsection:cheating} to improve the former, in the regime $\mu \leq \epsilon$ and with well-prepared data, to get an existence time of order $\log(\frac{1}{\epsilon})$; we give a rough statement of this result, that is a corollary of Theorem~\ref{thm:convergence:svt}. See Corollary~\ref{cor:tps_long} for the precise statement.
\newtheorem*{thm_temp3}{Corollary~\ref{cor:tps_long}{\it (Rough statement)}}
\begin{thm_temp3}
\label{thm:intro:tps_long}
Under the assumptions of Theorem~\ref{thm:euler_phi} and moreover if the initial data~\eqref{eqn:euler_euleriennes:ci} is close (with respect to $\mu$) to an admissible initial data for the non-linear shallow water equations~\eqref{eqn:nlsw} and in the regime $\mu \leq \epsilon$, then the solution $(V,w,\rho,\eta_0)$ of~\eqref{eqn:euler_euleriennes} given by Theorem~\ref{thm:euler_phi} can be extended on the time interval $[0,T \log(\frac{1}{\epsilon})]$, where $T > 0$ is a constant that is independent of $\epsilon,\mu$.
\end{thm_temp3}
The time of existence in Corollary~\ref{cor:tps_long} is much smaller than the time $\frac{T}{\epsilon}$ from~\cite{BreschMetivier10} for the non-linear shallow water equations, although it can also be made arbitrarily large by taking $\epsilon$ (and thus, $\mu$) small enough. Such a logarithmic time appears in semi-classical analysis and where it is called the Ehrenfest time~\cite{Shepelyansky2020,CarlesFermanian2011}. See also~\cite{LannesRauch2000}, in the context of non-linear geometric optics.
\subsection{Plan of the paper}
In Section~\ref{section:change_variables}, we first define the change of variables that we use throughout the paper to fix the fluid domain to the strip. We then state some results on Alinhac's good unknown, used to handle the non-linearity coming from the change of variables.\\
In Section~\ref{section:estimates}, we first show how to define the pressure from the other unknowns, and derive the estimates on the pressure in Proposition~\ref{prop:pression}. We then exhibit the quasilinear structure of the Euler equations, and use them to derive high-order energy estimates in Proposition~\ref{prop:energy}. We conclude this section with the construction of the solution of the Euler equations, using semi-Lagrangian coordinates, in the proof of Theorem~\ref{thm:euler_phi}.\\
In Section~\ref{section:cvce}, we first study the convergence from the solution of the Euler equations towards the solution of the non-linear shallow water equations in Theorem~\ref{thm:convergence:svt}. We then show how to use the large-time existence result from~\cite{BreschMetivier10} to improve slightly the existence result of Theorem~\ref{thm:euler_phi}, under some assumptions, in Corollary~\ref{cor:tps_long}.\\
Appendix ~\ref{appendix:A} is concerned with technical tools such as classical estimates in Sobolev spaces. 
\subsection{Notations}
\label{subsection:notations}
\begin{itemize}
\item In this paper (see for instance Theorems~\ref{thm:euler_phi},~\ref{thm:convergence:svt} or Propositions~\ref{prop:pression},~\ref{prop:energy}) the constant $C$ or the time $T$ denote constants which may depend implicitly on the other constants $s_0,s,d$ (see for instance the introduction of subsection~\ref{subsection:first_result}), $c_*,h_*$ (see Assumptions~\eqref{hyp:density},\eqref{hyp:non_cavitation:euleriennes},~\eqref{hyp:taylor:intro}). When they also depend on another constant, it will be explicitly stated (see for instance Proposition~\ref{prop:pression}, "C depends only on M", the use of the word "only" thus being a slight abuse of notation). However, it is crucial in our analysis that they do not depend on the parameters $\beta,\epsilon\mu,\delta$, nor on the initial data, other than through the aforementioned constants.
\item Bold letters denote vectors of dimension $d+1$, such as $\vec{g}$. 
\item If $(a,b) \in \R^2$, then $a\vee b$ denotes the maximum between $a$ and $b$, and $a \wedge b$ denotes the minimum.
\item For $s \in \R$, $|D|$ is the Fourier multiplier of symbol $|\xi|$, $\Lambda^s$ is the Fourier multiplier of symbol $(1 + |\xi|^2)^{\frac{s}{2}}$, and $\dot{\Lambda}^{s} := |D| \Lambda^{s-1}$.
\item If $f,g$ are smooth enough functions and $\sigma(D)$ is a Fourier multiplier (such as $\Lambda^{s}$ or $\dot{\Lambda}^{s}$ for $s \in \R)$, we define the commutator $[\sigma(D),f]g := \sigma(D)(fg)-f\sigma(D) g$ and the symmetric commutator $[\sigma(D);f,g] := \sigma(D)(fg) - f \sigma(D)(g) - \sigma(D)g$. The commutator coming taking into account Alinhac's good unknown $\commal{\sigma(D)}{f}{g}$ used in~\eqref{apdx:alinhac:comm1} is defined as
\begin{equation}
\label{apdx:alinhac:comm:def}
\commal{\dot{\Lambda}^{s}}{\gradphi[t,x,r]}{f} :=   \dot{\Lambda}^{s} (-\beta b + \epsilon \eta_0) \partial_r^{\varphi}(\gradphi f) + (\partial \varphi)^{-T} [\dot{\Lambda}^{s} ; (\partial \varphi)^T, \gradphi[t,x,r] f].
\end{equation}
\item We use the notation $\coo{v_1\\v_2}^{\perp} := \coo{-v_2 \\ v_1}$.
\item For $p \in [1,\infty]$, we use the standard $L^p-$based Sobolev spaces
\begin{equation}
\label{eqn:def:sobolev}
W^{s,p}(\R^d) := \{ \phi  \in L^p(\R^d), |\phi |_{W^{s,p}} < + \infty \}, 
\end{equation}
with the norm
$$|\phi |_{W^{s,p}} := |\Lambda^s \phi  |_{L^p}.$$
We denote with simple bars $|\cdot |$ norms in dimensions $1$ or $d$. \\
We write 
$$H^s := W^{s,2}.$$ \\
We use the functional spaces, for $s \in \R_+$, $0 \leq k \leq s$,$p \in [1,\infty]$:
$$ W^{s,k,p}(S) := \{\phi  \in L^p(S), \Vert \phi  \Vert_{W^{s,k,p}} < \infty \}, $$
where we define the norm 
$$\Vert \phi  \Vert_{W^{s,k,p}} :=  \sum\limits_{l = 0}^{k} \Vert \Lambda^{s-l} \partial_{r}^l \phi  \Vert_{L^p}, $$
and write $H^{s,k} := W^{s,k,2}$. 
In the case $s=k$, we write $H^s(S_r) := H^{s,k}(S_r)$. In the case $k=0$, this yields 
$$ \Vert f \Vert_{H^{s,0}(S)} := \Vert \Lambda^s f\Vert_{L^2(S)}.$$
We only use $L^2$ and $L^{\infty}$ based Sobolev spaces.
\item We say that $(V,w,\eta_0) \in C^{0}([0,T[,H^{s}(\Omega_t)^{d+1} \times H^s(\R^d))$ if $\eta_0 \in C^0([0,T[,H^s(\R^d))$ and, for $\varphi$ the diffeomorphism defined from $\eta_0$ through~\eqref{eqn:phi:def} and~\eqref{eqn:phi:def:bary}:
\begin{equation}
\label{eqn:def:Hsomegat}
(V,w)\circ \varphi \in C^0([0,T[,H^{s}(S)^{d+1}).
\end{equation}
\end{itemize}

\section{Straightening the fluid domain}
\label{section:change_variables}
\subsection{The change of variables}
In order to study the well-posedness of~\eqref{eqn:euler_euleriennes},  we use the following change of variables, whose inverse maps the time-dependent domain $\Omega_t$ to the strip $S := \R^d \times [-1,0] $:
\begin{equation}
\label{eqn:phi:def}
 \begin{aligned} \varphi : S &\to \Omega_t \\  (x,r) &\mapsto (x,\etab(x,r) + \epsilon \eta(t,x,r)) \end{aligned}
\end{equation}
where $\etab$ and $\etab + \epsilon \eta$ are strictly increasing with respect to $r$. For most of this study, we will use the following, standard change of variables:
\begin{equation}
\label{eqn:phi:def:bary}
\etab  := r (1 - \beta b) \qquad   \eta :=  (1+r) \eta_0.
\end{equation}
It is only in Subsection~\ref{subsection:scheme} that we will use other coordinates (i.e. other definitions for $\etab$ and $\eta$), that will be defined subsequently. We also denote by $\hb, h$ the derivatives with respect to $r$ of $\etab, \eta$ respectively. Both may depend on $r$ a priori, although in the case of the change of variables~\eqref{eqn:phi:def:bary}, we have
\begin{equation}
\label{eqn:h:def}
\hb + \epsilon h = 1 - \beta b + \epsilon \eta_0.
\end{equation}
The assumption~\eqref{hyp:non_cavitation:euleriennes} of non-vanishing (and bounded from above) fluid height writes, in these coordinates, that there exists $h_*,h^* > 0$ such that
\begin{hyp}
\label{hyp:H}
h_* \leq \hb + \epsilon h \leq h^*.
\end{hyp}The regularity assumptions~\eqref{hyp:b} on $b$ is replaced
\begin{hyp}
\label{hyp:Hb:reg}
\Vert \hb -1 \Vert_{H^{s+2}(S)} \leq \beta \Mb, 
\end{hyp}for $\Mb$ some constant, and $s \in \N$. In the case of the diffeomorphism given by~\eqref{eqn:phi:def:bary},~\eqref{hyp:b} and~\eqref{hyp:Hb:reg} are equivalent. The regularity assumption $\eta_0 \in H^s(\R^d)$, given the form~\eqref{eqn:phi:def:bary} of $\eta$ and the definition $h := \partial_r \eta$, yields 
\begin{hyp}
\label{hyp:H:reg}
\Vert h \Vert_{H^{s-1}(S)} \leq M,
\end{hyp}for $M > 0$ some constant. We define
\begin{equation}
\label{eqn:gradphi:def}
\begin{aligned} 
\gradphi := \nabla - \frac{\nabla (\etab + \epsilon \eta)}{\hb + \epsilon h} \partial_r, \\
\partialphi{t} := \partial_t - \epsilon \frac{\partial_t \eta}{\hb + \epsilon h} \partial_r, \\
\partialphi{r} := \frac{1}{\hb + \epsilon h} \partial_r.
\end{aligned} 
\end{equation}
With these definitions, for a quantity $f$ defined in $\Omega_t$, we have
$$ (\nabla f)\circ \varphi = \gradphi (f \circ \varphi),$$
resp. $\partialphi{t}, \partialphi{r}$. We can now reformulate the Euler equations~\eqref{eqn:euler_euleriennes} in these coordinates. Still denoting $V$ for $V \circ\varphi$ (resp. $\vec{U}$, $w$, $\rho$, $P$) by abuse of notations, we get 
\begin{equation}
\label{eqn:euler_phi}
\sys{ \partialphi{t} V + \epsilon \vec{U} \cdot \gradphi[x,r] V + \frac{1}{\rhob + \epsilon \delta \rho}\gradphi P + \frac{g\rhob}{\rhob + \epsilon \delta \rho}  \nabla \eta_0 &= 0,\\
		\mu \left( \partialphi{t} w + \epsilon \vec{U} \cdot \gradphi[x,r] w \right) + \frac{1}{\rhob + \epsilon \delta \rho}\partialphi{r} P + \delta \frac{\rho}{\rhob + \epsilon \delta \rho} &= 0,\\
		\partialphi{t} \rho + \vec{U} \cdot \gradphi[x,r] \rho &= 0, \\
		\gradphi \cdot V + \partialphi{r} w &= 0,}\qquad \text{in } [0,T]\times S.
\end{equation}
As before, these equations are completed with boundary conditions
\begin{equation}
\label{eqn:euler_phi:bc}
\sys{ P_{|r=0} &= 0, \\
	\partial_t \eta_0 + \epsilon V_{|r=0} \cdot \nabla \eta_0 - w_{|r=0} &= 0,\\
	\vec{N_b} \cdot \vec{U}_{|r=-1} &= 0,} \qquad \text{in } [0,T]\times\R^d.
\end{equation}
We also need the initial conditions
\begin{equation}
\label{eqn:euler_phi:ci}
\sys{\vec{U}_{|t=0} &= \vec{U}_{\ini} & \qquad \text{in } S,\\
\rho_{|t=0} &= \rho_{\ini} & \qquad \text{in } S,\\
(\eta_0)_{|t=0} &= (\eta_0)_{\ini} & \qquad \text{on } \R^d. }
\end{equation}
We define the non-dimensional vorticity
\begin{equation}
\label{eqn:vort:def}
\intro*\vort = \begin{pmatrix} \vort^{[x]} \\ \vort^{[r]}\end{pmatrix} := \begin{pmatrix}
\frac{1}{\sqrt{\mu}} \partialphi{r} V^{\perp} - \sqrt{\mu} \gradphi^{\perp} w \\ \gradphi^{\perp} \cdot V 
\end{pmatrix}. 
\end{equation} 
The equation on the vorticity in these coordinates now reads
\begin{equation}
\label{eqn:euler_phi:vort}
\partialphi{t} \vort + \epsilon \vec{U} \cdot \gradphi[x,r] \vort + \epsilon \vort^{[x]} \cdot \gradphi \begin{pmatrix} V \\ \sqrt{\mu} w \end{pmatrix} + \epsilon \vort^{[r]} \partialphi{r} \begin{pmatrix} \frac{1}{\sqrt{\mu}} V \\ w \end{pmatrix} = \frac{\delta}{\sqrt{\mu}}\vec{F} \qquad \text{in } [0,T]\times S.
\end{equation}
Here, $\vec{F}$ is given by~\eqref{eqn:def:vort:source}, where $\nabla_{x,z}$ is replaced by $\gradphi[x,z]$.
\subsection{Alinhac's good unknown}
The system~\eqref{eqn:euler_phi} mainly differs from~\eqref{eqn:euler_euleriennes} in Eulerian coordinates by the operators $\gradphi[x,r], \partialphi{t}$, which contain coefficients that depend on the diffeomorphism $\varphi$ (and hence the free surface $\eta_0$ in the case of the diffeomorphism given by~\eqref{eqn:phi:def:bary}), arising from the chain rule when differentiating a composition of the unknowns with $\varphi$. We briefly explain in this subsection how the use of Alinhac's good unknowns introduced in~\cite{Alinhac1989} allows us to treat these operators. These results are proved for instance in~\cite[Appendix B]{Fradin2024}. Recall the definitions of the Sobolev spaces $H^s$ and their norms as well as $\dot{\Lambda}^s$ in Subsection~\ref{subsection:notations}.\\

First, note the following naive estimate, which involves a loss of derivatives on $\eta_0$. It is a direct consequence of Lemma~\ref{apdx:commutator}.
\begin{lemma}
Let $s_0 \in \N$ with $s_0 > \frac{d+1}{2}$, $s \in \N$ with $s \geq s_0 + 1$. Assume that $\eta \in H^{s+1}(S)$ and~\eqref{hyp:H}. Then there exists $C>0$ such that, for $f \in H^{s}$:
\begin{equation}
 \label{apdx:alinhac:comm_brutal}
 \Vert [\Lambda^s, \gradphi] f \Vert_{L^2(S)} \leq C(\epsilon\vee\beta) (1+ \Vert \eta \Vert_{H^{s+1}(S)}) \Vert \nabla f \Vert_{H^{s-1}(S)}.
 \end{equation}
\end{lemma}

The issue of loss of derivatives on $\eta$ (and thus, on $\eta_0$ in the case where $\eta$ is given by~\eqref{eqn:phi:def:bary}) is solved by using Alinhac's good unknowns, introduced in~\cite{Alinhac1989}, and defined as
\begin{equation}
\label{apdx:alinhac:definition}
\pointal{f}{s} := \dot{\Lambda}^{s} f -  \frac{\dot{\Lambda}^{s} (\etab+ \epsilon \eta)}{\hb+\epsilon h} \partial_r f.
\end{equation}
See also~\cite{AlazardMetivier2009} and~\cite{Lannes2013} for its use for the water waves equations, as well as~\cite[Proposition 3.17]{Castro2014a} and~\cite{MasmoudiRousset2012} for the present context. Note also that, because of~\eqref{hyp:Hb:reg}, the diffeomorphism $\varphi$ defined in~\eqref{eqn:phi:def} is a perturbation of size $\epsilon \vee \beta$ of the identity. Hence, the commutator in~\eqref{apdx:alinhac:comm_brutal} is of size $\epsilon \vee \beta$. This is also true for the commutators in Lemma~\ref{apdx:alinhac:def}.

\begin{lemma}[Alinhac's good unknown,~\cite{Alinhac1989}]
\label{apdx:alinhac:def}

Let $s_0,s \in \N$ with $s_0 > \frac{d+1}{2}$, $s \geq s_0 + 2$. Assume that $\eta_0 \in H^{s_0+2}$ and~\eqref{hyp:H}. Then there exists $C>0$ depending only on an upperbound on $\vert \eta_0 \vert_{H^{s_0+2}(\R^d)}$ such that the following holds. Let $f \in H^{s}$. We have 
\begin{equation}
\label{apdx:alinhac:comm1}
\dot{\Lambda}^{s}  \gradphi[t,x,r] f = \gradphi[t,x,r] \pointal{f}{s}  + \commal{\dot{\Lambda}^{s}}{\gradphi[t,x,r]}{f},
\end{equation}
and we can write
\begin{equation}
\label{apdx:alinhac:R1}
\Vert \commal{\dot{\Lambda}^{s}}{\gradphi[t,x,r]}{f}\Vert_{L^2(S)} \leq C(\epsilon \vee \beta) (1+\vert \eta_0 \vert_{H^{s}(\R^d)}) \Vert \gradphi[t,x,r] f\Vert_{H^{s-1}(S)} .
\end{equation}
\end{lemma}
\begin{remark}
The statements in Lemma~\ref{apdx:alinhac:def} also hold true with operators of the form $\dot{\Lambda}^{s-k} \partial_r^k$ or $\Lambda^{s-k} \partial_r^k$ with $k \geq 1$. Note however that we do not write such formulas for $\Lambda^s$. The reason is rather technical (its symbol evaluated at $\xi=0$ is not $0$, see the end of the proof of~\cite[Lemma B.2]{Fradin2024}, and in that case Alinhac's good unknown associated with this operator takes a slightly less convenient form).
\end{remark}
See~\cite[Lemma B.2]{Fradin2024} for a proof. We now state a result that shows how the control of Alinhac's good unknown $\pointal{f}{s}$ and of $\Lambda^s f$ are equivalent, up to the control in low-regularity norms.
\begin{lemma}
Let $s_0,s \in \N$, $s_0>\frac{d+1}{2}$, $s \geq s_0 + 2$.
We make the assumption~\eqref{hyp:H} as well as
		$$ \Vert \eta \Vert_{H^{s}(S)}  \leq M,$$
for a constant $M > 0$. Let $f \in H^{s}(S)$.
		Then there exists a constant $C$ depending only on $M$  such that
		\begin{equation}
		\label{eqn:alinhac:equiv}
		\frac{1}{C} \left(\Vert f \Vert_{H^{s,0}(S)} + \Vert f \Vert_{H^{s_0+1}(S)} \right) \leq\Vert \pointal{f}{s} \Vert_{L^2(S)} + \Vert f \Vert_{H^{s_0+1}(S)} \leq C\left(\Vert f \Vert_{H^{s,0}(S)} + \Vert f \Vert_{H^{s_0+1}(S)} \right)
		\end{equation}
\end{lemma}
See~\cite[Lemma B.4]{Fradin2024} for a proof. Note that in the case where $\eta$ is given from~\eqref{eqn:phi:def:bary}, the assumption on the regularity of $\eta$ is a direct consequence of the regularity of $\eta_0$.

\section{Main estimates and proof of Theorem~\ref{thm:euler_phi}}
\label{section:estimates}
In this section, we first assume the existence of a solution of~\eqref{eqn:euler_phi}, and perform energy estimates on it. In Subsection~\ref{subsection:pression} we state estimates on the pressure term, see Proposition~\ref{prop:pression}. Then in Subsection~\ref{subsection:quasilin} we start by defining the energy functional that will be used for the energy estimates on the solution of~\eqref{eqn:euler_phi}. We then differentiate the system~\eqref{eqn:euler_phi}, exhibit its quasilinear structure up to remainder terms, that we estimate in Proposition~\ref{prop:restes}. We perform the energy estimates on a slightly more general quasilinear system in Subsection~\ref{subsection:quasilin:gen}, and apply it in Subsection~\ref{subsection:nrj} to prove the desired energy estimates in Proposition~\ref{prop:energy}.

Eventually, in Subsection~\ref{subsection:scheme}, we state and prove the main result of this paper, Theorem~\ref{thm:euler_phi}, whose statement was sketched and discussed in the introduction. This proof consists on proving the existence on solutions of~\eqref{eqn:euler_phi}, by constructing a sequence of approximate solutions. This scheme for proving the existence of solutions is tailored to the free surface case in our setting and does not automatically come from the a priori energy estimates.

Except in Subsection~\ref{subsection:scheme} where it will be explicitly mentioned, we use the diffeomorphism given in~\eqref{eqn:phi:def:bary}.
\subsection{Pressure estimates}
\label{subsection:pression}
In this subsection we show how the pressure in~\eqref{eqn:euler_phi} is defined and give estimates in Proposition~\ref{prop:pression}.  We also state a technical lemma (Lemma~\ref{lemma:taylor:dt}) on the time derivative of the Taylor coefficient, which also relies on elliptic estimates.\\

Multiplying the first equation in~\eqref{eqn:euler_phi} by $\mu$ and taking the divergence $\gradphi[x,r] \cdot$ of the equations on $V$ and $w$ in~\eqref{eqn:euler_phi}, we get the elliptic equation satisfied by the pressure
\begin{equation}
\label{eqn:pression}
\mu \gradphi \cdot \frac{1}{\rhotot}\gradphi P + \partialphi{r} \left( \frac{1}{\rhotot} \partialphi{r} P\right) = - \epsilon \mu \gradphi[x,r] \cdot \left( \vec{U} \cdot \gradphi[x,r] \vec{U} \right) - \mu g\nabla \cdot \frac{\rhob}{\rhotot}\nabla \eta_0 - g \delta \partialphi{r} \frac{\rho}{\rhotot},
\end{equation}
completed with the boundary conditions
\begin{equation}
\label{eqn:pression:bc}
\sys{ P &= 0 \qquad & \text{ at } r = 0, \\
-\frac{1}{\rhotot} \vec{N_b} \cdot \begin{pmatrix} \mu \gradphi \\ \partialphi{r}\end{pmatrix} P &=  \vec{N_b} \cdot \left( \mu \epsilon \vec{U} \cdot \gradphi[x,r] \vec{U} \right) - \mu \beta\frac{\rhob}{\rhotot} \nabla_x b \cdot \nabla_x \eta_0 + \frac{g \delta \rho}{\rhotot} \qquad & \text{ at } r = -1;} 
\end{equation}
recall that
\begin{equation}
\label{eqn:def:Nb}
\vec{N_b} := \begin{pmatrix} -\beta \nabla b \\ 1 \end{pmatrix}
\end{equation}
is the upward vector normal to the bottom, and $\vec{e_{d+1}}$ the vertical unit vector pointing upward. The boundary condition at the bottom comes from taking the scalar product of the equations on $V$ and $w$ in ~\eqref{eqn:euler_phi} and using the impermeability condition in~\eqref{eqn:euler_phi:bc}. 
From~\eqref{eqn:pression} we get estimates on the pressure term, stated in the following proposition.
%
\begin{proposition}
\label{prop:pression}
Let $s_0,s \in \N$, with $s_0 > \frac{d+1}{2}$, $s \geq s_0 + 1$, $1 \leq k \leq s$, and assume~\eqref{hyp:b}, \eqref{hyp:density}, \eqref{hyp:H}, \eqref{hyp:Hb:reg}, \eqref{hyp:H:reg}. Let $M > 0$ and assume  $\Vert V, \sqrt{\mu} w \Vert_{H^{s,k}} + \Vert \rho \Vert_{H^{s,k}} + |\eta_0|_{H^{s+1}} \leq M$. Then there exists $C > 0$ depending only on $M$ such that
\begin{equation}
\label{eqn:pression:estimee}
\Vert \sqrt{\mu} \gradphi P, \partialphi{r} P \Vert_{H^{s,k}} \leq C \sqrt{\mu} \left(\Vert V, \sqrt{\mu} w \Vert_{H^{s+1,k+1}} + \Vert \vort  \Vert_{H^{s,k}} + \sqrt{\mu}|\eta_0|_{H^{s+1}} + \frac{\delta}{\sqrt{\mu}} \Vert \rho \Vert_{H^{s,k}}\right).
\end{equation}
\end{proposition}
\begin{proof}
First note that~\eqref{eqn:pression} can be reformulated as
\begin{equation}
\label{eqn:pression:A}
\nablamu \cdot \A \nablamu P  = \nablamu \cdot \vec{R},
\end{equation}
with
\begin{equation}
\label{eqn:A:def}
\A := \frac{1}{\rhotot}\coo{(\hb + \epsilon h)I_d & - \sqrt{\mu} \nabla \eta \\
			- \sqrt{\mu} (\nabla \eta)^T & \frac{1+\mu |\nabla \eta|^2}{\hb + \epsilon h}},
\end{equation}
\begin{equation}
\label{eqn:nablamu:def}
\nablamu := \coo{\sqrt{\mu} \nabla \\ \partial_r},
\end{equation}
and
\begin{equation}
\label{eqn:def_src_pression}
\vec{R} = \sqrt{\mu}\coo{ - \frac{\hb + \epsilon h}{\rhotot} \vec{U} \cdot \gradphi[x,r]V - \frac{g\rhob}{\rhotot} \nabla \eta_0 \\
				\frac{\sqrt{\mu}}{\rhotot}\nabla(\etab + \epsilon \eta) \cdot \left( \vec{U} \cdot \gradphi[x,r] V \right) + \frac{\sqrt{\mu}}{\rhotot}\vec{U} \cdot \gradphi[x,r] w + \frac{\delta}{\sqrt{\mu}} \frac{g \rho}{\rhotot}}. 
\end{equation}
We now notice that the bottom boundary condition in~\eqref{eqn:pression:bc} can be recast as
\begin{equation}
\label{eqn:pression:bc:A}
\vec{e_{d+1}} \cdot \A \nablamu P = \vec{e_{d+1}} \cdot \vec{R} \qquad \text{ at } r=-1.
\end{equation}
Therefore, we can use~\cite[Lemma 4.1 eq. (4.9)]{Duchene2022} to solve~\eqref{eqn:pression:A} together with~\eqref{eqn:pression:bc:A} and the homogeneous Dirichlet boundary condition at the surface, and to write the estimates
\begin{equation}
\label{eqn:pression:estimee1}
\Vert \sqrt{\mu} \nabla P, \partial_r P \Vert_{H^{s,k}} \leq C \Vert \vec{R} \Vert_{H^{s,k}},
\end{equation}
where $C$ depends on $M$. Note that according to~\eqref{eqn:gradphi:def} and the product estimate~\eqref{apdx:eqn:pdt:algb}, we get
$$\Vert \sqrt{\mu} \gradphi P, \partialphi{r} P \Vert_{H^{s,k}} \leq C' \Vert \vec{R} \Vert_{H^{s,k}},$$
with $C'>0$ depending on $M$. Thus, it only remains to bound the right-hand side of~\eqref{eqn:pression:estimee1}. The remaining of the proof is devoted to the proof of the following bound, for $s' \geq s_0$ and $1 \leq k' \leq s'$:
\begin{equation}
\label{eqn:R:estimee:finale}
\Vert \vec{R} \Vert_{H^{s',k'}} \leq C \sqrt{\mu}\left( \Vert V, \sqrt{\mu} w \Vert_{H^{s'+1,k'+1}} + \Vert \vort \Vert_{H^{s',k'}} + \sqrt{\mu} \vert \eta_0\vert_{H^{s+1}} + \frac{\delta}{\sqrt{\mu}}\Vert \rho \Vert_{H^{s',k'}} \right).
\end{equation}
Note that the condition $s' \geq s_0$ is less restrictive than the one on $s$ in Proposition~\ref{prop:pression}, hence the notation. Using the definition~\eqref{eqn:def_src_pression} of $\vec{R}$ and the product estimate~\eqref{apdx:eqn:pdt:algb:tame}, we get
\begin{equation}
\label{eqn:R:estimee}
\begin{aligned}
 \frac{1}{\sqrt{\mu}} \Vert \vec{R} \Vert_{H^{s',k'}} &\lesssim \Vert\frac{\hb + \epsilon h}{\rhotot} V \cdot \gradphi V \Vert_{H^{s',k'}} + \Vert \frac{\hb + \epsilon h}{\rhotot} w \partialphi{r} V \Vert_{H^{s',k'}} \\
 &+ \sqrt{\mu}\Vert \frac{1}{\rhotot} V \cdot \gradphi w \Vert_{H^{s',k'}} + \sqrt{\mu} \Vert \frac{1}{\rhotot} w \partialphi{r} w \Vert_{H^{s',k'}}\\
 & + \Vert(1-\epsilon \delta \frac{\rho}{\rhotot} )\nabla \eta_0 \Vert_{H^{s',k'}} + \frac{\delta}{\sqrt{\mu}}\Vert \frac{\rho}{\rhotot} \Vert_{H^{s',k'}}.
\end{aligned}
\end{equation}
We now treat only the second and fourth terms on the right-hand side of~\eqref{eqn:R:estimee}, as the other ones are treated the same way. 
First note that
\begin{equation}
\label{eqn:trick:compo}
\frac{1}{\rhotot} = \frac{1}{\rhob} - \epsilon \delta \frac{\rho/\rhob}{\rhotot}.
\end{equation}
Using the composition estimate \eqref{apdx:eqn:composition:gen} for the second factor with $F:y \mapsto \frac{y/\rhob}{\rhob + \epsilon \delta y}$, with \eqref{hyp:density} and \eqref{hyp:H}, \eqref{hyp:Hb:reg}, \eqref{hyp:H:reg} and the product estimate \eqref{apdx:eqn:pdt:algb}, we get
\begin{equation}
\label{eqn:R:estimee2}
\Vert \frac{\hb + \epsilon h}{\rhob + \epsilon \delta \rho} w \partialphi{r} V \Vert_{H^{s',k'}} \leq C \Vert \sqrt{\mu} w \Vert_{H^{s',k'}} \Vert \frac{1}{\sqrt{\mu}} \partialphi{r} V \Vert_{H^{s',k'}}.
\end{equation}
Using the definition~\eqref{eqn:vort:def} of $\vort$ and triangular inequality, we get
$$ \Vert \frac{1}{\sqrt{\mu}} \partialphi{r} V \Vert_{H^{s',k'}} \leq C \left( \Vert \vort \Vert_{H^{s',k'}} + \Vert \sqrt{\mu} w \Vert_{H^{s'+1,k'+1}} \right).$$
Together with~\eqref{eqn:R:estimee2}, this yields the control on the second term in~\eqref{eqn:R:estimee}. \\
We now treat the fourth term in the right-hand side of~\eqref{eqn:R:estimee}. The factor depending on $\rho$ is treated as before and we omit it. Thanks to the product estimate~\eqref{apdx:eqn:pdt:algb} we get
\begin{equation}
\label{eqn:R:estimee3}
\Vert\sqrt{\mu}  w \partialphi{r} w \Vert_{H^{s',k'}} \leq C \Vert \sqrt{\mu} w \Vert_{H^{s',k'}} \Vert  \partialphi{r} w \Vert_{H^{s',k'}}.
\end{equation}
We use the incompressibility condition in~\eqref{eqn:euler_phi} and triangular inequality to get
$$\Vert\sqrt{\mu}  w \partialphi{r} w \Vert_{H^{s',k'}} \leq C \Vert \sqrt{\mu} w \Vert_{H^{s',k'}}  \Vert  V \Vert_{H^{s'+1,k'+1}}.$$
\end{proof}
Recall now the definition of the Rayleigh-Taylor coefficient
\begin{equation}
\taylor := g\rhob - (\frac{\epsilon}{\hb + \epsilon h} \partial_r P)_{|r=0}.
\end{equation}
We write a similar result as Proposition~\ref{prop:pression}, although with lower regularity, that controls the term
\begin{equation}
\label{eqn:def:dta}
\frac{d}{dt} \taylor := \frac{d}{dt} \left(g\rhob - (\frac{\epsilon}{\hb + \epsilon h}\partial_r P)_{|r=0}\right).
\end{equation}
\begin{lemma}
\label{lemma:taylor:dt}
Under the same assumptions as Proposition~\ref{prop:pression}, we have
\begin{equation}
\label{eqn:taylor:dt}
 \left\vert (\taylor-g\rhob,\frac{d}{dt} \taylor) \right\vert_{L^{\infty}} \leq C\epsilon \left( \Vert V,\sqrt{\mu}w,\sqrt{\mu} \rho \Vert_{H^{s_0+2}}  +  \vert \eta_0 \vert_{H^{s_0+2}} \right). 
\end{equation}
\end{lemma}
\begin{proof}
We only treat the estimate on $\frac{d}{dt} \taylor$, as the other one is treated the same way. The first term in~\eqref{eqn:def:dta} is actually zero because $\rhob$ does not depend on time. We now use the definition~\eqref{eqn:taylor:def} of $\taylor$ so that the second term on the right-hand side of~\eqref{eqn:def:dta} writes
$$\epsilon \frac{\epsilon\partial_t \eta_0}{(\hb + \epsilon h)^2} (\partial_r P)_{|r=0} + \frac{\epsilon}{\hb + \epsilon h} (\partial_r \partial_t P)_{|r=0}.$$
Taking the $L^{\infty}$-norm of this expression, we use the classical embedding $H^{s_0}(\R^d) \hookrightarrow L^{\infty}(\R^d)$ and the embedding~\eqref{apdx:eqn:embedding} to write
$$ \vert\frac{d}{dt} \taylor \vert_{L^{\infty}} \leq C \epsilon \left(\vert \epsilon\partial_t \eta_0 \vert_{H^{s_0}} \Vert(\partial_r P)\Vert_{H^{s_0,1}} + \Vert \partial_r \partial_t P \Vert_{H^{s_0,1}}\right),$$
where we also used~\eqref{hyp:H}.
The first term on the right-hand side is bounded using the equation~\eqref{eqn:euler_phi:bc} for $\partial_t \eta_0$ and the elliptic estimate~\eqref{eqn:pression:estimee} together with the trace estimate~\eqref{apdx:eqn:trace}. For the second term, using~\eqref{hyp:H} and the trace estimate~\eqref{apdx:eqn:trace}, we get
$$ \left\vert\frac{1}{\hb + \epsilon h} (\partial_r \partial_t P)_{|r=0}\right\vert_{L^2} \leq C \Vert \partial_r \partial_t P \Vert_{H^{s_0,1}}.$$
We write the elliptic system satisfied by $\partial_t P$:
\begin{equation}
\label{eqn:pression:A:2}
\nablamu \cdot \A \nablamu \partial_t P  = \nablamu \cdot \vec{R'},
\end{equation}
with
\begin{equation}
\label{eqn:a:temp:1}
\vec{R'} := \partial_t \vec{R} + [\partial_t,\A]\nablamu P,
\end{equation}
where $\vec{R}$ is defined in~\eqref{eqn:def_src_pression}. The equation~\eqref{eqn:pression:A:2} is completed by the boundary conditions obtained by differentiating~\eqref{eqn:pression:bc} with respect to time. We use~\cite[Lemma 2.38]{Lannes2013} with $s=s_0+\frac12$ and~\eqref{eqn:R:estimee:finale} and use that $V,w,\rho,\eta_0$ are solutions to~\eqref{eqn:euler_phi} to estimate their time derivative through space derivatives, to finally get~\eqref{eqn:taylor:dt}.
\end{proof}
\begin{remark}
In this subsection, we use two different elliptic estimates, namely~\cite[Lemma 4.1]{Duchene2022} for Proposition~\ref{prop:pression} and~\cite[Lemma 2.38]{Lannes2013} for Lemma~\ref{lemma:taylor:dt}. The advantage of the first one is that it controls high order vertical derivatives in $P$. The advantage of the second one is that it is more precise for the low regularities, namely $s_0$ here.
\end{remark}

\subsection{Quasilinear structure}
\label{subsection:quasilin}
In this subsection, we define an energy functional in~\eqref{eqn:energy:def}. This is the energy functional whose control can be propagated through the equations~\eqref{eqn:euler_phi}, as shown in Subsection~\ref{subsection:nrj}. We show in Lemma~\ref{lemma:energy:equiv} that it controls the Sobolev norms of the unknowns, as well as two quantities that would only be controlled up to a factor $\frac{1}{\sqrt{\mu}}$ with more naive estimates. In Proposition~\ref{prop:restes} we show how to estimate the remainder terms from the quasilinear formulation~\eqref{eqn:euler_quasilin} of~\eqref{eqn:euler_phi}.\\

We start by describing the energy functional that we will use in Subsection~\ref{subsection:nrj}. To this end, recall the definition of the Rayleigh-Taylor coefficient
\begin{equation}
\label{eqn:taylor:def}
\taylor := g\rhob - (\frac{\epsilon}{\hb + \epsilon h} \partial_r P)_{|r=0}.
\end{equation}
We make the following assumption:
\begin{hyp}
\label{hyp:taylor}
\taylor \geq c_* > 0.
\end{hyp}We consider the energy, for $s_0,s \in \N, s_0 > \frac{d+1}{2},  s \geq s_0+2$
\begin{equation}
\label{eqn:energy:def}
\begin{aligned}
\E :=  \Elow &+ \left\Vert\left(\sqrt{(\hb + \epsilon h)(\rhob + \epsilon \delta \rho)}\pointal{V}{s},\sqrt{\mu (\hb + \epsilon h)(\rhob + \epsilon \delta \rho)}\pointal{w}{s}, \sqrt{\mu (\hb + \epsilon h)} \pointal{\rho}{s}\right)\right\Vert_{L^2(S)}^2 \\
&+ \vert \taylor \point{\eta}{s}_0\vert_{L^2(\R^d)}^2 + \Vert \vort \Vert_{H^{s-1}(S)}^2 ,
\end{aligned}
\end{equation}
where 
\begin{equation}
\label{eqn:energy:def:Elow}
\Elow := \left\Vert\left(V,\sqrt{\mu} w,\rho\right)\right\Vert_{H^{s_0+1}}^2+ g\rhob \vert  \eta_0 \vert_{H^{s_0+1}}^2,
\end{equation}
and where we define
\begin{equation}
\label{eqn:alinhac:def}
\begin{aligned} \point{f}{s} &:= \dot{\Lambda}^{s} f, \qquad
				&\pointal{f}{s} := \dot{\Lambda}^{s}f - \frac{\dot{\Lambda}^{s}(\etab + \epsilon \eta)}{\hb + \epsilon h} \partial_r f.
				\end{aligned} 
				\end{equation}	
We assume the existence of a solution $(V,w,\rho,\eta_0) \in C^0([0,T],H^{s}(S)^{d+2}\times H^s(\R^d))$ of~\eqref{eqn:euler_phi} such that there exists $M > 0$ such that
\begin{hyp}
\label{hyp:M}
\sup_{0\leq t\leq T}\E \leq M.
\end{hyp}
\begin{lemma}
\label{lemma:energy:equiv}
Under Assumptions~\eqref{hyp:H},\eqref{hyp:Hb:reg},~\eqref{hyp:H:reg} and~\eqref{hyp:M}, there exists $C > 0$ depending only on $M$ such that:
\begin{equation}
\label{eqn:drV}
\frac{1}{\sqrt{\mu}}\Vert \partial_r V \Vert_{H^{s-1}} \leq C \E,
\end{equation}
as well as
\begin{equation}
\label{eqn:drw}
\Vert \partial_r w \Vert_{H^{s-1}} \leq C \E,
\end{equation}
and
\begin{equation}
\label{eqn:w:trick}
\Vert w \Vert_{H^{s-1}} \leq C \E.
\end{equation}
Moreover, under the additional assumption~\eqref{hyp:taylor}, for $s\geq s_0+2$, we can write 
\begin{equation}
\label{eqn:energy:equiv}
\frac{1}{C} \left(\Vert V,\sqrt{\mu} w,\sqrt{\mu} \rho \Vert_{H^s}^2 + \Vert \vort \Vert_{H^{s-1}}^2 + \vert \eta_0 \vert_{H^s}^2 \right) \leq \E \leq C \left(\Vert V,\sqrt{\mu} w,\sqrt{\mu} \rho\Vert_{H^s}^2 + \Vert \vort \Vert_{H^{s-1}}^2 + \vert \eta_0 \vert_{H^s}^2 \right).
\end{equation}
\end{lemma}
\begin{proof}
The energy $\E$ is bounded from above with a constant depending on $\taylor > 0$ and the constant in~\eqref{hyp:H}, thanks to Cauchy-Schwarz inequality. \\
The energy $\E$ controls $\Vert \vort\Vert_{H^{s-1}(S)}$, and $\vert \eta \vert_{H^s(\R^d)}$ as $\taylor > 0$. It also controls $\Vert V,\sqrt{\mu} w\Vert_{H^{s,0}}$ by~\eqref{eqn:alinhac:equiv}. Let us now prove that it also controls $\Vert V,\sqrt{\mu} w\Vert_{H^{s,k}}$ for any $k \leq s$ by induction. The base case $k=0$ has already been discussed. Assume that $\E$ controls $\Vert V,\sqrt{\mu} w\Vert_{H^{s,l}}$ for every $0 \leq l < k \leq s$.
From the definition of the first component $\vort^{[x]}$ of the vorticity and the incompressibility condition, we can write
\begin{equation}
\label{eqn:1}
\frac{1}{\sqrt{\mu}} \left( I_d + \mu \epsilon^2 \begin{pmatrix} - \partial_y \eta \partial_x \eta & - (\partial_y \eta)^2 \\ (\partial_x \eta)^2 & \partial_y \eta \partial_x \eta \end{pmatrix} \right) \partialphi{r} V  = \vort^{[x]} - \sqrt{\mu} \nabla^{\perp} w - \sqrt{\mu} (\nabla \eta)^{\perp} (\nabla \cdot V).
\end{equation}
The matrix on the left-hand side is invertible, for any $\mu \geq 0,\epsilon \geq 0$ (its eigenvalues are $1$). Now, because $\eta \in H^{s,l}$ (recall from its definition~\eqref{eqn:phi:def:bary} that $\eta$ is linear in $r$), the matrix $(I_d + \mu \epsilon^2 (\nabla \eta)^{\perp} \nabla \eta)^{-1}$ is bounded uniformly in $\mu$ in $H^{s-1,l}$. Using~\eqref{eqn:1} and the product estimate~\eqref{apdx:eqn:pdt:algb} yields $\partialphi{r} V \in H^{s-1,l}$, so we get that $V \in H^{s,l+1}$, and more precisely we have~\eqref{eqn:drV}. \\
For the bound on $\Lambda^{s-(l+1)} \partial_r^{l+1} w$ needed to derive~\eqref{eqn:drw}, we use the incompressibility condition to get
\begin{equation}
\label{eqn:w:trick:temp}
\Vert \Lambda^{s-(l+1)} \partial_r^{l+1} w \Vert_{L^2} \leq \Vert (\hb + \epsilon h) \gradphi \cdot V \Vert_{H^{s-1,l}}.
\end{equation}
We use product estimates~\eqref{apdx:eqn:pdt:algb} and~\eqref{apdx:eqn:pdt:infty} to write
$$\Vert \Lambda^{s-(l+1)} \partial_r^{l+1} w \Vert_{H^{s,l+1}} \leq C (1+|\eta_0|_{s_0+\frac12}) \vert \eta_0 \vert_{H^{s}} \Vert V \Vert_{H^{s,l+1}}.$$
The previous control on $\Vert V \Vert_{H^{s,l+1}}$ yields the result.\\
For the bound~\eqref{eqn:w:trick}, we write for $(t,x,r) \in [0,T)\times \R^d \times [-1,0]$,
$$\begin{aligned}
w(t,x,r) &= w(t,x,-1) + \int_{-1}^{r} \partial_r w(t,x,r')dr' \\
&= \beta \nabla b \cdot V(t,x,-1) + \int_{-1}^{r}\partial_r w(t,x,r')dr',
\end{aligned}
$$
by the boundary condition at the bottom~\eqref{eqn:euler_phi:bc}. Taking the $H^{s-1}$ norm of the previous expression and using the trace estimate~\eqref{apdx:eqn:trace} as well as~\eqref{eqn:drw}, we get~\eqref{eqn:w:trick}.
\end{proof}
Recall the notation $\dot{\Lambda}^s$ from Subsection~\ref{subsection:notations}. Applying the operator $\diff$ with $s \geq 1$ to~\eqref{eqn:euler_phi} and using~\eqref{apdx:alinhac:comm1} 
yields the following system 
\begin{equation}
\label{eqn:euler_quasilin}
\sys{ \partialphi{t} \pointal{V}{s} + \epsilon \vec{U} \cdot \gradphi[x,r] \pointal{V}{s} + \epsilon \point{w}{s} \partialphi{r} V  + \frac{1}{\rhotot}\gradphi \pointal{P}{s} + g\frac{\rhob}{\rhotot} \nabla \point{\eta}{s}_0 &= \Run, \\
		\mu \left( \partialphi{t} \pointal{w}{s} + \epsilon \vec{U} \cdot \gradphi[x,r] \pointal{w}{s} \right) + \frac{1}{\rhotot} \partialphi{r} \pointal{P}{s} &= \sqrt{\mu} \Rdeux, \\
		\gradphi \cdot \pointal{V}{s} + \partialphi{r} \pointal{w}{s} &= \Rtrois, }
	\qquad \text{ in } S,
\end{equation}
where
\begin{equation}
\label{eqn:euler_quasilin:restes}
\sys{\Run &:= \commal{\diff}{\partialphi{t}}{V} + \epsilon [\diff , V \cdot] \gradphi V +  \epsilon [\diff ; w ,\partialphi{r} V] \\
&+ \epsilon \vec{U} \cdot \commal{\diff}{\gradphi[x,r]}{V} + [\diff, \frac{1}{\rhotot}] \gradphi P + \frac{1}{\rhotot}\commal{\diff}{\gradphi}{P} \\
&+ [\diff, \frac{g\rhob}{\rhotot}]\nabla \eta_0,\\
\sqrt{\mu}\Rdeux &:= \mu \commal{\diff}{\partialphi{t}}{w} + \epsilon \mu [\diff , \vec{U} \cdot] \gradphi[x,r] w \\
&+ \epsilon \mu \vec{U} \cdot \commal{\diff}{\gradphi[x,r]}{w}  + [\diff, \frac{1}{\rhotot}] \partialphi{r} P + \frac{1}{\rhotot} \commal{\diff}{\partialphi{r}}{P}\\
& - \delta \diff \frac{\rho}{\rhotot},\\
\Rtrois &:= \commal{\diff}{\gradphi \cdot }{V} + \commal{\diff}{\partialphi{r}}{w}.
}
\end{equation}
The quasilinear system~\eqref{eqn:euler_quasilin} is completed with the boundary conditions
\begin{equation}
\label{eqn:euler_quasilin:bc}
\sys{\pointal{P}{s}_{|r=0} + g\rhob \point{\eta}{s}_0 &= \taylor \point{\eta}{s}_0, \\
	\partial_t \point{\eta}{s}_0 + \epsilon V_{|r=0} \cdot \nabla \point{\eta}{s}_0 + \epsilon \pointal{V}{s}_{|r=0} \cdot \nabla \eta_0 -\pointal{w}{s}_{|r=0} &= \Rquatre,\\
	\vec{N_b} \cdot \point{\vec{U}}{s}_{|r=-1} &= \Rcinq,\\
}
\end{equation}
where
\begin{equation}
\label{eqn:euler_quasilin:bc:restes}
\sys{\Rquatre &:= \epsilon [\diff;V_{|r=0}\cdot ,\nabla \eta_0] + \epsilon^2 \point{\eta}{s}_0 (\partialphi{r} V)_{|r=0} \cdot \nabla \eta_0 - \epsilon  \point{\eta}{s}_0 (\partialphi{r} w)_{|r=0} ,\\
	\Rcinq &:= [\diff,\vec{N_b}] \cdot \vec{U}_{|r=-1}.}
\end{equation}
Note that the commutator in $\Rquatre$ is a symmetric one, and that the other terms in $\Rquatre$ are here so that we make the good unknowns $\pointal{V}{s}$ and $\pointal{w}{s}$ appear on the left-hand side of~\eqref{eqn:euler_quasilin:bc}. Note also that the boundary condition on $\pointal{P}{s}$ in~\eqref{eqn:euler_quasilin:bc} stems from the definition~\eqref{eqn:alinhac:def} of Alinhac's good unknown $\pointal{P}{s}$ as well as the boundary condition~\eqref{eqn:euler_phi:bc} on the pressure and the definition~\eqref{eqn:taylor:def} of $\taylor$.\\
We also apply an operator of the form $\dot{\Lambda}^{s-k}\partial_r^k$ to the equation on $\rho$ in~\eqref{eqn:euler_phi}, to get
\begin{equation}
\label{eqn:euler_quasilin:rho}
\partialphi{t} \pointal{\rho}{s,k} + \epsilon \vec{U} \cdot \gradphi[x,r] \pointal{\rho}{s,k} =\frac{1}{\sqrt{\mu}}\Rdeuxetdemi,
\end{equation}
where $\pointal{\rho}{s,k}$ is Alinhac's good unknown
$$\pointal{\rho}{s,k} := \dot{\Lambda}^{s-k}\partial_r^k \rho - \frac{\dot{\Lambda}^{s-k}\partial_r^k(\etab + \epsilon \eta)}{\hb + \epsilon h} \partial_r \rho,$$
and where 
\begin{equation}
\label{eqn:euler_quasilin:rho:reste}
\frac{1}{\sqrt{\mu}}\Rdeuxetdemi :=  \commal{\dot{\Lambda}^{s-k}\partial_r^k}{\partialphi{t}}{\rho} + \epsilon [\dot{\Lambda}^{s-k}\partial_r^k, \vec{U}] \cdot \gradphi[x,r] \rho + \epsilon \vec{U} \cdot \commal{\dot{\Lambda}^{s-k}\partial_r^k}{\gradphi[x,r]}{\rho}.
\end{equation}
We also apply an operator of the form $\Lambda^{s-k}\partial_r^{k-1}$ with $0 \leq k-1 \leq s-1$ to the equation~\eqref{eqn:euler_phi:vort}, and get
\begin{equation}
\label{eqn:euler_quasilin:vort}
\begin{aligned}
\partialphi{t} \point{\vort}{s-1,k-1} &+ \epsilon \vec{U} \cdot \gradphi[x,r] \point{\vort}{s-1,k-1} = \Rsix,
\end{aligned}
\end{equation}
with
\begin{equation}
\label{eqn:euler_quasilin:vort:reste}
\begin{aligned}
 \Rsix &:= [\Lambda^{s-k}\partial_r^{k-1}, \partialphi{t} + \epsilon \vec{U} \cdot \gradphi[x,r]] \vort - \epsilon \Lambda^{s-k}\partial_r^{k-1} \left(  \vort^{[x]} \cdot \gradphi \begin{pmatrix} V \\ \sqrt{\mu} w \end{pmatrix}
+ \vort^{[z]} \partialphi{r} \begin{pmatrix} \frac{1}{\sqrt{\mu}} V \\ w \end{pmatrix} \right) \\
&+ \frac{\delta}{\sqrt{\mu}} \Lambda^{s-k}\partial_r^k \vec{F}.
\end{aligned}
\end{equation}
We estimate the remainders $\Run, \dots , \Rsix$ in the following proposition.
\begin{proposition}
\label{prop:restes}
Let $s_0,s \in \N$, $s_0 > \frac{d+1}{2}$, $s \geq s_0 + 2$. Assume that there exists $T>0$ and $(V,w,\eta_0) \in C^0([0,T), (H^{s}(S))^{d+1} \times H^s(\R^d))$ satisfying~\eqref{eqn:euler_phi} as well as~\eqref{hyp:b},~\eqref{hyp:H},~\eqref{hyp:M}.
Then there exists $C > 0$ a constant depending only on $M$ such that
\begin{equation}
\label{eqn:restes:estimee}
\Vert (\Run,\Rdeux,\Rtrois,\Rdeuxetdemi,\Rsix)\Vert_{L^2}  + \vert \Rquatre \vert_{L^2} + \vert \Rcinq \vert_{H^{\frac12}} \leq  (\epsilon \vee \beta\vee (\delta/\sqrt{\mu})) C \E^{\frac12}.
\end{equation}
\end{proposition}
\begin{proof}
We only treat the terms in $\Run, \Rquatre,\Rcinq$ and $\Rdeuxetdemi $, the other ones can be treated the same way. For the first term in $\Run$, we use~\eqref{apdx:alinhac:R1} to write
$$\Vert \commal{\diff}{\partialphi{t}}{V} \Vert_{L^2} \leq C (\epsilon \vee \beta) \Vert \gradphi[t,x,r] V \Vert_{H^{s-1}}.$$
Using the product estimates~\eqref{apdx:eqn:pdt:algb} and~\eqref{apdx:eqn:pdt:infty}, we get
$$\Vert \gradphi[x,r] V \Vert_{H^{s-1}} \leq C (1+| \eta_0 |_{H^s})\Vert V \Vert_{H^{s}} \leq C \E^{\frac12},$$
where the last inequality stems from~\eqref{eqn:energy:equiv}. Using now~\eqref{eqn:euler_phi}, we can write
\begin{equation}
\label{eqn:quasilin:dt}
\begin{aligned}
\Vert \partialphi{t} V \Vert_{H^{s-1}} &= \Vert  \epsilon \vec{U} \cdot \gradphi[x,r] V + \frac{1}{\rhotot}\gradphi P + g \frac{\rhob}{\rhotot} \nabla \eta_0 \Vert_{H^{s-1}}\\
&\leq \Vert  \epsilon V \cdot \gradphi V\Vert_{H^{s-1}} + \epsilon \left\Vert \sqrt{\mu} w \left( \frac{1}{\sqrt{\mu}} \partialphi{r} V \right) \right\Vert_{H^{s-1}}\\
& + \Vert \frac{1}{\rhotot}\gradphi P \Vert_{H^{s-1}} + \vert g\frac{\rhob}{\rhotot} \nabla \eta_0 \vert_{H^{s-1}}.
\end{aligned}
\end{equation}
For the first term on the right-hand side of~\eqref{eqn:quasilin:dt}, we use the product estimate~\eqref{apdx:eqn:pdt:algb} as well as the definition~\eqref{eqn:gradphi:def} of $\gradphi$, to write
$$\Vert V \cdot \gradphi V\Vert_{H^{s-1}} \leq C \Vert V \Vert_{H^s} \leq C \E^{\frac12}.$$
Recall indeed that we allow $C$ to depend on an upperbound on $\Vert V \Vert_{H^s}$. For the second term in~\eqref{eqn:quasilin:dt}, we use the product estimate~\eqref{apdx:eqn:pdt:algb} and~\eqref{eqn:drV} to write
\begin{equation}
\label{eqn:restes:drV:utilisation}
\left\Vert \sqrt{\mu} w \left( \frac{1}{\sqrt{\mu}} \partialphi{r} V \right) \right\Vert_{H^{s-1}} \leq C \E^{\frac12}.
\end{equation}
For the third term in~\eqref{eqn:quasilin:dt}, we use the expression~\eqref{eqn:trick:compo} for the factor $1/(\rhotot)$ and the product estimate~\eqref{apdx:eqn:pdt:algb} to write
$$\Vert \frac{1}{\rhotot} \gradphi P \Vert_{H^{s-1}} \leq C (\frac{1}{\rhob} + \epsilon \delta \Vert \frac{\rho/\rhob}{\rhotot} \Vert_{H^{s-1}}) \Vert \gradphi P \Vert_{H^{s-1}}. $$
Using the composition estimate~\eqref{apdx:eqn:composition:gen} and the estimate~\eqref{eqn:pression:estimee}, we get
$$\Vert \frac{1}{\rhotot} \gradphi P \Vert_{H^{s-1}} \leq C \E^{\frac12}.$$
The last term in~\eqref{eqn:quasilin:dt} is treated like the third one and we omit it. For the second term of $\Run$, we use the commutator estimate~\eqref{apdx:eqn:commutator_sym:horizontal}, and the definition~\eqref{eqn:gradphi:def} of $\gradphi$ together with product estimate~\eqref{apdx:eqn:pdt:algb} to write
$$ \begin{aligned} \Vert [\diff , V \cdot] \gradphi V \Vert_{L^2} &\leq C \Vert V \Vert_{H^{s}} \Vert \gradphi V \Vert_{H^{s-1}} \\
& \leq C \E^{\frac12}. \end{aligned} $$
For the third term of $\Run$, we use the symmetric commutator estimate~\eqref{apdx:eqn:commutator_sym:horizontal} to write
$$ \begin{aligned} \Vert [\diff ; w ,\partialphi{r} V] \Vert_{L^2} &\leq C \Vert \sqrt{\mu}w \Vert_{H^{s-1}} \Vert \frac{1}{\sqrt{\mu}}\partialphi{r} V \Vert_{H^{s-1}}. \end{aligned} $$
We then use~\eqref{eqn:drV} to write
$$ \Vert [\diff ; w ,\partialphi{r} V] \Vert_{L^2} \leq C \E^{\frac12}.$$
The remaining terms in $\Run$ are treated as the first term in $\Run$. Combining all these estimates, we get
$$ \Vert \Run \Vert_{L^2} \leq C (\epsilon \vee \beta) \E^{\frac12}.$$
For the first term of $\Rquatre$, we use the commutator estimate~\eqref{apdx:eqn:commutator:Rd:classic} to write
$$ \vert [\diff;V\cdot ,\nabla \eta_0]\vert_{L^2} \leq C \vert V_{|r=0}\vert_{H^{s-1}} \vert \nabla \eta_0 \vert_{H^{s-1}} .$$
Using the trace estimate~\eqref{apdx:eqn:trace}, we get
$$ \vert [\diff;V\cdot ,\nabla \eta_0]\vert_{L^2} \leq C \Vert V\Vert_{H^{s}} \vert  \eta_0 \vert_{H^{s}} \leq C \E^{\frac12}.$$
For the second term of $\Rquatre$, we use the Hölder inequality and the embedding~\eqref{apdx:eqn:embedding} to get
$$\vert \point{\eta}{s}_0 \partialphi{r} V \cdot \nabla \eta_0 \vert_{L^2} \leq \vert \point{\eta}{s}_0 \vert_{L^2} \vert \partialphi{r} V \vert_{L^2} \vert \nabla \eta_0 \vert_{L^{\infty}} \leq C \vert \eta_0 \vert_{H^s} \Vert V \Vert_{H^{s_0+1,2}} \vert \eta_0 \vert_{H^{s_0+1}} \leq C \E^{\frac12}.$$
The third term of $\Rquatre$ is  treated the same way and we eventually get
$$ \vert \Rquatre \vert_{L^2} \leq C (\epsilon \vee \beta) \E^{\frac12}.$$
For the estimate on $\Rcinq$, we use the definition of $\vec{N_b}$ to write
$$[\diff,\vec{N_b}]\cdot \vec{U} = \left[\diff,\begin{pmatrix} -\beta \nabla b \\ 1 \end{pmatrix} \right] \cdot \vec{U} = -\beta[\diff, \nabla b] \cdot V.$$
We now use the commutator estimate~\eqref{apdx:eqn:commutator:Rd:sprime} with $s' = \frac12$ to write
$$\vert \beta [\diff,\nabla b] \cdot V_{|r=-1} \vert_{H^{\frac12}} \leq \beta C \vert V_{|r=-1} \vert_{H^{s-\frac12}} \leq \beta C \E,$$
where the last inequality results from the trace estimate~\eqref{apdx:eqn:trace}.\\
For $\Rdeuxetdemi$, we only treat the second term, as the other ones are treated the same way. We write, by the commutator estimate~\eqref{apdx:eqn:commutator:gen}, as $s\geq s_0+2$:
$$ \Vert\sqrt{\mu} [\dot{\Lambda}^{s-k} \partial_r^k, \vec{U}] \cdot \gradphi[x,r] \rho \Vert_{L^2} \leq C  \Vert \sqrt{\mu} \vec{U} \Vert_{H^{s,k}} \Vert \rho \Vert_{s,k} \leq C \E.$$
The terms in $\Rdeux$ and $\Rsix$ are treated the same way and we omit then. Not however that they yield a contribution of size $\frac{\delta}{\sqrt{\mu}}$, see the definitions~\eqref{eqn:euler_quasilin:restes} and~\eqref{eqn:euler_quasilin:vort:reste}.
\end{proof}
\subsection{$L^2$ estimates on quasilinear equations}
\label{subsection:quasilin:gen}
In this subsection we study quasilinear systems of the form~\eqref{eqn:euler_quasilin} and~\eqref{eqn:euler_quasilin:vort}, with general source terms $\Run,\Rdeux,\Rtrois,\Rquatre,\Rcinq,\Rdeuxetdemi,\Rsix$.\\
First, we consider the following system on the unknowns $\pointgen{V},\pointgen{w},\pointgen{\eta}_0$
\begin{equation}
\label{eqn:quasilin:gen}
\sys{ \partialphi{t} \pointgen{V} + \epsilon \vec{U} \cdot \gradphi[x,r] \pointgen{V} + \epsilon \pointgen{w} \partialphi{r} V  + \frac{1}{\rhob + \epsilon \delta \rho}\gradphi \pointgen{P} + g \frac{\rhob}{\rhob + \epsilon \delta \rho} \nabla \pointgen{\eta_0} &= \Run, \\
		\mu \left( \partialphi{t} \pointgen{w} + \epsilon \vec{U} \cdot \gradphi[x,r] \pointgen{w} \right) + \frac{1}{\rhob + \epsilon \delta \rho}\partialphi{r} \pointgen{P} &= \sqrt{\mu}\Rdeux, \\
		\gradphi \cdot \pointgen{V} + \partialphi{r} \pointgen{w} &= \Rtrois, }
	\qquad \text{ in } S.
\end{equation}
Here, $\Run, \Rdeux,\Rtrois$ are source terms. The quantity $\pointgen{P}$ is also considered as a source term, in the sense that we do not use the elliptic equation~\eqref{eqn:pression}, but rather keep the dependency in $\pointgen{P}$ in the estimates of the following Lemma~\ref{lemma:quasilin:gen}. We take $\vec{U} := (V,w)^T$, $\eta_0$ and $\rho$ a vector field and two scalar fields, respectively, that are coefficients, not related to the unknowns $\pointgen{\vec{U}},\pointgen{\eta}_0$  a priori. Recall that the operators $\gradphi[x,r], \partialphi{t}$ are defined from $\eta_0$ through~\eqref{eqn:phi:def},\eqref{eqn:phi:def:bary} and $\hb,h$ are defined through~\eqref{eqn:h:def}. We assume that $\vec{U}$ is divergence free
\begin{hyp}
\label{hyp:U:div}
\gradphi[x,r] \cdot \vec{U} = 0,
\end{hyp}and that it is tangential to the surface and bottom, that is
\begin{hyp}
\label{hyp:U:BC}
\begin{aligned}
\partial_t \eta_0 + \epsilon V_{|r=0} \cdot \nabla \eta_0 - w_{|r=0} &= 0,\\
\vec{U}_{|r=-1} \cdot \vec{N_b} &= 0.
\end{aligned}
\end{hyp}We also assume that there exists $\Mb > 0$ such that
\begin{hyp}
\label{hyp:Mb}
\Vert (V,\sqrt{w}) \Vert_{H^{s_0+2}} + \Vert \frac{1}{\sqrt{\mu}} \partialphi{r} V \Vert_{H^{s_0+1}} + \vert \eta_0 \vert_{H^{s_0+2}} \leq \Mb.
\end{hyp}We complete~\eqref{eqn:quasilin:gen} with the following boundary condition
\begin{equation}
\label{eqn:quasilin:gen:bc}
\sys{ \pointgen{P}_{|r=0} + g\rhob \pointgen{\eta}_0& = \taylorgen  \pointgen{\eta}_0,\\
\partial_t \pointgen{\eta}_0 + \epsilon V_{|r=0} \cdot \nabla \pointgen{\eta}_0 + \epsilon \pointgen{V}_{|r=0} \cdot \nabla \eta_0 -\pointgen{w}_{|r=0} &= \Rquatre,\\
	\vec{N_b} \cdot \pointgen{\vec{U}} &= \Rcinq,\\
}
\end{equation}
where $\Rquatre,\Rcinq$ are source terms. The quantity $\taylorgen $ is the Rayleigh-Taylor coefficient; see~\eqref{eqn:quasilin:taylor} below for  the derivation of the first boundary condition~\eqref{eqn:quasilin:gen:bc} in the quasilinearization process of~\eqref{eqn:euler_phi}.
We then consider the following transport equation
\begin{equation}
\label{eqn:quasilin:vort:gen}
\partialphi{t} \pointf + \epsilon \vec{U} \cdot \gradphi[x,r] \pointf = \Rsix,
\end{equation}
where $\Rsix$ is a source term. Note that~\eqref{eqn:quasilin:vort:gen} is the quasilinear structure of the equations~\eqref{eqn:euler_quasilin:rho} on the density and~\eqref{eqn:euler_quasilin:vort} on the vorticity.\\
We also assume that there exists a constant $c_*>0$ such that 
\begin{hyp}
\label{hyp:quasilin:taylor}
\taylorgen \geq c_* \qquad \text{ and } \qquad \vert(\taylorgen , \partial_t \taylorgen ) \vert_{L^{\infty}(\R^d)} \leq \Mb.
\end{hyp}
\begin{remark}
Note that in~\eqref{eqn:euler_quasilin} and~\eqref{eqn:euler_quasilin:bc}, the unknowns are $\pointgen{\vec{U}} = (\pointgen{V},\pointgen{w})^T$ and $\pointgen{\eta}_0$. The quantities $\vec{U} = (V,w)^T$ and $\eta_0$ play the role of general coefficients and are not related to the unknowns $\pointgen{\vec{U}},\pointgen {\eta}_0$ a priori.
\end{remark}
\begin{lemma}
\label{lemma:quasilin:gen}
Let $s_0,s \in \N$, with $s_0 > \frac{d+1}{2}$ and $s \geq s_0+2$, and let $T > 0$. Let $(\vec{U},\eta_0) \in C^0([0,T), H^{s_0+2,3}(S)^{d+1} \times H^{s_0+2}(\R^d))$ satisfying \eqref{hyp:b},~ \eqref{hyp:H}, \eqref{hyp:Hb:reg}, \eqref{hyp:H:reg}, \eqref{hyp:U:div}, \eqref{hyp:U:BC}, \eqref{hyp:Mb}, \eqref{hyp:quasilin:taylor}. \\
Let $(\pointgen{V},\pointgen{w},\pointgen{\eta_0}) \in C^0([0,T),H^1(S)^{d+1} \times H^1(\R^d))$ be a solution of~\eqref{eqn:quasilin:gen} with boundary conditions~\eqref{eqn:quasilin:gen:bc}, of energy bounded by a constant $M$ as in~\eqref{hyp:M}. Then there exists $C$ depending only on $\Mb$ and $M$ such that
\begin{equation}
\label{eqn:quasilin:gen:energy:L2}
\begin{aligned}
&\frac{d}{dt} \left( \Vert \sqrt{(\hb + \epsilon h)(\rhob + \epsilon \delta \rho)} \pointgen{V}, \sqrt{\mu(\hb+ \epsilon h)(\rhob + \epsilon \delta \rho)} \pointgen{w}\Vert_{L^2}^2 + \vert \sqrt{\taylorgen } \pointgen{\eta}_0 \vert_{L^2}^2 \right)\\
& \leq C (\epsilon \vee \beta \vee \frac{\delta}{\mu} ) \left( \Vert \pointgen{V},\sqrt{\mu}\pointgen{w},\sqrt{\mu}\pointgen{\rho}\Vert_{L^2}^2 + \vert\pointgen{\eta}_0 \vert_{L^2}^2 \right) \\
&+ C \left(\Vert (\Run,\Rdeux,\Rtrois)\Vert_{L^2} + \vert (\Rquatre,\Rcinq)\vert_{L^2} \right) \left( \Vert \pointgen{V},\sqrt{\mu}\pointgen{w},\sqrt{\mu} \pointgen{\rho}\Vert_{L^2} + \vert\pointgen{\eta}_0 \vert_{L^2} \right)\\
 &+ C \Vert \pointgen{P} \Vert_{L^2} \Vert \Rtrois \Vert_{L^2} + C \vert \pointgen{P}_{r=-1} \vert_{H^{-\frac12}} \vert \Rcinq \vert_{H^{\frac12}}.
\end{aligned}
\end{equation}
If now $\pointf$  is a smooth solution of~\eqref{eqn:quasilin:vort:gen}, then
\begin{equation}
\label{eqn:quasilin:gen:vort:energy:L2}
\frac{d}{dt}\Vert \sqrt{\hb + \epsilon h} \pointf \Vert_{L^2}^2 \leq C \epsilon \Vert \pointf \Vert_{L^2}^2 + C \Vert \Rsix \Vert_{L^2} \Vert \pointf \Vert_{L^2}.
\end{equation}
\end{lemma}
\begin{proof}
We start with the estimate~\eqref{eqn:quasilin:gen:energy:L2}. 
We integrate the first two equations in~\eqref{eqn:quasilin:gen} and the second equation in~\eqref{eqn:quasilin:gen:bc} against $((\hb + \epsilon h)(\rhob + \epsilon \delta \rho)\pointgen{V}, (\hb + \epsilon h)(\rhob + \epsilon \delta \rho) \pointgen{w}, \taylorgen  \pointgen{\eta}_0)$, to get
\begin{equation}
\label{eqn:energy:V:1}
\begin{aligned}
\int_S & (\hb + \epsilon h)(\rhotot) \partialphi{t} \left( |\pointgen{V}|^2 + \mu |\pointgen{w}|^2 \right) & (i)\\
&+ \int_S \epsilon (\hb + \epsilon h)(\rhotot) \vec{U} \cdot \gradphi[x,r] |\pointgen{V}|^2 + \epsilon (\hb + \epsilon h)(\rhotot) \vec{U} \cdot \gradphi[x,r] |\sqrt{\mu}\pointgen{w}|^2 & (ii)\\
&+ \int_{R^d} \taylorgen  \frac{d}{dt}|\pointgen{\eta}_0|^2 + \epsilon \taylorgen  V_{|r=0} \cdot \nabla |\pointgen{\eta}_0|^2 & (iii)\\ 
&+ \int_S \epsilon (\hb + \epsilon h)(\rhotot) \pointgen{V}\pointgen{w} \partialphi{r} V  & (iv)\\
&+ \int_S (\hb + \epsilon h) \gradphi[x,r] \pointgen{P} \cdot \pointgen{\vec{U}} +  g\rhob (\hb + \epsilon h)\nabla \pointgen{\eta_0}\cdot \pointgen{V} & (v) \\
&+ \int_{\R^d} \taylorgen  (\pointgen{V}_{|r=0} \cdot \nabla \eta_0 - \pointgen{w}_{|r=0}) \pointgen{\eta}_0 & (vi)\\
& = \int_S (\hb + \epsilon h)(\rhotot) \Run \cdot \pointgen{V} + (\hb + \epsilon h)(\rhotot) \Rdeux \pointgen{w} + \int_{R^d} \taylorgen  \Rquatre \pointgen{\eta}_0;
\end{aligned} 
\end{equation}
here, we denote by $\pointgen{\vec{U}} := (\pointgen{V},\pointgen{w})^T$.
The terms on the right-hand side can be bounded from above by the Cauchy-Schwarz inequality so that we focus on the terms on the left-hand side.
Using the definitions~\eqref{eqn:h:def} and~\eqref{eqn:phi:def:bary}, we can write
\begin{equation}
\label{eqn:quasilin:dtint}
\begin{aligned}
\frac{d}{dt} \int_S (\hb + \epsilon h) f &= \int_S \partial_t(\epsilon \eta_0) f + \int_S (\hb + \epsilon h) \partial_t f \\
&= \int_{\R^d} \epsilon \partial_t \eta_0 f_{|r=0} + \int_S (\hb + \epsilon h) \partialphi{t} f,
\end{aligned}
\end{equation}
for $f$ any smooth enough function defined on $[0,T]\times S$. Using~\eqref{eqn:quasilin:dtint} and the divergence-free condition~\eqref{hyp:U:div} and boundary condition~\eqref{hyp:U:BC} yields that the terms $(i),(ii)$ and $(iii)$ amount to the left-hand side of~\eqref{eqn:quasilin:gen:energy:L2}, up to remainder terms that are easily bounded by the product estimate~\eqref{apdx:eqn:pdt:tame}.\\

For the first term in $(iv)$ in~\eqref{eqn:energy:V:1}, we use Hölder inequality and the embedding of Lemma~\ref{apdx:lemma:embedding} to get
$$\begin{aligned}
\int_S \epsilon (\hb + \epsilon h) \pointgen{V}\pointgen{w} \partialphi{r} V &\leq \epsilon C \Vert \hb + \epsilon h \Vert_{L^{\infty}} \Vert \pointgen{V} \Vert_{L^2} \Vert \pointgen{w} \Vert_{L^2} \Vert \partialphi{r} V \Vert_{L^{\infty}}\\
&\leq \epsilon C \Vert \pointgen{V} \Vert_{L^2} \Vert \sqrt{\mu} \pointgen{w} \Vert_{L^2} \Vert\frac{1}{\sqrt{\mu}}  \partialphi{r} V \Vert_{H^{s_0}}\\
& \leq \epsilon C \Vert (\pointgen{V},\sqrt{\mu}\pointgen{w})\Vert_{L^2};
\end{aligned}$$
we crucially used the second term in~\eqref{hyp:Mb} to obtain a bound that is uniform in $\mu$.\\
Now note that the terms in $(v)$ can be written as
$$ \int_S (\hb + \epsilon h) \gradphi[x,r] \left( \pointgen{P} + g\rhob \pointgen{\eta}_0 \right) \cdot \pointgen{\vec{U}};$$
indeed, as $\pointgen{\eta}_0$ depends only on $t,x$ but not on $r$, the definition~\eqref{eqn:gradphi:def} of $\gradphi[x,r]$ yields $\gradphi[x,r] \pointgen{\eta}_0 = \begin{pmatrix} \nabla \pointgen{\eta}_0 \\ 0 \end{pmatrix}.$ We thus use the definition~\eqref{eqn:quasilin:gen:bc} of $\taylorgen $ and integrate by parts thanks to Lemma~\ref{apdx:lemma:IPP}, to write
\begin{equation}
\label{eqn:pression:ipp}
\begin{aligned}
\int_S (\hb + \epsilon h) \gradphi[x,r] \left( \pointgen{P} + g \rhob\pointgen{\eta_0}\right)  \cdot \pointgen{\vec{U}} &= \int_{r=0}\taylorgen \pointgen{\eta_0} \pointgen{\vec{U}} \cdot \vec{N}_s \\
&- \int_{r=-1} (\pointgen{P}_{|r=-1} + g \rhob\pointgen{\eta_0})  \pointgen{\vec{U}}_{|r=-1} \cdot \vec{N}_b \\
&- \int_S (\hb + \epsilon h)(\pointgen{P} + g \rhob\pointgen{\eta_0})  \gradphi[x,r] \cdot \pointgen{\vec{U}},
\end{aligned}
\end{equation}
where we used the notation $\vec{N}_s := \begin{pmatrix} -\epsilon \nabla \eta_0 \\ 1\end{pmatrix}$ and~\eqref{eqn:def:Nb} of $\vec{N}_b$. The third term on the right-hand side of~\eqref{eqn:pression:ipp} is bounded as one can use the divergence-free equation in~\eqref{eqn:quasilin:gen}, Cauchy-Schwarz inequality.\\
For the second term on the right-hand side of~\eqref{eqn:pression:ipp}, we use the boundary condition~\eqref{eqn:quasilin:gen:bc} to write

\begin{equation}
\label{eqn:temp:3}
\begin{aligned}
\left| \int_{r=-1} \left( \pointgen{P} + g \rhob \pointgen{\eta}_0 \right) \pointgen{\vec{U}} \cdot \vec{N}_b \right|  &\leq \left| \int_{r=-1}  g \rhob \pointgen{\eta}_0 \Rcinq \right| + \left| \int_{r=-1}  \pointgen{P} \Rcinq \right|\\
&\leq C \vert \pointgen{\eta}_0 \vert_{L^2} \vert \Rcinq \vert_{L^2} + C\vert \pointgen{P}_{r=-1} \vert_{H^{-\frac12}} \vert \Rcinq \vert_{H^{\frac12}}.
\end{aligned}
\end{equation}
For the first term on the right-hand side of~\eqref{eqn:pression:ipp}, recall that the boundary condition~\eqref{eqn:quasilin:gen:bc} on $\pointgen{P}_{|r=0}$ reads
$$ \pointgen{P}_{|r=0} + g\rhob  \pointgen{\eta}_0 = \taylorgen  \pointgen{\eta}_0,$$
where $\taylorgen $ is the Taylor coefficient. Therefore, this term and the term in $(vi)$ in~\eqref{eqn:energy:V:1} cancel out each other.\\
We eventually get the estimate~\eqref{eqn:quasilin:gen:energy:L2}. \\
For the equation~\eqref{eqn:quasilin:gen:vort:energy:L2}, we integrate~\eqref{eqn:quasilin:vort:gen} against $(\hb + \epsilon h)\pointf$ to get
\begin{equation}
\label{eqn:energy:vort:gen}
\begin{aligned}
&\int_S (\hb + \epsilon h) \partialphi{t} |\pointf|^2 + \epsilon (\hb + \epsilon h) \vec{U} \cdot \gradphi[x,r] |\pointf|^2\\
&= \int_S (\hb + \epsilon h)\Rsix \cdot\pointf.
\end{aligned}
\end{equation}
The right-hand side of~\eqref{eqn:energy:vort:gen} is bounded, by Cauchy-Schwarz inequality. The first and second terms of~\eqref{eqn:energy:vort:gen} are treated the same way as~\eqref{eqn:energy:V:1} (i) and~\eqref{eqn:energy:V:1} (ii) and we omit it.
\end{proof}
\subsection{Energy estimates}
\label{subsection:nrj}
We are now in a position to prove the energy estimates on the system~\eqref{eqn:euler_phi}. For this, we assume the existence of a solution of~\eqref{eqn:euler_phi} such that the energy defined in~\eqref{eqn:energy:def} is bounded from above by a constant $M>0$, i.e.~\eqref{hyp:M} holds.
\begin{proposition}
\label{prop:energy}
Let $s_0, s \in \N$, with $s_0 > \frac{d+1}{2}$, $s \geq s_0 + 2$. Assume that there exists $T>0$ and $(V,w,\rho,\eta_0) \in C^0([0,T),H^s(S)^{d+2} \times H^s(\R^d))$ a solution of~\eqref{eqn:euler_phi}, satisfying the initial condition~\eqref{eqn:euler_phi:ci} and boundary conditions~\eqref{eqn:euler_phi:bc}, as well as~\eqref{hyp:b},~\eqref{hyp:density},~\eqref{hyp:H},~\eqref{hyp:Hb:reg},~\eqref{hyp:H:reg},~\eqref{hyp:taylor},~\eqref{hyp:M}. Then there exists $C$ depending only on $M$ such that
$$\frac{d}{dt} \E \leq C(\epsilon \vee \beta \vee \frac{\delta}{\mu}) \E.$$
\end{proposition}
\begin{remark}
\label{rk:tps_long:vort}
The energy estimate in Proposition~\ref{prop:energy} mainly relies on two main estimates. First, we derive in Step 2 an estimate on horizontal derivatives of $(V,w,\eta_0)$. The main advantage to use only horizontal derivatives in the operator $\dot{\Lambda}^{s}$ used to differentiate the system~\eqref{eqn:euler_phi} is that we can also differentiate the boundary conditions~\eqref{eqn:euler_phi:bc}. \\
The reason why we also derive an estimate on the vorticity in Step 3 is two-fold. First, it allows us to control the vertical derivatives of $V$ and $w$, that was not possible in Step 1, thanks to Lemma~\ref{lemma:energy:equiv}. Then, it allows us to propagate the condition $\partial_r V \sim \sqrt{\mu}$ (see for instance~\eqref{eqn:drV}).\\
Lastly, note that the estimate~\eqref{eqn:vort:estimee} yields a factor $\epsilon$ on the right-hand side. Thus, the presence of vorticity does not seem to be a major obstruction to get a large-time existence result for~\eqref{eqn:euler_phi}. The terms of order $\beta$ (responsible for the short-time result in the case $\delta = 0$ of no density variations) rather appear in Step 2, in the boundary term at the bottom resulting from the integration by parts for the pressure term  (see the term $(v)$ in~\eqref{eqn:energy:V:1} as well as the bound on $\Rcinq$ in~\eqref{eqn:restes:estimee}).
\end{remark}
\begin{proof}
In the proof we assume that $V,w,\rho,\eta_0$ are smooth enough so that the following computations make sense, see Remark~\ref{rk:triche}. \\
{\bf \underline{Step 1}: Control of the low regularity part $\Elow$:} \\
For the low-regularity part $\E_0$ in the definition~\eqref{eqn:energy:def} of $\E$, {\it let $s' := s_0+1$.} We first apply $\Lambda^{s'}$ to~\eqref{eqn:euler_phi}. We use extensively the notation $\point{f}{s'} := \Lambda^{s'} f$ for a quantity $f$ defined in the strip $S$ or on $\R^d$. We thus obtain an equation of the form 
\begin{equation}
\label{eqn:euler_quasilin_dummy}
\sys{
\partial_t \point{V}{s'} + g\rhob \nabla \point{\eta}{s'}_0&= R^0_1,\\
\mu \partial_t \point{w}{s'} &= \sqrt{\mu}R^0_2,\\
(\hb + \epsilon h)\gradphi \cdot \point{V}{s'} + \partial_r \point{w}{s'} &= R^0_4,
}
\end{equation}
with
\begin{equation}
\label{eqn:quasilin_dummy:restes}
\Vert (R^0_1,R^0_2,R^0_3,R^0_4) \Vert_{L^2} \leq C(\epsilon \vee \beta \frac{\delta}{\sqrt{\mu}}) \left(\Vert (V,\sqrt{\mu} w, \sqrt{\mu} \rho, \vort) \Vert_{H^{s_0+2}} +\vert \eta_0 \vert_{H^{s_0+2}}\right).
\end{equation}
The estimate~\eqref{eqn:quasilin_dummy:restes} is proven as is~\eqref{eqn:restes:estimee} and we omit it.
We also apply $\Lambda^{s'}$ to the boundary conditions~\eqref{eqn:euler_phi:bc} to get
\begin{equation}
\label{eqn:quasilin_dummy:bc}
\partial_t \point{\eta}{s'}_0 - \point{w}{s} =R^0_3, \\
\end{equation}
with 
$$\vert R^0_3 \vert_{H^{s_0+1}} \leq C \epsilon \left( \Vert (V,\sqrt{\mu} w) \Vert_{H^{s_0+2}} +\vert \eta_0 \vert_{H^{s_0+2}}\right).$$
Note that there is a loss of derivatives in the above upper bound of the remainder terms, but as $s' = s_0+1$ and $s \geq s_0 +2$, we can still write
$$\Vert (R^0_1,R^0_2,R^0_4) \Vert_{L^2} + \vert R^0_3 \vert_{L^2}\leq C(\epsilon \vee \beta \vee \frac{\delta}{\sqrt{\mu}}) \E^{\frac12}.$$
However, note that we do not consider the terms $\nabla \point{\eta}{s'}_0$ and $-\point{w}{s'}$ in~\eqref{eqn:euler_quasilin_dummy} and~\eqref{eqn:quasilin_dummy:bc} as remainders, since they are not bounded from above by terms containing a factor $\epsilon \vee \beta\vee\frac{\delta}{\sqrt{\mu}}$. The system~\eqref{eqn:euler_quasilin_dummy} together with the boundary conditions~\eqref{eqn:quasilin_dummy:bc} are in the form of the quasilinear system~\eqref{eqn:quasilin:gen}. Using the estimate~\eqref{eqn:quasilin:gen:energy:L2} together with~\eqref{eqn:quasilin_dummy:restes}, we get
\begin{equation}
\label{eqn:temp:10}
\frac{d}{dt}\left( \Vert  V, \sqrt{\mu}w\Vert_{H^{s_0+1,0}} + \vert \eta_0 \vert_{H^{s_0+1}} \right) \leq C(\epsilon \vee \beta\vee\frac{\delta}{\mu})  \E.
\end{equation}
According to Lemma~\ref{lemma:quasilin:gen}, $C$ depends on an upperbound $\Mb > 0$ such that
$$\Vert (V,\sqrt{\mu}w) \Vert_{H^{s_0+2}} + \Vert \frac{1}{\sqrt{\mu}} \partialphi{r} V \Vert_{H^{s_0+1}} + \vert \eta_0 \vert_{H^{s_0+2}} \leq \Mb,$$
as well as the constants from Assumptions~\eqref{hyp:b},~\eqref{hyp:H}. The first and last terms are bounded from above by $M$ by~\eqref{hyp:M}. The second term is bounded from above, up to a constant, by $M$ thanks to~\eqref{eqn:drV}. Hence the constant $C$ in~\eqref{eqn:temp:10} depends only on $M$. 
To conclude on the estimate on the low regularity part $\Elow$ in the definition~\eqref{eqn:energy:def} of $\E$, we need to apply an operator of the form $\Lambda^{s'-k}\partial_r^k$ to the system~\eqref{eqn:euler_phi}, with $1\leq k \leq s'$. This yields a system of the form~\eqref{eqn:euler_quasilin_dummy}, where
$$ \point{\eta}{s',k}_0 := \Lambda^{s'-k} \partial_r^k \eta_0 = 0,$$
as $\eta_0$ only depends on the horizontal variable $x$ and $k \geq 1.$ The estimates~\eqref{eqn:quasilin_dummy:restes} still hold and using~\eqref{eqn:quasilin:gen:energy:L2} we eventually get
$$ \frac{d}{dt} \Vert (V,\sqrt{\mu}w) \Vert_{H^{s_0+1}} \leq C(\epsilon \vee \beta \vee \frac{\delta}{\mu}) \E.$$
The estimate on the density $\rho$ using the equation~\eqref{eqn:euler_phi} is very similar and we omit it.\\

\noindent{\bf \underline{Step 2}: Control of the horizontal and vertical velocities:}
We now account for the term of main interest in the definition~\eqref{eqn:energy:def} of the energy, that is 
$$\Vert\sqrt{(\hb + \epsilon h)(\rhotot)}\pointal{V}{s},\sqrt{\mu(\hb + \epsilon h)(\rhotot)}\pointal{w}{s}, \sqrt{\mu(\hb + \epsilon h)} \rho\Vert_{L^2(S)}^2 + \vert \sqrt{\taylor} \point{\eta}{s}\vert_{L^2(\R^d)}^2.$$ Recall that applying $\dot{\Lambda}^{s}$ with $s \geq s_0+2$ to~\eqref{eqn:euler_phi} yields~\eqref{eqn:euler_quasilin}, where $\pointal{V}{s}, \pointal{w}{s}$ are defined through~\eqref{eqn:alinhac:def}. The system~\eqref{eqn:euler_quasilin} is of the form of~\eqref{eqn:quasilin:gen} with $(\pointgen{V},\pointgen{w},\pointgen{\rho},\pointgen{P},\pointgen{\eta}_0) := (\pointal{V}{s},\pointal{w}{s},\pointal{\rho}{s},\pointal{P}{s},\point{\eta}{s}_0)$. Note that because of the boundary condition on the pressure~\eqref{eqn:euler_phi:bc}, we can write:
$$\pointal{P}{s}_{|r=0} = \diff P_{|r=0} - \frac{\point{\eta}{s}_0}{\hb + \epsilon h} \partial_r P_{|r=0}.$$
Thus, the boundary condition on the pressure in~\eqref{eqn:quasilin:gen:bc} is satisfied, as
\begin{equation}
\label{eqn:quasilin:taylor}
\pointal{P}{s} + g \rhob\point{\eta}{s}_0 = (g\rhob - \frac{1}{\hb + \epsilon h} \partial_r P) \point{\eta}{s}_0 = \taylor \point{\eta}{s}_0,
\end{equation} 
where $\taylor$ is defined in~\eqref{eqn:taylor:def}, and corresponds to $\taylorgen $ in~\eqref{eqn:quasilin:gen:bc}. Then~\eqref{hyp:quasilin:taylor} is satisfied because of Lemma~\ref{lemma:taylor:dt} and~\eqref{hyp:taylor}. We can therefore use~\eqref{eqn:quasilin:gen:energy:L2} together with the estimate on the remainder terms~\eqref{eqn:restes:estimee} and on the pressure~\eqref{eqn:pression:estimee} to get
\begin{equation}
\label{eqn:temp:4}
\begin{aligned}
&\frac{d}{dt} \left(\Vert(\sqrt{(\hb + \epsilon h)(\rhotot)}\pointal{V}{s},\sqrt{\mu(\hb + \epsilon h)(\rhotot)}\pointal{w}{s}, \sqrt{\mu(\hb + \epsilon h)} \rho)\Vert_{L^2(S)} + \vert \sqrt{\taylor} \point{\eta}{s}\vert_{L^2(\R^d)}^2\right)\\
& \leq C (\epsilon \vee \beta \vee \frac{\delta}{\mu}) \E  + \vert \pointal{P}{s} \vert_{H^{-\frac12}} \vert \Rcinq \vert_{H^{\frac12}}. 
\end{aligned}
\end{equation}
It remains to estimate the last term in~\eqref{eqn:temp:4}. We write, thanks to the definition~\eqref{eqn:alinhac:def} of $\pointal{P}{s}$:
$$\vert \pointal{P}{s}_{|r=-1} \vert_{H^{-\frac12}} \leq \vert |D|\Lambda^{s-\frac32} P_{|r=-1} \vert_{L^2} + \vert \point{\eta}{s}_0 \partialphi{r} P_{|r=-1} \vert_{H^{-\frac12}}.$$
We use the product estimate~\eqref{apdx:eqn:pdt:s1s2} (with $s = - \frac12$, $s_1=-\frac12$, $s_2 = s_0+1$) for the second term, and the trace estimate~\eqref{apdx:eqn:trace} as well as the estimate~\eqref{eqn:pression:estimee} on the pressure to get
\begin{equation}
\label{eqn:temp:5}
\begin{aligned}
\vert \pointal{P}{s}_{|r=-1} \vert_{H^{-\frac12}} &\leq C \Vert \nabla P \Vert_{H^{s-1}}  + C \vert \eta_0 \vert_{H^{s-\frac12}} \Vert \partialphi{r} P \Vert_{H^{s_0+1}} \\
&\leq C \E^{\frac12}.
\end{aligned}
\end{equation}
Plugging~\eqref{eqn:temp:5} and the control on $\vert \Rcinq \vert_{H^{\frac12}}$ from~\eqref{eqn:restes:estimee} into~\eqref{eqn:temp:4} yield
$$\begin{aligned}
&\frac{d}{dt} \left(\Vert\sqrt{(\hb + \epsilon h)(\rhotot)}\pointal{V}{s},\sqrt{\mu(\hb + \epsilon h)(\rhotot)}\pointal{w}{s}, \sqrt{\mu(\hb + \epsilon h)} \rho\Vert_{L^2(S)}^2 + \vert \sqrt{\taylor} \point{\eta}{s})\vert_{L^2(\R^d)}^2\right) \\
&\leq C (\epsilon \vee \beta\vee\frac{\delta}{\mu}) \E.
\end{aligned}$$

\noindent{\bf \underline{Step 3}: Control of the vorticity and the density:}
For the estimate on $\vort$, take $0 \leq k-1 \leq s-1$ and define 
$$ \point{\vort}{s-1,k-1} := \Lambda^{s-k} \partial_r^{k-1} \vort.$$
Then $\point{\vort}{s-1,k-1}$ satisfies~\eqref{eqn:euler_quasilin:vort}, which is of the form of~\eqref{eqn:quasilin:vort:gen}. Using~\eqref{eqn:quasilin:gen:vort:energy:L2}, we thus get
$$\frac{d}{dt} \Vert \point{\vort}{s-1,k-1} \Vert_{L^2}^2 \leq \epsilon C \E.$$
Finally, summing the previous expression on $0 \leq k-1 \leq s-1$, we get
\begin{equation}
\label{eqn:vort:estimee}
\frac{d}{dt} \Vert \vort \Vert_{H^{s-1}}^2 \leq \epsilon C \E.
\end{equation}
The estimate on $\rho$ follows the same line: multiplying the whole quasilinear equation~\eqref{eqn:euler_quasilin:rho} by $\mu$, we get an equation of the form~\eqref{eqn:quasilin:vort:gen} with no singular term in $\mu$. This concludes the proof.
\end{proof}
\begin{remark}
\label{rk:triche}
The unknowns $V,w,\rho,\eta_0$ are in $H^s$, so that $\pointal{V}{s},\pointal{w}{s},\pointal{\rho}{s},\point{\eta}{s}_0,\point{\vort}{s-1,k-1}$ defined in Step 2 and Step 3 are in $L^2$, but not in $H^1$, and this is not enough regularity to apply Lemma~\ref{lemma:quasilin:gen}. This is solved by following the standard smoothing strategy used for instance in~\cite[Lemma 2.38]{Lannes2013}, namely applying smoothing operators, both in the horizontal direction $x$ and the vertical direction $r$. For the horizontal direction, one should replace $\Lambda^s$ by $\chi(\iota  \Lambda) \Lambda^s$ in the proof of Proposition~\ref{prop:energy}, with $\chi$ a cutoff and $\iota  > 0$. Using that fact that the flow is tangent to the boundaries, one can also define a smoothing operator in the vertical direction, see the operator $K^{\iota }(a \partial_z)$ in~\cite[Lemma 4.8]{Castro2014a}. With $\iota  > 0$, the unknowns are now sufficiently regular so that Lemma~\ref{lemma:quasilin:gen} can be applied, and the estimates are independent of $\iota $. Taking the limit $\iota  \to 0$ yields the estimate of Proposition~\ref{prop:energy}.
\end{remark}
\subsection{Existence scheme}
\label{subsection:scheme}
In this subsection we state and prove Theorem~\ref{thm:euler_phi}, which is the main result of this paper. Given the energy estimate from Proposition~\ref{prop:energy}, it only remains to construct a solution, from a sequence of suitable approximate solutions. Because of the presence of boundaries, and particularly  a free-surface, this is not completely standard.\\
Let $s_0, s\in \N$, with $s_0>\frac{d+1}{2}$, $s \geq s_0+2$. Let $(V_{\ini},w_{\ini},\rho_{\ini}, (\eta_0)_{\ini}) \in (H^{s})^{d+2}(S) \times H^s(\R^d)$, more precisely we assume that there exists $M_{\ini} > 0$ such that
\begin{hyp}
\label{hyp:Mini}
\Vert V_{\ini}, \sqrt{\mu} w_{\ini},\sqrt{\mu} \rho_{\ini}\Vert_{H^s(S_r)} + \vert (\eta_0)_{\ini}\vert_{H^s(\R^d)} + \Vert (\vort)_{\ini} \Vert_{H^{s-1}(S)} \leq M_{\ini},
\end{hyp}where $(\vort)_{\ini}$ is defined through~\eqref{eqn:vort:def}. We have the following theorem.
\begin{theorem}
\label{thm:euler_phi}
Let $\epsilon,\beta,\delta \in [0,1]$, $\mu \in (0,1]$. Let $M_{\ini} >0$ and $(V_{\ini},w_{\ini},\rho_{\ini}, (\eta_0)_{\ini}) \in (H^{s})^{d+1}(S) \times H^s(\R^d)$ satisfying~\eqref{hyp:b},~\eqref{hyp:H},~\eqref{hyp:taylor},~\eqref{hyp:Mini}, the first and third boundary conditions in~\eqref{eqn:euler_phi:bc} and the incompressibility condition in~\eqref{eqn:euler_phi}.
Then there exists $T > 0$ depending only on $d,s_0,s,M_{\ini}$ such that the following holds. There exists $T_{\max} > 0$ and a unique solution $(V,w,\rho,\eta_0)$ to~\eqref{eqn:euler_phi} in $C^0([0,T_{\max}),H^{s}(S)^{d+2} \times H^s(\R^d))$ satisfying the boundary conditions~\eqref{eqn:euler_phi:bc} and the initial condition~\eqref{eqn:euler_phi:ci}, and we have the lower bound
\begin{equation}
\label{eqn:T:lowerbound}
T_{\max} \geq \frac{T}{\epsilon \vee \beta \vee \frac{\delta}{\mu}}.
\end{equation}
Moreover, if $T_{\max}$ is finite, we can write
\begin{equation}
\label{eqn:critere_explosion}
\begin{aligned}
 &\limsup\limits_{t \to T_{\max}} \Vert (V, \sqrt{\mu} w,\sqrt{\mu} \rho)(t,\cdot)\Vert_{H^s(S_r)} + \vert \eta_0(t,\cdot)\vert_{H^s(\R^d)} + \Vert \vort(t,\cdot) \Vert_{H^{s-1}(S)} = \infty, \\
\text{or } \quad & \liminf\limits_{t \to T_{\max}} \inf\limits_{(x,r)\in S} \taylor = 0.\\
\end{aligned}
\end{equation}
\end{theorem}
\begin{remark}
The blow-up criterion~\eqref{eqn:critere_explosion} states that either the norm of the solutions blows up, or a hyperbolicity criterion (related to $\taylor$) fails when $t \to T_{\max}$. One might expect the alternatives $\liminf\limits_{t \to T_{\max}} \inf\limits_{(x,r)\in S} (1 - \beta b + \epsilon \eta_0) =0$ and $\liminf\limits_{t \to T_{\max}} \inf\limits_{(x,r)\in S} (\rhob + \epsilon \delta \rho) = 0$ to appear in the blow-up criterion, as they are needed to derive the energy estimates from Proposition~\ref{prop:energy}. We show however at the end of the following proof that as long as the solution $(V,w)$ remains bounded in Sobolev norms, then the water height $1 - \beta b + \epsilon \eta_0$ and the density $\rhob + \epsilon \delta \rho$ cannot go to $0$. 
\end{remark}
\begin{proof}
The proof adapts arguments from~\cite{Fradin2024} to the free surface case and is organized as follows. In Step $1$ we define some technical tools, in particular the semi-Lagrangian coordinates. These are used in Steps $2$ and $3$ to construct a solution to~\eqref{eqn:euler_phi} on a short time interval using the standard strategy from~\cite{AJMajda2002}, as the initial boundary value problem~\eqref{eqn:euler_phi} becomes an initial value problem, see~\eqref{eqn:euler:slag}. More precisely, in Step $2$ we mollify the system~\eqref{eqn:euler:slag} into~\eqref{eqn:moll:123} with three regularization parameters $\iota_1,\iota_2,\iota_3$ and in Step $3$ we show that the gradient of the pressure defined from the unknowns in~\eqref{eqn:moll:123} is Lipschitz continuous with respect to these unknowns, so that~\eqref{eqn:moll:123} is a system of ODEs. A standard compactness argument yields a solution to~\eqref{eqn:moll:123}. In Step $4$ we derive several energy estimates to send the regularization parameters $\iota_1,\iota_2,\iota_3$ to 0, in that order. The reason why we need several parameters is that the energy estimate of Proposition~\ref{prop:energy} does not hold for the regularized system~\eqref{eqn:moll:123}. We thus use other, less accurate energy estimates that allow us to send $\iota_1,\iota_2$ to $0$ in Steps $4.1,4.2$ and only then we use the energy estimate of Proposition~\ref{prop:energy} in Step $4.3$ to send $\iota_3$ to $0$. In Step $5$ we show the blow-up criterion~\eqref{eqn:critere_explosion}.\\

We start with the local existence of a solution to~\eqref{eqn:euler_phi}. At this stage, we do not keep track of the dependency in the parameters $\beta,\mu,\epsilon,\delta$, so that we set them to $1$ to simplify the notations.\\

\noindent{\bf \underline{Step 1:} Technical tools} \saut
We use a different coordinate system in Steps $2$ and $3$ to prove the local existence part, namely semi-Lagrangian coordinates. To this end, we start from the system~\eqref{eqn:euler_euleriennes} in Eulerian coordinates. Let $\eta_{\ini} : S \to \R$, be such that $(x,r) \mapsto r + \eta_{\ini}(x,r)$ is a diffeomorphism that sends $\Omega_0$ to $S$, where $\Omega_0$ is defined through~\eqref{eqn:def:Omegat}, for instance of the form~\eqref{eqn:phi:def:bary}. We now {\it define } $\eta: [0,T]\times S \to \R$ as a solution of
\begin{equation}
\label{eqn:eta:def}
\partial_t \eta + V \cdot \nabla \eta - w = 0 \qquad \text{ in } [0,T] \times S,
\end{equation}
with initial condition
\begin{equation}
\label{eqn:eta:ci}
\eta_{|t=0} = \eta_{\ini}.
\end{equation}
We define $\varphi$ from $\eta$ through 
\begin{equation}
\label{eqn:phi:def:slag}
\varphi : \begin{aligned} S &\to \Omega_t \\ (x,r) &\mapsto (x,(r + \eta)(t,x,r)) \end{aligned}
\end{equation}
and $\partialphi{t}$,$\partialphi{r}$,$\gradphi$ through~\eqref{eqn:gradphi:def}. Note that the free-surface $\eta_0$ is nothing but the trace of $\eta$ evaluated at $r=0$.\\
This coordinate system enjoys a semi-Lagrangian property, namely 
\begin{equation}
\label{eqn:slag}
\partialphi{t} + \vec{U} \cdot \gradphi[x,r] = \partial_t + V \cdot \nabla.
\end{equation}
We now write the Euler equations in this system of coordinates, coupled with~\eqref{eqn:eta:def}, namely
\begin{equation}
\label{eqn:euler:slag}
\sys{
\partial_t \vec{U} + V \cdot \nabla \vec{U} + \frac{1}{\rhob +  \rho} \gradphi[x,r]P + \frac{g\rhob}{\rhob + \rho}\nabla_{x,r} \eta_0 +  \frac{g\rho}{\rhotot} \vec{e_z}&= 0,\\
\partial_t \rho + V \cdot \nabla_x \rho &= 0,\\
\partial_t \eta + V \cdot \nabla \eta - w &= 0, \\
\gradphi[x,r] \cdot \vec{U} &= 0,}
\end{equation}
where $\vec{e_z} := (0,0,1)^T$ and recall that $\eta_0$ is independent of $r$ so that $\partial_r \eta_0 = 0$.
This is completed with the boundary conditions~\eqref{eqn:euler_euleriennes:bc} and the initial condition~\eqref{eqn:euler_euleriennes:ci} for $\vec{U}$ as well as~\eqref{eqn:eta:ci} for $\eta$.\\
Note that there are no vertical derivatives in~\eqref{eqn:euler:slag} (except for the incompressibility constraint and its associted Lagrange multiplier $P$, but we show how to treat these in Step $3$), so that the associated Cauchy problem is not an initial boundary value problem. Thanks to this remark, we only need to define mollification operators with respect to the horizontal coordinate. Following~\cite{AJMajda2002}, let $\mathcal{J}_{\iota}$ for any $\iota > 0$ the mollification operator defined by the product with a smooth cutoff function in the horizontal Fourier space (see ~\cite{AJMajda2002}), to regularize $\nabla_x$. It satisfies the following properties, see~\cite[Lemma 3.5]{AJMajda2002}.
\begin{lemma}
\label{lemma:moll}
Let $s,s' \in \R$, $f \in H^{s'}(\R^d)$. For $\iota > 0$, we can write
\begin{equation}
\label{eqn:moll:borne}
\vert \mathcal{J}_{\iota} f \vert_{H^s(\R^d)} \leq C_{\iota,s,s'} \vert f \vert_{H^{s'}(\R^d)},
\end{equation}
for some constant $ C_{\iota,s,s'}$.\\
For $\iota' > 0$, we also have 
\begin{equation}
\label{eqn:moll:lip}
\vert (\mathcal{J}_{\iota} - \mathcal{J}_{\iota'})f\vert_{H^{s'-1}} \leq C |\iota - \iota'| \vert f \vert_{H^{s'}},
\end{equation}
where $C$ does not depend on $\iota$ nor $\iota'$.
\end{lemma}

\noindent{\bf \underline{Step 2:} The mollified system} \saut
We now introduce the following mollified system
\begin{equation}
\tag{\ensuremath{E_{\iota}}}
\label{eqn:moll:123}
\sys{ \partial_t \vec{U} + \mathcal{J}_{\iota_1} \left[ \mathcal{J}_{\iota_2}V \cdot \nabla_x \mathcal{J}_{\iota_1} \vec{U} \right]+ \frac{1}{\rhob + \rho}\gradphi[x,r] P  + \frac{g\rhob}{\rhob + \rho} \mathcal{J}_{\iota_2} \nabla_{x,r} \eta_0 + g \frac{\rho}{\rhob + \rho} \vec{e_z}& = \frac{\iota_3}{\rhob + \rho} \mathcal{J}_{\iota_2} \gradphi[x,r] \Lambda \eta_0^h, \\
\partial_t \rho  + \mathcal{J}_{\iota_1} \left[ \mathcal{J}_{\iota_2}V \cdot \nabla_x \mathcal{J}_{\iota_1} \rho \right] &= 0,\\
		\partial_t \eta + \mathcal{J}_{\iota_1} \left[ \mathcal{J}_{\iota_2}V \cdot \nabla_x \mathcal{J}_{\iota_1} \eta \right] - \reg{2}w &= 0,}
\end{equation}
where $\eta_0^h$ is the harmonic extension of $\eta_0$ to the strip $S$. The benefit from taking a harmonic extension is the following property (see for instance~\cite[Prop. 2.36]{Lannes2013}):
\begin{lemma}
If $\eta_0 \in H^{s}(\R^d)$ for some $s \geq s_0+1$, then $\Lambda^{s-\frac12} \nabla_{x,z} \eta_0^h \in L^2(S)$.
\end{lemma}This set of equations is completed with the initial conditions
\begin{equation}
\label{eqn:moll:CI}\vec{U}_{|t=0} = \vec{U}_{\ini}, \qquad 
	\rho_{|t=0} = \rho_{\ini}, \qquad 
		\eta_{t=0} = \eta_{\ini}.
\end{equation}
We define the pressure $P$ through the elliptic equation
\begin{equation}
\label{eqn:existence:elliptic}
\gradphi[x,r] \cdot \frac{1}{\rhob + \rho}\gradphi[x,r] P = - \partial_i^{\varphi} \reg{2}\vec{U}_j  \partial_j^{\varphi} \vec{U}_i - \reg{2} \Delta_x \eta_0 - \partialphi{r} \frac{\rho}{\rhob+\rho} + \iota_3 \gradphi[x,r] \cdot \reg{2} \gradphi[x,r] \Lambda \eta_0^h,
\end{equation}
completed with the boundary condition at the top
\begin{equation}
\label{eqn:existence:elliptic:bc:top}
P = 0 \qquad \text{ at } r=0,
\end{equation}
and at the bottom
\begin{equation}
\label{eqn:existence:elliptic:bc:bot}
\vec{N_b} \cdot \gradphi[x,r] P = - \reg{2} \vec{N_b} \cdot (V \cdot \nabla_x \vec{U})  - \frac{\rho}{\rhob+\rho} + \iota_3 \vec{N_b} \cdot \reg{2} \gradphi[x,r] \Lambda \eta_0^h \qquad \text{ at } r=-1,
\end{equation}
where $\vec{N}_b$ given by~\eqref{eqn:def:Nb} is the upward normal at the bottom.

\noindent{\bf \underline{Step 3:} $(\vec{U},\eta) \mapsto \gradphi[x,r] P $ is well-defined and Lipschitz continuous} \saut
Recall that $s \geq s_0+2$. The following lemma states that $(\vec{U},\eta) \mapsto \gradphi[x,r] P $ is well-defined and a Lipschitz continuous operator from $H^s(S)$ to itself, although not uniformly in the regularization parameters $\iota_2$ and $\iota_3$.
\begin{lemma}
\label{lemma:moll:pression}
Under the assumptions and notations of Theorem~\ref{thm:euler_phi},~\eqref{eqn:existence:elliptic} completed with~\eqref{eqn:existence:elliptic:bc:top} and~\eqref{eqn:existence:elliptic:bc:bot} has a unique solution $P$ and there exists $C_{\iota_2,\iota_3}>0$ depending only on $M$ and $\iota_2,\iota_3$ such that
\begin{equation}
\label{eqn:existence:elliptic:borne}
\Vert P \Vert_{L^2} + \Vert \gradphi[x,r] P\Vert_{H^s} \leq C_{\iota_2,\iota_3} \Vert (\vec{U},\rho,\eta)\Vert_{H^s}.
\end{equation}
We also have the bound, with $C$ independent of $\iota_2$ and $\iota_3$:
\begin{equation}
\label{eqn:existence:elliptic:borne:unif}
\Vert P \Vert_{L^2} + \Vert \gradphi[x,r] P\Vert_{H^s} \leq C \Vert (\vec{U},\eta)\Vert_{H^{s+1}}+ C \iota_3 \vert \eta_0\vert_{H^{s+\frac32}}.
\end{equation}
Moreover, if $P^{(1)}$ and $P^{(2)}$ are defined from $(\vec{U}^{(1)},\rho^{(1)},\eta^{(1)})$ and $(\vec{U}^{(2)},\rho^{(2)},\eta^{(2)})$ through~\eqref{eqn:existence:elliptic}, then
\begin{equation}
\label{eqn:existence:elliptic:lip}
\Vert P^{(1)} - P^{(2)} \Vert_{L^2} + \Vert \gradphi[x,r] (P^{(1)} - P^{(2)}) \Vert_{H^s} \leq C_{\iota_2,\iota_3} \Vert (\vec{U}^{(1)}-\vec{U}^{(2)},\rho^{(1)}-\rho^{(2)},\eta^{(1)}-\eta^{(2)})\Vert_{H^s}.
\end{equation}
\end{lemma}
\begin{proof}
Using standard elliptic estimates (see for instance~\cite[Lemma 2.38]{Lannes2013} or~\cite[Lemma 4]{Desjardins2019}), we get
\begin{equation}
\label{eqn:pression:lowreg}
\Vert P\Vert_{L^2} + \Vert \gradphi[x,r] P\Vert_{H^{s_0+1}} \leq C \Vert (\vec{U},\rho,\eta)\Vert_{H^{s_0+2}}.
\end{equation}
Using~\cite[Proposition 2.6]{Fradin2024} and~\eqref{eqn:pression:lowreg}, we get~\eqref{eqn:existence:elliptic:borne}.\\
The strategy to obtain~\eqref{eqn:existence:elliptic:lip} relies on Alinhac's good unknown and~\cite[Proposition 2.6]{Fradin2024}. This is detailed in~\cite[Step 1 of the proof of Prop. 4.1]{Fradin2024} and we omit it.
\end{proof}
Therefore, the system~\eqref{eqn:moll:123} has an ODE structure with a Lipschitz source term, thanks to Lemma~\ref{lemma:moll:pression}, and as such, admits a unique solution on a time interval $[0,T_{\iota_1,\iota_2,\iota_3}]$ with respect to the initial data given by~\eqref{eqn:moll:CI}. \\

\noindent{\bf\underline{Step 4.1:} Sending $\iota_1$ to $0$} \saut
The time of existence $T_{\iota_1,\iota_2,\iota_3}$ given in Step $3$ is not independent of the regularization parameters. A standard energy estimate shows that $T_{\iota_1,\iota_2,\iota_3}$ is bounded from below by a constant independent of $\iota_1$, see for instance~\cite[Section 3.2]{AJMajda2002}.\\
\begin{remark}
\label{rk:iota1}
The energy estimate from Proposition~\ref{prop:energy} does not apply here: the regularization operators $\reg{1}$ forbid to use the identity~\eqref{eqn:slag}.  In particular, the divergence-free condition is not propagated and the use of Alinhac's good unknown is not possible. We thus use a more standard energy estimate to show that $t_{\iota_1,\iota_2,\iota_3}$ is independent of $\iota_1$. This energy estimate is, however, not enough to conclude the independence with respect to the other regularization parameters.
\end{remark}
\noindent{\bf \underline{Step 4.2:} Incompressibility and sending $\iota_2$ to $0$} \saut
We now have a solution to the system~\eqref{eqn:moll:123} with $\iota_1 = 0$. 
Thanks to the equation on $\eta$ in~\eqref{eqn:moll:123} with $\iota_1=0$, we can write~\eqref{eqn:moll:123} with $\iota_1=0$ and without the semi-Lagrangian property~\eqref{eqn:slag}; this reads
\begin{equation}
\label{eqn:moll:23}
\sys{ \partialphi{t} \vec{U} + \mathcal{J}_{\iota_2} V \cdot \gradphi \vec{U} + \reg{2}w \partialphi{r} \vec{U} + \frac{1}{\rhob + \rho}\gradphi[x,r] P  + \frac{g\rhob}{\rhob + \rho} \mathcal{J}_{\iota_2} \nabla_{x,r} \eta_0 + g \frac{\rho}{\rhob + \rho} \vec{e_z}& = \frac{\iota_3}{\rhob + \rho} \mathcal{J}_{\iota_2} \gradphi[x,r] \Lambda \eta_0^h, \\
\partialphi{t} \rho  + \mathcal{J}_{\iota_2}V \cdot \gradphi \rho  + \reg{2}w \partialphi{r} \rho &= 0,\\
\partial_t \eta + \reg{2} V \cdot \nabla_x \eta - \reg{2} w &= 0,}
\end{equation}
with the kinematic equation for the surface
\begin{equation}
\label{eqn:moll:23:eta}
\partial_t \eta_0 + \mathcal{J}_{\iota_2}V_{|r=0} \cdot \nabla_x \eta_0  - \reg{2}w_{|r=0} = 0.
\end{equation}
Applying $\gradphi[x,r] \cdot$ to the equations on $V$ and $w$ in~\eqref{eqn:moll:23} and using the equation~\eqref{eqn:existence:elliptic} satisfied by the pressure, we can see that $\gradphi[x,r] \cdot \vec{U}$ satisfies a transport equation at velocity $\reg{2} \vec{U}$. Thus, as the initial condition $(V_{\ini},w_{\ini})^T$ is divergence-free, so is the solution to~\eqref{eqn:moll:23} that we constructed in Step 4.1.\\

We now perform an energy estimate similar to Proposition~\ref{prop:energy}, in order to show that $T_{\iota_2,\iota_3}$ is independent of $\iota_2$. More precisely, we apply $\dot{\Lambda}^{s}$ with $s\geq s_0+2$ to~\eqref{eqn:moll:23} and get the following equations in $S$
\begin{equation}
\label{eqn:moll:23:quasilin}
\sys{ \partialphi{t} \pointal{\vec{U}}{s} + \mathcal{J}_{\iota_2} V \cdot \gradphi \pointal{\vec{U}}{s} + w \partialphi{r} \pointal{\vec{U}}{s} + \frac{1}{\rhob + \rho}\gradphi[x,r] \pointal{P}{s}  + g \rhob \mathcal{J}_{\iota_2} \nabla_{x,r} \pointal{\eta_0}{s} & \approx \frac{\iota_3}{\rhob + \rho} \mathcal{J}_{\iota_2} \gradphi[x,r] \Lambda (\pointal{\eta_0}{s})^h, \\
\partialphi{t} \pointal{\rho}{s}  + \mathcal{J}_{\iota_2}V \cdot \gradphi \pointal{\rho}{s} + w \partialphi{r} \pointal{\rho}{s} &\approx 0,\\
\partial_t \pointal{\eta}{s} + \reg{2} V \cdot \nabla_x \pointal{\eta}{s} & \approx 0,}
\end{equation}
\begin{equation}
\label{eqn:moll:23:quasilin:eta0}
\partial_t \pointal{\eta_0}{s} + \reg{2} V_{|r=0} \cdot \nabla \pointal{\eta_0}{s} + \reg{2} \pointal{V}{s}_{|r=0} \cdot \nabla \eta_0 - \reg{2} \pointal{w}{s}_{|r=0} \approx 0,
\end{equation}
where the symbol $\approx$ means that remainder terms are omitted. Those remainder terms are of the form of the ones in~\eqref{eqn:euler_quasilin:restes} with additional but harmless terms stemming from commutators with $\reg{2}$ and the dispersive regularization $\iota_3 \gradphi[x,r] \Lambda \eta_0^h$. Note that we use Alinhac's good unknown here. We now define the following energy functional $\tilde{\mathcal{E}}_s$ defined as
\begin{equation}
\label{eqn:nrj:def:moll}
 \left\Vert\left(\sqrt{(\hb +  h)(\rhob + \rho)}\pointal{V}{s},\sqrt{ (\hb + h)(\rhob +\rho)}\pointal{w}{s}, \sqrt{(\hb + h)} \pointal{\rho}{s}\right)\right\Vert_{L^2(S)}^2 + \vert (g\rhob+\iota_3 \Lambda^{\frac12}) \pointal{\eta}{s}_0\vert_{L^2(\R^d)}^2 + \Vert \vec{\omega} \Vert_{H^{s-1}(S)}^2.
\end{equation}
Note that because of the additional dispersion term $\frac{\iota_3}{\rhob + \rho} \mathcal{J}_{\iota_2} \gradphi[x,r] \Lambda (\pointal{\eta_0}{s})^h$~\eqref{eqn:moll:23:quasilin}, we control an extra half-derivative on $\eta_0$ compared to Proposition~\ref{prop:energy}.
Following the same steps as the proofs of Propositions~\ref{prop:energy} and~\ref{prop:restes}, we can write
\begin{equation}
\label{eqn:nrj:estimee:moll}
\begin{aligned}
\frac{d}{dt} \tilde{\mathcal{E}}_s &+ \epsilon \int_{\R^d} \point{\eta_0}{s} (\partialphi{r} P)_{|r=0} \cdot \pointal{\vec{U}}{s}_{|r=0} + \int_S \reg{2} \gradphi[x,r] (g\rhob \pointal{\eta_0}{s} + \iota_3\Lambda (\pointal{\eta_0}{s})^h) \cdot \pointal{\vec{U}}{s}\\
&  + \int_{\R^d}(\reg{2} \pointal{V}{s}_{|r=0} \cdot \nabla \eta_0 - \reg{2} w) g \rhob \pointal{\eta_0}{s} \approx 0.
\end{aligned}
\end{equation}
The first integral comes from the integration by parts on the pressure term, see~\eqref{eqn:pression:ipp}. We bound this term from above, up to a constant, by
$$\vert \point{\eta_0}{s} \vert_{H^{\frac12}(\R^d)} \vert \pointal{\vec{U}}{s}_{|r=0} \vert_{H^{-\frac12}(\R^d)},$$
which in turn is bounded by $\tilde{\mathcal{E}}$, thanks to the trace estimate~\eqref{apdx:eqn:trace}.
Now using the symmetry of $\reg{2}$ for the $L^2-$scalar product in the second integral and integrating by parts, we get, similarly to~\eqref{eqn:energy:V:1}(v):
$$\frac{d}{dt} \tilde{\mathcal{E}} \approx 0.$$
This shows that the time of existence of the solutions of~\eqref{eqn:moll:23} $T_{\iota_2,\iota_3}$ is independent of $\iota_2$.\\
\begin{remark}
\label{rk:iota2}
The only difference between this energy estimate and the one in Proposition~\ref{prop:energy} is the second term in~\eqref{eqn:nrj:estimee:moll}. Because it involves $\pointal{\vec{U}}{s}$ and not $\reg{2}\pointal{\vec{U}}{s}$ (like the third term in~\eqref{eqn:nrj:estimee:moll}), it cannot be treated via a symmetry argument (as in Proposition~\ref{prop:energy}). We rather use the dispersion term that controls an extra half derivative on $\eta_0$.
\end{remark}
\noindent{\bf \underline{Step 4.3:} Sending $\iota_3$ to $0$} \saut
We now have the existence of a solution to the system~\eqref{eqn:moll:23} with $\iota_2 = 0$. We can now apply the energy estimate of Proposition~\ref{prop:energy}, in particular using the Rayleigh-Taylor coefficient as in~\eqref{eqn:pression:ipp}, with the additional dispersion term $\iota_3 \Lambda^{\frac12} \point{\eta_0}{s}$ as in~\eqref{eqn:nrj:def:moll}. Contrary to Step $4.2$, we do not use the regularizing property of the dispersion (except to control terms stemming from the dispersion itself, which are of size $\iota_3$), hence the time of existence $T_{\iota_3}$ is independent of $\iota_3$. We eventually have existence of solutions to the system~\eqref{eqn:euler_phi}. Finally, changing variables, this yields a solution of the original system~\eqref{eqn:euler_euleriennes} in Eulerian coordinates, or equivalently,~\eqref{eqn:euler_phi}. This same energy estimate also provides the dependency with respect to $\epsilon,\beta,\mu,\delta$ of the time of existence in Theorem~\ref{thm:euler_phi} and the uniqueness of the solution.\\

\noindent{\bf \underline{Step 5:} Blow-up criterion} \saut
Classically (see~\cite{AJMajda2002}), the solutions can be continued as long as the assumptions of the energy estimate of Proposition~\ref{prop:energy} are satisfied, and we denote by $T_{\max}$ the maximal time of existence of a solution. This yields the following blow-up criterion
\begin{equation}
\label{eqn:critere_explosion:temp}
\text{or } \quad
\begin{aligned}
 &\limsup\limits_{t \to T_{\max}} \Vert (V, \sqrt{\mu} w,\sqrt{\mu} \rho)(t,\cdot)\Vert_{H^s(S_r)} + \vert \eta_0(t,\cdot)\vert_{H^s(\R^d)} + \Vert \vort(t,\cdot) \Vert_{H^{s-1}(S)}= + \infty \\
& \liminf\limits_{t \to T_{\max}}\inf\limits_{(x,r)\in S} \taylor = 0,\\
& \liminf\limits_{t \to T_{\max}}\inf\limits_{(x,r)\in S} (\rhob + \epsilon \delta \rho) = 0,\\
& \liminf\limits_{t \to T_{\max}}\inf\limits_{x\in \R^d} 1 - \beta b(x) + \epsilon \eta_0(t,x) = 0.\\
\end{aligned}
\end{equation}
We now show that the last two conditions cannot happen. The condition on the density cannot happen as $\rho$ is transported by the fluid, therefore if~\eqref{hyp:density} holds initially, it holds as long as $V$ and $w$ remain Lipschitz continuous. for the last alternative in~\eqref{eqn:critere_explosion:temp}, we can compute, thanks to the kinematic equation~\eqref{eqn:euler_phi:bc} on $\eta_0$ and the divergence-free condition in~\eqref{eqn:euler_phi}:
$$\begin{aligned}
\frac{d}{dt} (1 - \beta b(x) + \epsilon \eta_0(t,x)) &= \epsilon \partial_t \eta_0(t,x) = -V(t,x,0) \cdot \nabla \eta_0(t,x) + w(t,x,0) \\
&= \int_{-1}^0 \partial_r (-V(t,x,r) \cdot \nabla (\etab(x,r) + \epsilon \eta(t,x,r)) + w(t,x,r))dr \\
&= \int_{-1}^0 \left( -\nabla(\etab + \epsilon \eta)\partial_r V + \partial_r w - V \nabla (\hb + \epsilon h) \right)\\
&= -\int_{-1}^0 \nabla \cdot ((\hb + \epsilon h)V)\\
\partial_t (1 - \beta b + \epsilon \eta_0)&= -\nabla \cdot ((1 - \beta b + \epsilon \eta_0) \overline{V}),
\end{aligned} $$
where 
$$\overline{V}(t,x) := \frac{1}{1 - \beta b(x) + \epsilon \eta_0(t,x)} \int_{-1}^0 (\hb(x,r) + \epsilon h(t,x,r)) V(t,x,r)dr$$
is the vertically averaged horizontal velocity. This is the conservation of the fluid height, and so if initially the fluid height is bounded from below (see Assumption~\eqref{hyp:non_cavitation:euleriennes}), then it holds as long as $\overline{V}$ remains Lipschitz continuous, by the characteristics method, and the blow-up criterion~\eqref{eqn:critere_explosion:temp} boils down to~\eqref{eqn:critere_explosion}. \\
\end{proof}
\section{The shallow water limit}
\label{section:cvce}
In this section, we study the shallow water limit, namely $\mu \to 0$, in the system~\eqref{eqn:euler_phi}. We show in Subsection~\ref{subsection:cvce_short} that the solution of~\eqref{eqn:euler_phi} converges towards the solution of the non-linear shallow water equations~\eqref{eqn:nlsw}. In Subsection~\ref{subsection:cheating}, we use this result to improve the time of existence of Theorem~\ref{thm:euler_phi}, although in a specific setting and only to a time of order $\log(\frac{1}{\epsilon})$.
\subsection{Quantitative comparison with the non-linear shallow water equations}
\label{subsection:cvce_short}
Recall the non-linear shallow water equations~\eqref{eqn:nlsw}. In order to compare a solution $(V_{\sw},\eta_{\sw})$ of this system defined on $\R^d$ to a solution of~\eqref{eqn:euler_phi} defined in $S$, we identify $V_{\sw}$ to its natural extension $(t,x,r) \mapsto V_{\sw}(t,x)$ defined in $S$. Moreover, we define a shallow water vertical velocity $w_{\sw}$ as
\begin{equation}
\label{eqn:def_wsw}
w_{\sw} := \beta V_{\sw} \cdot \nabla b - (r + 1)(1-\beta b + \epsilon \eta_{\sw}) \nabla \cdot V_{\sw},
\end{equation}
so that the incompressibility condition in~\eqref{eqn:euler_phi} and the boundary condition~\eqref{eqn:euler_phi:bc} at the bottom and at the top are also satisfied by $(V_{\sw},w_{\sw})$. Defining $P_{\sw} := 0$, we can thus write the non-linear shallow water equations~\eqref{eqn:nlsw} as
\begin{equation}
\label{eqn:nlsw:euler_form}
\sys{\partialphi{t} V_{\sw} + \epsilon \vec{U}_{\sw} \cdot \gradphi[x,r] V_{\sw} + \frac{1}{\rhob} \gradphi[x] P_{\sw} + \nabla_x \eta_{\sw} &= 0 \\
\mu \left( \partialphi{t} w_{\sw} + \epsilon \vec{U} \cdot \gradphi[x,r] w_{\sw} \right) + \frac{1}{\rhob} \partialphi{r} P_{\sw} &= \mu f_2,\\
\nabla_x \cdot V_{\sw} + \partial_z w_{\sw} &= 0,} \qquad \text{ in } [0,T] \times \R^d \times [-1,0],
\end{equation}
with $f_2 := \partialphi{t} w_{\sw} + \epsilon \vec{U} \cdot \gradphi[x,r] w_{\sw}$ so that the second equation in~\eqref{eqn:nlsw:euler_form} is trivial. We still write it to stress the comparison with the Euler equations~\eqref{eqn:euler_phi}. Note that the analogous of the equation on the density $\rho$ in~\eqref{eqn:euler_phi} still holds if we define $\rho_{\sw} = 0$. The equations~\eqref{eqn:nlsw:euler_form} are completed with the boundary conditions
\begin{equation}
\label{eqn:nlsw:euler_form:bc}
\sys{ P_{\sw|r=0} &= 0,\\
	\partial_t \eta_{\sw} + \epsilon V_{\sw | r=0} \cdot \nabla_x \eta_{\sw} - w_{\sw|r=0} &= 0,\\
	\vec{N}_b \cdot \vec{U}_{\sw |r=-1} &= 0,\\}
\end{equation}
coming from~\eqref{eqn:nlsw} and the definition of $w_{\sw}$ and $P_{\sw}$. They are also completed with initial conditions analogous to~\eqref{eqn:euler_phi:ci}.

Let us now state the assumptions for the following theorem to hold true. Let $s_0,s \in \N$, with $s_0>\frac{d+1}{2}$ and $s \geq s_0+2$. Let $(V,w,\eta_0) \in C^0([0,T),(H^s)^{d+1}(S) \times H^s(\R^d))$ be a solution of~\eqref{eqn:euler_phi} and $(V_{\sw},\eta_{\sw}) \in C^0([0,T[,H^{s+1}(\R^d)^{d+1})$ be a solution of~\eqref{eqn:nlsw}. We assume that the two solutions are close initially, that is
\begin{hyp}
\label{hyp:E0}
\Vert (V - V_{\sw}, \sqrt{\mu}(w - w_{\sw})) \Vert_{H^s} + \vert \eta - \eta_{\sw} \vert_{H^s} + \Vert \sqrt{\mu} \rho \Vert_{H^s} + \frac{1}{\sqrt{\mu}}  \Vert \partial_r V \Vert_{H^{s-1}} \leq  \sqrt{\mu} \qquad \text{ at } t=0,
\end{hyp}We also write $M_{\sw}>0$ a constant such that
\begin{hyp}
\label{hyp:Msw}
\vert (V_{\sw},\eta_{\sw})\vert_{H^{s+1}} \leq M_{\sw}.
\end{hyp}We also assume weak density perturbations , that is
\begin{hyp}
\label{hyp:weak_density}
\delta \leq \mu.
\end{hyp}
\begin{theorem}
\label{thm:convergence:svt}
Let $T > 0$ such that there exist a solution $(V_{\sw},\eta_{\sw})$ of \eqref{eqn:nlsw} in \\$C^0([0,T], H^{s+1}(\R^d)^{d+1})$ and a solution $(V,w,\rho,\eta_0)$ of \eqref{eqn:euler_phi} in $C^0([0,T[, H^{s}(S)^{d+2} \times H^s(\R^d))$ satisfying \eqref{hyp:b}, \eqref{hyp:H}, \eqref{hyp:taylor}, \eqref{hyp:M}, \eqref{hyp:Mini}, \eqref{hyp:E0}, \eqref{hyp:Msw}, \eqref{hyp:weak_density}. Then there exists a constant $C$ depending only on $M$ (defined in \eqref{hyp:M}) and $M_{\sw}$ such that
$$\sup_{t\in[0,T]} \left( \Vert V - V_{\sw} \Vert_{H^{s}(S)} + \vert \eta_0-\eta_{\sw} \vert_{H^s(\R^d)} \right) \leq C\sqrt{\mu}.$$
\end{theorem}
\begin{remark}
The constant $C$ depends on $M_{\sw}$, an upperbound on $\vert (V_{\sw},\eta_{\sw})\vert_{H^{s+1}}$, which is a stronger norm than the one used in Theorem~\ref{thm:convergence:svt} for the convergence (i.e. we only control the difference of the solutions of~\eqref{eqn:euler_phi} and~\eqref{eqn:nlsw} in $H^s$). However, it only depends on an upperbound on the $H^s$-norm of the solution $(V,w,\rho,\eta_0)$ of~\eqref{eqn:euler_phi}, through~\eqref{hyp:M}. This is achieved by plugging the solution of the non-linear shallow water equations into the Euler equations~\eqref{eqn:euler_phi}, not the other way around. This fact will be used in Subsection~\ref{subsection:cheating}.
\end{remark}
\begin{remark}
\label{rk:euler_diff:estimate}
The proof of Theorem~\ref{thm:convergence:svt} yields the slightly more precise estimate, for any $t \in [0,T]$:
\begin{equation}
\label{eqn:euler_diff:estimate}
\begin{aligned}
&\Vert (V - V_{\sw}, \sqrt{\mu}(w - w_{\sw}))(t,\cdot) \Vert_{H^s} + \vert (\eta - \eta_{\sw})(t,\cdot) \vert_{H^s} + \Vert \sqrt{\mu} \rho \Vert_{H^s}+\frac{1}{\sqrt{\mu}} \Vert \partial_r V(t,\cdot) \Vert_{H^{s-1}} \\
&\leq C \sqrt{\mu} (1+t) e^{C (\epsilon \vee \beta\vee\frac{\delta}{\mu}) t},
\end{aligned}
\end{equation}
using~\eqref{hyp:E0}. This will be used in Subsection~\ref{subsection:cheating}.

Note also that the control of the difference of the solutions given by~\eqref{eqn:euler_diff:estimate} is in $\sqrt{\mu}$. This mostly comes from the fact that the columnar motion is only true up to order $O(\sqrt{\mu})$, see Remark~\ref{rk:columnar}. This is consistent with the result in~\cite{Castro2014a}. Note however that~\eqref{eqn:euler_diff:estimate} can be improved to a rate $O(\mu)$ in the homogeneous and irrotational case, see for instance~\cite{Lannes2013}.
\end{remark}
\begin{proof}
With the notations and assumptions of Theorem~\ref{thm:convergence:svt}, let 
\begin{equation}
\label{eqn:def:Vdiff}
(\Vdiff,\wdiff,\etadiff) := (V-V_{\sw},w-w_{\sw},\eta_0-\eta_{\sw}).
\end{equation}
We look for an equation of the form~\eqref{eqn:quasilin:gen} satisfied by $(\Vdiff,\wdiff,\rho,\etadiff)$. Taking the difference between~\eqref{eqn:euler_phi} and~\eqref{eqn:nlsw:euler_form} yields
\begin{equation}
\label{eqn:euler_diff:ordre0}
\sys{
\partialphi{t} \Vdiff + \epsilon \vec{U} \cdot \gradphi[x,r] \Vdiff + \frac{1}{\rhob + \epsilon \delta \rho} \gradphi \Pdiff + \frac{g\rhob}{\rhob + \epsilon \delta \rho} \nabla \etadiff &= \Fun \\
\mu\left( \partialphi{t} \wdiff + \epsilon \vec{U} \cdot \gradphi[x,r] \wdiff \right) + \frac{1}{\rhob + \epsilon \delta \rho}\partialphi{r}\Pdiff + \frac{\delta \rho}{\rhotot} &= \mu\Fdeux, \\
\partialphi{t} \rho + \vec{U} \cdot \gradphi[x,r] \rho &= 0,\\
\gradphi \cdot \Vdiff + \partialphi{r} \wdiff &= \Ftrois,
}
\end{equation}
where we used that $V_{\sw}$ does not depend on $r$. This system is completed with boundary conditions
\begin{equation}
\label{eqn:euler_phi:order0:bc}
\sys{ \Pdiff_{r=0} &= 0, \\
	\partial_t \etadiff + \epsilon V_{|r=0} \cdot \nabla \etadiff - \wdiff_{|r=0} &= \Fquatre, \\
	\vec{N_b} \cdot \vec{\tilde{U}}_{|r=-1} &= 0;}
\end{equation}
where we have used the following notations:
\begin{equation}
\label{eqn:euler_diff:ordre0:restes}
\begin{aligned}
\Fun &:=\epsilon \Vdiff \cdot \nabla V_{\sw} + \delta \epsilon \frac{\rho}{\rhotot} \nabla \eta_{\sw},\\
\Fdeux&:= \partialphi{t} w_{\sw} + \epsilon \vec{U} \cdot \gradphi[x,r] w_{\sw},\\
\Ftrois&:= \epsilon \frac{(1+r)\etadiff}{(\hb + \epsilon h)(\hb + \epsilon h_{\sw})} \partial_r w_{\sw},\\
\Fquatre&:=\epsilon \Vdiff \cdot \nabla \eta_{\sw}.
\end{aligned}
\end{equation}
We also define a shallow water vorticity :
$$\Vortsw := \begin{pmatrix} 0 \\ \vortsw \end{pmatrix} := \begin{pmatrix} 0 \\ \nabla^{\perp} \cdot V_{\sw} \end{pmatrix},$$
as well as the difference of the vorticities
$$\vortdiff := \vort - \Vortsw.$$
Applying $\nabla^{\perp}$ to the equation on $V_{\sw}$ in~\eqref{eqn:nlsw} and taking the difference with the equation~\eqref{eqn:euler_phi:vort} on $\vort$, we get
\begin{equation}
\label{eqn:euler_diff:vort}
\partialphi{t} \vortdiff + \epsilon \vec{U} \cdot \gradphi[x,r] \vortdiff  = \Fcinq.
\end{equation}	
where
\begin{equation}
\label{eqn:vort:diff:reste}
\Fcinq := \epsilon \Vdiff \cdot \nabla \Vortsw + \epsilon \vortsw \partialphi{r} \begin{pmatrix} \frac{1}{\sqrt{\mu}} \Vdiff \\ \wdiff \end{pmatrix} - \epsilon \vortdiff^{[x]} \cdot \gradphi \begin{pmatrix} V \\ \sqrt{\mu} w \end{pmatrix} - \epsilon \vortdiff^{[z]} \partialphi{r} \begin{pmatrix} \frac{1}{\sqrt{\mu}} V \\ w \end{pmatrix} + \frac{\delta}{\sqrt{\mu}} \vec{F},
\end{equation}
where $\vec{F}$ is defined from $\rho,P,\eta_0$ through~\eqref{eqn:def:vort:source}.
We used that the horizontal component of $\Vortsw$ is zero and that $\vortsw$ does not depend on $r$ to simplify the above equation.\\

As the equation on $\rho$ in~\eqref{eqn:euler_diff:ordre0} is the same as the one in~\eqref{eqn:euler_phi}, the energy estimate is the same as in the proof of Proposition~\ref{prop:energy} and we omit it in this proof. Let us define
\begin{equation}
\label{eqn:def:Vdiffs}
(\Vdiffs,\wdiffs,\etadiffs) := (\pointal{V}{s} - \point{V}{s}_{\sw},\pointal{w}{s} - \point{w}{s}_{\sw},\point{\eta}{s}_0 - \point{\eta}{s}_{\sw}), 
\end{equation}
where  $(\point{V}{s}_{\sw},\point{\eta}{s}_{\sw}) := (\diff V_{\sw},\point{\eta}{s}_{\sw})$, $\point{\eta}{s}_0 := \dot{\Lambda}^{s}\eta_0$ and $\pointal{V}{s},\pointal{w}{s}$ are Alinhac's good unknowns defined through~\eqref{apdx:alinhac:definition}. We look for a system of equations on $(\Vdiffs,\wdiffs,\etadiffs)$.
To this end, we apply the operator $\diff$ to~\eqref{eqn:nlsw:euler_form} and~\eqref{eqn:nlsw:euler_form:bc} to get
\begin{equation}
\label{eqn:nlsw:quasilin}
\sys{	\partial_t \point{V}{s}_{\sw} + \epsilon V_{\sw} \cdot \nabla \point{V}{s}_{\sw} + g\nabla \point{\eta}{s}_{\sw} &=\Run^{\sw}, \\
\nabla \cdot \point{V}{s}_{\sw} + \frac{1}{1-\beta b + \epsilon \eta_{\sw}}\partial_r \point{w}{s}_{\sw} &= \Rtrois^{\sw}, \\
\partial_t \point{\eta}{s}_{\sw} + \epsilon V_{\sw} \cdot \nabla  \point{\eta}{s}_{\sw}  + (w_{\sw})_{|r=0}^s & = \Rquatre^{\sw},\\
	} \qquad \text{ in } \R^d,
\end{equation} 
with
\begin{equation}
\label{eqn:nslw:quasilin:restes}
\begin{aligned}
\Run^{\sw} &:= \epsilon [\diff,V_{\sw}] \cdot \nabla V_{\sw},\\
\Rtrois^{\sw} &:= [\diff,\frac{1}{1-\beta b + \epsilon \eta_{\sw}}] \partial_r w_{\sw},\\
\Rquatre^{\sw} &:= [\diff,V_{\sw} \cdot \nabla, \eta_{\sw}].\\
\end{aligned}
\end{equation}
Recall now the quasilinear system~\eqref{eqn:euler_quasilin} satisfied by Alinhac's good unknowns $(\pointal{V}{s},\pointal{w}{s})$, and $\point{\eta}{s}_0$. Taking the difference between~\eqref{eqn:nlsw:quasilin} and~\eqref{eqn:euler_quasilin} yields
\begin{equation}
\label{eqn:euler_diff:quasilin}
\sys{
\partialphi{t} \Vdiffs + \epsilon \vec{U} \cdot \gradphi[x,r] \Vdiffs + \frac{1}{\rhotot} \gradphi \Pdiffs + \frac{g\rhob}{\rhotot} \nabla \etadiffs &= \tilde{\Run} \\
\mu\left( \partialphi{t} \wdiffs + \epsilon \vec{U} \cdot \gradphi[x,r] \wdiffs \right) + \frac{1}{\rhotot}\partialphi{r}\Pdiffs  + \delta \frac{\rho}{\rhotot}&= \sqrt{\mu} \tilde{\Rdeux},\\
\gradphi \cdot \Vdiffs + \partialphi{r} \wdiffs &= \tilde{\Rtrois},
}
\end{equation}
where
\begin{equation}
\label{eqn:source:diff}
\begin{aligned}
\tilde{\Run} &:= \Run - \Run^{\sw}- \epsilon \Vdiffs \cdot \gradphi V_{\sw} - \delta \epsilon \frac{\rho}{\rhotot},\\
\tilde{\Rdeux} &:= \Rdeux-\sqrt{\mu} \left( \partialphi{t} \point{w}{s}_{\sw} + \epsilon U \cdot \gradphi[x,r] \point{w}{s}_{\sw} \right),\\
\tilde{\Rtrois} &= \Rtrois - \Rtrois^{\sw} + \epsilon \frac{(1+r)\etadiffs}{(\hb + \epsilon h)(\hb + \epsilon h_{\sw})} \partial_r w_{\sw},
\end{aligned}
\end{equation}
where $h_{\sw}$ is defined from $\eta_{\sw}$ through~\eqref{eqn:h:def}. Here, $\pointal{P}{s}$ is Alinhac's good unknown defined from $\varphi$ and $P$ through~\eqref{eqn:alinhac:def}, and $\Run,\Rdeux,\Rtrois$ are defined through~\eqref{eqn:euler_quasilin:restes}.
This set of equations is completed with the boundary conditions
\begin{equation}
\label{eqn:euler_phi:diff:bc}
\sys{ P_{r=0} &= 0, \\
	\partial_t \etadiffs + \epsilon V_{|r=0} \cdot \nabla \etadiffs - \wdiffs_{|r=0} &= \tilde{\Rquatre}, \\
	\vec{N_b} \cdot \vec{\tilde{U}}_{|r=-1} &= \tilde{\Rcinq},}
\end{equation}
where
\begin{equation}
\label{eqn:source:diff:eta}
\begin{aligned}
\tilde{\Rquatre} &:= \Rquatre - \Rquatre^{\sw} - \epsilon \Vdiffs_{r=0} \cdot \nabla \eta_{\sw},\\
\tilde{\Rcinq} &:= \Rcinq,
\end{aligned}
\end{equation}
and $\Rquatre,\Rcinq$ are defined through~\eqref{eqn:euler_quasilin:bc:restes}. We now define the following energy functional
\begin{equation}
\label{eqn:energy:Ediff:def}
\begin{aligned}
\Ediff &:=  \Edifflow + \Vert(\sqrt{(\hb + \epsilon h)(\rhotot)}\Vdiffs,\sqrt{\mu(\hb + \epsilon h)(\rhotot)}\wdiffs, \sqrt{\taylor} \etadiffs \Vert_{L^2}^2 \\
&+ \sum\limits_{k \leq s} \Vert \sqrt{\mu(\hb + \epsilon h)} \pointal{\rho}{(s,k)} \Vert_{L^2}^2 + \Vert \vortdiff \Vert_{H^{s-1}(S)}^2,
\end{aligned}
\end{equation}
where 
\begin{equation}
\label{eqn:energy:def:Edifflow}
\Edifflow := \Vert(\Vdiff,\sqrt{\mu} \wdiff)\Vert_{H^{s_0+1,2}}^2 + \Vert \sqrt{\mu} \rho\Vert_{L^2}^2 +   g\vert  \etadiff \vert_{H^{s_0+1}}^2.
\end{equation}
Lemma~\ref{lemma:equiv:Ediff} states that this functional is equivalent to the $H^s$-norm of $(\Vdiff,\wdiff,\rho,\etadiff)$, and in this sense, is analogous to Lemma~\ref{lemma:energy:equiv}.
\begin{lemma}
\label{lemma:equiv:Ediff}
Under Assumptions~\eqref{hyp:H},\eqref{hyp:taylor},\eqref{hyp:M}, for any $s\geq s_0+2$, $\Ediff$ is equivalent to $\Vert \Vdiff,\sqrt{\mu} \wdiff,\sqrt{\mu} \rho \Vert_{H^s}$ $+ \vert \etadiff \vert_{H^s}$, that is there exists $C > 0$ such that
\begin{equation}
\label{eqn:energy:equiv:Ediff}
\frac{1}{C} \left(\Vert \Vdiff,\sqrt{\mu} \wdiff,\sqrt{\mu} \rho \Vert_{H^s}^2  + \Vert \vortdiff \Vert_{H^{s-1}}^2+ \vert \etadiff \vert_{H^s}^2 + \mu\right) \leq \Ediff \leq C \left(\Vert \Vdiff,\sqrt{\mu} \wdiff,\sqrt{\mu} \rho \Vert_{H^s}^2  + \Vert \vortdiff \Vert_{H^{s-1}}^2+ \vert \etadiff \vert_{H^s}^2 + \mu \right) 
\end{equation}
\end{lemma}
\begin{proof}
First note that
$$\Vdiffs = \pointal{V}{s} - \point{V}{s}_{\sw} = \diff \Vdiff - \frac{\diff (\etab + \epsilon \eta)}{\hb + \epsilon h} \partial_r \Vdiff,$$
as $V_{\sw}$ does not depend on $r$. Thus, $\Vdiffs$ is Alinhac's good unknown associated to $\Vdiff$. Now, notice that
$$\wdiffs = \pointal{w}{s} - \point{w}{s}_{\sw} = \diff \wdiff - \frac{\diff (\etab + \epsilon \eta)}{\hb + \epsilon h} \partial_r \wdiff -\frac{\diff (\etab + \epsilon \eta)}{\hb + \epsilon h} \partial_r w_{\sw}.$$
Therefore, $\sqrt{\mu} \wdiffs$ is Alinhac's good unknown associated to $\sqrt{\mu}\wdiff$, up to an error term of order $\sqrt{\mu}$. We can thus write:
$$\Ediff = \E + \mu C,$$
where $C$ depends only on $M$ and $M_{\sw}$ and $\E$ is defined through~\eqref{eqn:energy:def}. Then the estimate~\eqref{eqn:energy:equiv} yields the result.
\end{proof}
The remaining part of the proof is devoted to the proof of the following estimate
\begin{equation}
\label{eqn:energy:Ediff}
\frac{d}{dt} \Ediff \leq C (\epsilon \vee \beta \vee \frac{\delta}{\mu}) \Ediff + C \sqrt{\mu} \Ediff^{\frac12},
\end{equation}
as from~\eqref{eqn:energy:Ediff}, Lemma~\ref{lemma:equiv:Ediff} and Gönwall's lemma imply~\eqref{eqn:euler_diff:estimate}, which in turn implies the estimate in Theorem~\ref{thm:convergence:svt} thanks to~\eqref{hyp:E0}. \\
We start by estimating the remainder terms in~\eqref{eqn:euler_diff:quasilin},~\eqref{eqn:euler_phi:diff:bc} as well as in~\eqref{eqn:euler_diff:ordre0:restes}.
\begin{lemma}
\label{lemma:reste:quasilin:diff}
Under the assumptions of Theorem~\ref{thm:convergence:svt}, there exists $C$ depending only on $M$, $M_{\sw}$ such that
\begin{equation}
\label{eqn:pression:estimee:diff}
\Vert \tilde{P} \Vert_{H^s} \leq C \Ediff^{\frac12} + C \sqrt{\mu},
\end{equation}
as well as
\begin{equation}
\label{eqn:restes:estimee:diff}
\Vert (\tilde{\Run},\tilde{\Rdeux},\tilde{\Rtrois},\tilde{\Rsix})\Vert_{L^2}  + \vert \tilde{\Rquatre} \vert_{L^2} + \vert \tilde{\Rcinq} \vert_{H^{\frac12}} \leq (\epsilon \vee \beta\vee\frac{\delta}{\mu})C \Ediff^{\frac12} + C \sqrt{\mu}.
\end{equation}
and
\begin{equation}
\label{eqn:restes:estimee:diff:0}
\Vert (\Fun,\Fdeux,\Ftrois,\Fquatre,\Fcinq)\Vert_{H^{s-1}}  \leq (\epsilon \vee \beta\vee\frac{\delta}{\mu}) C \Ediff^{\frac12} + C \sqrt{\mu}.
\end{equation}
\end{lemma}
\begin{proof}
The estimate~\eqref{eqn:restes:estimee:diff:0} comes from the product, commutator estimates~\eqref{apdx:eqn:pdt:algb},~\eqref{apdx:eqn:commutator:gen}. Similar estimations have already been performed in Proposition~\ref{prop:restes} and we refer to them for the details. Note however that the operator $\gradphi$ in $\Fdeux$ is of order $1$, and is responsible for the loss of derivatives in~\eqref{eqn:restes:estimee:diff:0}.\\ 
We now prove~\eqref{eqn:pression:estimee:diff}. Note that as $(\pointal{V}{s},\pointal{w}{s},\pointal{P}{s})$ satisfy the system~\eqref{eqn:euler_diff:quasilin}, which is of the form~\eqref{eqn:euler_phi} with additional source terms, we can write an elliptic equation satisfied by $\pointal{P}{s}$. It is of the form of~\eqref{eqn:pression}, with additional source terms. A slight adaptation of the pressure estimate~\eqref{eqn:pression:estimee} to take these source terms into account yields
\begin{equation}
\label{eqn:temp:8}
\begin{aligned}
\Vert \sqrt{\mu} \gradphi \tilde{P}, \partialphi{r} \tilde{P} \Vert_{H^{s}} \leq C \sqrt{\mu} &( \Vert \Vdiff, \sqrt{\mu} \wdiff,\sqrt{\mu} \rho \Vert_{H^{s+1}} + \Vert \vortdiff  \Vert_{H^{s}} \\
&+ \sqrt{\mu}|\etadiff|_{H^{s+1}} + \Vert(\Fun,\sqrt{\mu} \Fdeux,\Ftrois,\partialphi{t}\Fquatre) \Vert_{H^{s}} ),
\end{aligned}
\end{equation}
with $(\Fun,\Fdeux,\Ftrois,\Fquatre)$ defined in~\eqref{eqn:euler_diff:ordre0:restes}.
Using~\eqref{eqn:restes:estimee:diff:0} to bound the source terms in~\eqref{eqn:temp:8}, this yields~\eqref{eqn:pression:estimee:diff}.\\
For the estimate~\eqref{eqn:restes:estimee:diff}, we only treat $\tilde{\Run}$, the other terms are treated the same way. Let us expand the terms in $\tilde{\Run}$ to get: 
\begin{equation}
\label{eqn:temp:7}
\begin{aligned}
\tilde{\Run} &= \commal{\diff}{\partialphi{t}}{V}   + \epsilon \vec{U} \cdot \commal{\diff}{\gradphi[x,r]}{V} + \frac{1}{\rhotot}\commal{\diff}{\gradphi}{P} \\ 
 &+ \epsilon [\diff , V \cdot] \gradphi V- \epsilon [\diff,V_{\sw}] \cdot \nabla V_{\sw} \\
 & +  \epsilon [\diff ; w ,\partialphi{r} V] - \epsilon \Vdiffs \cdot \gradphi V_{\sw}\\
 &+ [\diff, \frac{1}{\rhotot}] \gradphi P + [\diff, \frac{g\rhob}{\rhotot}]\nabla \eta_0\\
 &- \delta \epsilon \frac{\rho}{\rhotot}.\\
\end{aligned}
\end{equation}
For the first term, we write
$$ \commal{\diff}{\partialphi{t}}{V} = \diff (\etab + \epsilon \eta) \partialphi{t} \partialphi{r} V + (\partial \varphi)^{-T} [\diff;(\partial \varphi)^T-I_{d+2},\gradphi[t,x,r] V],$$
where we made $[\diff;I_{d+2},\gradphi[t,x,r] V] = 0$ appear. Thus, according to the particular form of $(\partial \varphi)^T$, the symmetric commutator only depends on $\partialphi{r}V$:
$$ \begin{aligned}\commal{\diff}{\partialphi{t}}{V} &= \diff (\etab + \epsilon \eta) \partialphi{t} \partialphi{r} V \\
&+ (\partial \varphi)^{-T} [\diff;(\epsilon \partial_t \eta, (\nabla(\etab+\epsilon \eta))^T, -\beta b + \epsilon \eta_0)^T ,\partialphi{r}(V-V_{\sw})], \end{aligned}$$
where we used $\partialphi{r} V_{\sw}=0$. Thus, we get by~\eqref{apdx:alinhac:R1}:
$$\Vert \commal{\diff}{\partialphi{t}}{V} \Vert_{L^2} \leq C (\epsilon \vee \beta) \Ediff.$$
The second and third terms in~\eqref{eqn:temp:7} are treated the same way, with the estimate~\eqref{eqn:pression:estimee:diff} for the third one. The sixth term in~\eqref{eqn:temp:7} is treated the same way as it only depends on $\partialphi{r} V$. \\
For the fourth and fifth terms in~\eqref{eqn:temp:7}, we write
$$ [\diff , V \cdot] \gradphi V- [\diff,V_{\sw}] \cdot \nabla V_{\sw} = [\diff , V \cdot] \gradphi \Vdiff - [\diff,\Vdiff] \cdot \nabla V_{\sw}.$$
The commutator estimate~\eqref{apdx:eqn:commutator:gen} then yields
$$ \Vert [\diff , V \cdot] \gradphi V- [\diff,V_{\sw}] \cdot \nabla V_{\sw} \Vert_{L^2} \leq C \Ediff^{\frac12}.$$
The last terms in~\eqref{eqn:temp:7} are bounded by the product estimate~\eqref{apdx:eqn:pdt:algb}, commutator estimate~\eqref{apdx:eqn:commutator:horizontal} as well as the estimate~\eqref{eqn:pression:estimee:diff} on the pressure.
\end{proof}
We are now in a position to prove~\eqref{eqn:energy:Ediff}. 
Applying $\Lambda^{s_0+1-k}\partial_r^k$ to~\eqref{eqn:euler_diff:ordre0} for $0\leq k \leq s_0+1$ yields a quasilinear system of the form~\eqref{eqn:quasilin:gen}. The estimate~\eqref{eqn:quasilin:gen:energy:L2} together with ~\eqref{eqn:restes:estimee:diff:0} then yields
$$\frac{d}{dt} \Edifflow \leq C(\epsilon \vee \beta\vee\frac{\delta}{\mu})\Ediff + C\sqrt{\mu} \Ediff^{\frac12}. $$
The computations are the same as in the Step 1 of the proof of Proposition~\ref{prop:energy} and we refer the reader to this part for more details. \\
Now, the system~\eqref{eqn:euler_diff:quasilin} is of the form~\eqref{eqn:quasilin:gen}. Applying~\eqref{eqn:quasilin:gen:energy:L2} and the estimates~\eqref{eqn:restes:estimee:diff}, we readily get
$$\frac{d}{dt} \left(\Vert((\hb + \epsilon h)\Vdiffs,\sqrt{\mu} (\hb + \epsilon h)\wdiffs, \taylor \etadiffs \Vert_{L^2}^2\right) \leq C(\epsilon \vee \beta\vee\frac{\delta}{\mu})\Ediff + C\sqrt{\mu}\Ediff^{\frac12}. $$
For the estimate on the vorticity, we notice that~\eqref{eqn:euler_diff:vort} is exactly of the form~\eqref{eqn:quasilin:vort:gen}. Therefore, the estimate~\eqref{eqn:quasilin:gen:vort:energy:L2} yields
$$\frac{d}{dt} \Vert \vortdiff \Vert_{H^{s-1}}^2 \leq C (\epsilon \vee \beta\vee\frac{\delta}{\mu}) \Ediff + \Vert \Fcinq \Vert_{H^{s-1}}\Vert \vortdiff \Vert_{H^{s-1}},$$
Where $\Fcinq$ is defined through~\eqref{eqn:vort:diff:reste}. We readily get, by the product estimate~\eqref{apdx:eqn:pdt:algb}:
$$\Vert \Fcinq \Vert_{H^{s-1}} \leq C \epsilon \Ediff^{\frac12}.$$
\end{proof}

\subsection{A note on a larger time of existence for~\eqref{eqn:euler_phi}}
\label{subsection:cheating}
In this subsection, we use the result in~\cite{BreschMetivier10} that the solution $(V_{\sw},\eta_{\sw})$ of~\eqref{eqn:nlsw} exists on a time of order $\frac{T}{\epsilon}$, even in the presence of large topographic variations. As explained in Remark~\ref{rk:tps_long}, this is a better result than the one in Theorem~\ref{thm:euler_phi}. We can use the result in~\cite{BreschMetivier10} to improve the time of existence in Theorem~\ref{thm:euler_phi}, although with well-prepared initial data (Assumption~\eqref{hyp:E0}), and state this result in Corollary~\ref{cor:tps_long}. In the specific regime described by Assumption~\eqref{hyp:mu}, this boils down to a time of order $T \log(\frac{1}{\epsilon})$. This is much smaller that the large time $\frac{T}{\epsilon}$ from~\cite{BreschMetivier10}, and is called the Ehrenfest time in semi-classical analysis~\cite{Shepelyansky2020}.\\
In order to simplify the notations, we now assume $\beta = 1$ (large topographic variations) and $\delta \leq \mu$ (weak density variations). \\
Let $(V_{\sw},\eta_{\sw}) \in C^0([0,T/\epsilon[,(H^{s+1})^{d+1}(\R^d))$ the solution to the equations~\eqref{eqn:nlsw} with initial conditions in $H^{s+1}$, and let $M_{\sw}>0$ such that~\eqref{hyp:Msw} is satisfied.
In the following corollary, we assume that the initial data~\eqref{eqn:euler_euleriennes:ci} is well-prepared in the sense of~\eqref{hyp:E0}, and assume the regime
\begin{hyp}
\label{hyp:mu}
\mu \leq \epsilon.
\end{hyp}We take $M>0$ such that 
\begin{hyp}
\label{hyp:M:tpslong}
\Vert V_{\ini},\sqrt{\mu}w_{\ini},\sqrt{\mu} \rho_{\ini} \vert_{H^s} + \vert (\eta_0)_{\ini}\vert_{H^s} \leq M.
\end{hyp}We are in a position to state and prove the following corollary.
\begin{corollary_alpha}
\label{cor:tps_long}
Let $s_0,s \in \N$, $s_0 > \frac{d+1}{2}$, $s\geq s_0+2$. Let $\mu > 0$, $\beta=1$, $\epsilon,\delta \geq 0$ with $\delta \leq \mu$. Let $(\eta_0)_{\ini} \in H^s(\R^d)$, $(V_{\ini},w_{\ini}) \in H^s(S)^{d+1}$  satisfying~\eqref{hyp:H},\eqref{hyp:taylor}, the incompressibility condition in~\eqref{eqn:euler_phi:bc}, the well-preparation~\eqref{hyp:E0}, the regime~\eqref{hyp:mu} and~\eqref{hyp:M:tpslong}.
Then the solution of~\eqref{eqn:euler_phi} $(V,w,\rho,\eta_0)$ given by Theorem~\ref{thm:euler_phi} is in $C^0([0,\log(\frac{1}{\epsilon}) T],H^{s}(S)^{d+2} \times H^s(\R^d))$.
\end{corollary_alpha}
\begin{remark}
The proof of Corollary~\ref{cor:tps_long} yields a slightly more precise result. Indeed, we do not use the assumption~\eqref{hyp:mu}, and in this case the time of existence is $(\log(\frac{1}{\mu}) \wedge \frac{1}{\epsilon}) T$.
\end{remark}
\begin{proof}
We prove the slightly more precise time of existence $(\log(\frac{1}{\mu}) \wedge \frac{1}{\epsilon}) T$, which boils down to $\log\frac{1}{\epsilon}$ if moreover we assume $\mu \leq \epsilon$.\\
First, the time of existence in Corollary~\ref{cor:tps_long} cannot exceed $\frac{1}{\epsilon}$, since we compare the solution of~\eqref{eqn:euler_phi} to the solution of~\eqref{eqn:nlsw}, the latter one being defined on a time interval of length $\frac{1}{\epsilon}$ thanks to~\cite{BreschMetivier10}. We thus assume $\log(\frac{1}{\mu}) \leq \frac{1}{\epsilon}$ for the remaining of the proof.
Let
\begin{equation}
\label{eqn:blowup_with_cte}
\begin{aligned}
T_* := \sup\{ T \leq \frac{1}{\epsilon},
 &\sup_{0 \leq t\leq T} \Vert (V,\sqrt{\mu}w,\sqrt{\mu} \rho)(t,\cdot) \Vert_{H^s} + \vert \eta_0(t,\cdot) \vert_{H^s} \leq 2M,\\
&\inf_{t\leq T} \inf_{(x,r)\in S} \taylor \geq c_*/2 \},\\
\end{aligned}
\end{equation}
so that in particular the maximal time of existence of the solution $(V,w,\rho,\eta_0)$ to~\eqref{eqn:euler_phi} $T_{\max}$ satisfies $T_{\max} \geq T_*$ by the blow-up criterion from Theorem~\ref{thm:euler_phi}. 

We start from the estimate~\eqref{eqn:euler_diff:estimate} in the special case where $\delta\leq\mu,\beta=1$ and with the well-preparation~\eqref{hyp:E0}, to get, by triangular inequality
\begin{equation}
\label{eqn:cheating:temp}
\Vert (V,\sqrt{\mu}w,\sqrt{\mu} \rho)(t,\cdot) \Vert_{H^s} + \vert \eta_0(t,\cdot) \vert_{H^s} \leq \Vert (V_{\sw},\sqrt{\mu}w_{\sw})(t,\cdot) \Vert_{H^s} + \vert \eta_{\sw}(t,\cdot) \vert_{H^s} + C \sqrt{\mu} e^{Ct}.
\end{equation}
Evaluating~\eqref{eqn:cheating:temp} at time $t$ with $t\leq \frac{1}{4C}\log(\frac{1}{\mu})$, the first term of the right-hand side of~\eqref{eqn:cheating:temp} stays bounded by $2M$ thanks to~\cite[Theorem 2.1]{BreschMetivier10}, and $\sqrt{\mu}( 1+ \frac{1}{4C}|\log \mu |) e^{\log(\frac{1}{\mu^{1/4}})}$ is bounded from above, uniformly in $\mu$. Hence the first condition in the definition of the set in~\eqref{eqn:blowup_with_cte} can only fail after the time $t_*:=\frac{1}{4C}\log(\frac{1}{\mu})$.\\
For the second condition in~\eqref{eqn:blowup_with_cte}, we use the bound on $\taylor$~\eqref{eqn:taylor:dt} to get
\begin{equation}
\label{eqn:cheating:temp:2}
\vert \taylor(t,\cdot) - g\rhob \vert_{L^{\infty}} \leq C \epsilon (\Vert (V,\sqrt{\mu} w,\sqrt{\mu} \rho)\Vert_{H^{s_0+2}} + \vert \eta_0 \vert_{H^{s_0+1}}).
\end{equation}
Together with the bound~\eqref{eqn:cheating:temp}, this yields
\begin{equation}
\label{eqn:cheating:temp:3}
\sup_{t\leq t_*}\vert \taylor(t,\cdot) - g\rhob \vert_{L^{\infty}} \leq C M \epsilon.
\end{equation}
Then $\inf\limits_{t\leq t_*}\inf\limits_{x\in \R^d}\taylor(t,x) \geq \frac12 c_*$, if $\epsilon \leq (g\rhob - c_*/2)/(CM)$. Hence the second condition in~\eqref{eqn:blowup_with_cte} can only fail after the time $t_*$. This eventually yields $T_{\max} \geq T_* \geq t_* = \frac{1}{2C} \log(\frac{1}{\mu})$ and this concludes the proof.
\end{proof}


\appendix
\section{Product, commutator and composition estimates}
\label{appendix:A}
We now state technical results used in this study.
\begin{lemma}[Integration by parts]
\label{apdx:lemma:IPP}
If $(\vec{F},g)  \in H^{1}(S)^{d+1}\times H^{1}(S) $, then\\
\begin{equation}
\label{apdx:eqn:IPP}
\begin{aligned}
\int_{S} (\hb+\epsilon h) g \gradphi[x,r] \cdot \vec{F} dx dr= &- \int_{\R^d} \vec{F}_{|r=0} \cdot\coo{\epsilon \nabla \eta_0 \\ -1}  g_{|r=0} dx+ \int_{\R^d} \vec{F}_{|r=-1} \cdot\coo{\beta \nabla b \\ -1} g_{|r=-1} dx  \\
&-\int_{S} (\hb+\epsilon h) \vec{F} \cdot  \gradphi[x,r] g dx dr.
\end{aligned}
\end{equation}
\end{lemma}
\begin{lemma}[Equivalence of norms by change of variables]
Let $t_0 > d/2$, $s \in \N$ with $s \geq t_0+\frac32$, $f \in H^{s}(S)$. Let $\eta \in H^s(S)$, satisfying~\eqref{hyp:H} as well as
$$\Vert \eta \Vert_{H^s}\leq M,$$
for some $M > 0$. We write $\check{f}(t,x,r) := f(t,x,r+\eta(t,x,r))$. Then there exists $C > 0$ depending only on $M$ such that
\begin{equation}
\label{apdx:eqn:equiv}
\Vert \check{f}\Vert_{H^s} \leq C \Vert f\Vert_{H^s}.
\end{equation}
\end{lemma}
\begin{proof}
Let $\alpha \in \N^{d+1}$ a multi-index of norm $|\alpha| \leq s$. We write
$$\partial^{\alpha} \check{f} = ((\partial^{\varphi})^{\alpha} f)\circ \varphi,$$
where $\partialphi{}$ is defined through~\eqref{eqn:gradphi:def}. Thus, by the product estimate~\eqref{apdx:eqn:pdt:tame}, we get~\eqref{apdx:eqn:equiv}.
\end{proof}
\begin{lemma}[A continuous embedding for anisotropic spaces]
\label{apdx:lemma:embedding}
Let $s \in \R$, and $f \in H^{s+\frac{1}{2},1}$. Then $f \in C^0([-1,0],H^s(\R^d))$ and moreover, there exists $C > 0$ such that
\begin{equation}
\label{apdx:eqn:embedding}
 \max\limits_{0 \leq  r \leq 1}|f(\cdot, r) |_{H^s} \leq C \Vert f \Vert_{H^{s+\frac12,1}}. 
\end{equation}
In particular, denoting by $f_{|r=-1}$ the trace of $f$ at the bottom of the strip $S$ (resp. $f_{|r=-1}$ at the surface), we can write
\begin{equation}
\label{apdx:eqn:trace}
\vert f_{|r=-1} \vert_{H^s(\R^d)} + \vert f_{|r=0} \vert_{H^s(\R^d)}\leq C \Vert f \Vert_{H^{s+\frac12,1}(S)},
\end{equation}
\end{lemma}
\begin{proof}
See for instance~\cite[Lemma A.1]{Duchene2022} or~\cite[Proposition 2.12]{Lannes2013}.
\end{proof}

\begin{lemma}[Standard product estimates in Sobolev spaces $H^s(\R^d)$]
Let $d \in \N^*$, $t_0 > d/2$, $s\in \R$, $s_1,s_2 \in \R$ such that $s\leq s_1$, $s\leq s_2$, $s_1 + s_2 \geq 0$ and $s \leq s_1 + s_2 - t_0$. There exists a constant $C > 0$ such that for any $f \in H^{s_1}(\R^d)$, $g \in H^{s_2}(\R^d)$, we can write
\begin{equation}
\label{apdx:eqn:pdt:s1s2}
 | fg |_{H^s(\R^d)} \leq C |f|_{H^{s_1}} |g|_{H^{s_2}}. 
\end{equation}
\end{lemma}
\begin{proof}
See for instance~\cite[Proposition B.2]{Lannes2013} and references therein.
\end{proof}

\begin{lemma}[Product estimates in anisotropic spaces]
\label{apdx:pdt}
Let $d\in \N^*$, $t_0 > d/2$, $s \in \R$ and $k \in \N$ with $0 \leq k \leq s$.
\begin{itemize}[label=\textbullet]
\item There exists $C > 0$ such that for any $f \in L^2(S_r)$ and $g \in H^{t_0+\frac12,1}(S_r)$,  we can write
\begin{equation}
\label{apdx:eqn:pdt:tame}
\Vert fg\Vert_{L^2} \leq C \Vert f \Vert_{L^2} \Vert g \Vert_{H^{t_0+1/2,1}}.
\end{equation}
\item If $s \geq t_0+\frac12$, $k \geq 1$, then there exists $C > 0$ such that for any $g\in H^{s,k}$, $f \in H^{s,k} $
\begin{equation}
\label{apdx:eqn:pdt:algb}
\Vert fg\Vert_{H^{s,k}} \leq C \Vert f \Vert_{H^{s,k}} \Vert g \Vert_{H^{s,k}}.
\end{equation}
\item If $s \geq t_0+\frac32$, $k \geq 2$, then there exists $C > 0$ such that for any $g\in H^{s,k}$, $f \in H^{s,k} $
\begin{equation}
\label{apdx:eqn:pdt:algb:tame}
\Vert fg\Vert_{H^{s,k}} \leq C \Vert f \Vert_{H^{s-1,k-1}} \Vert g \Vert_{H^{s,k}} + C \Vert f \Vert_{H^{s,k}} \Vert g \Vert_{H^{s-1,k-1}}.
\end{equation}
\item If $f \in W^{k,\infty}_r$ is independent of $x$ and $g \in H^{s,k}$, then 
\begin{equation}
\label{apdx:eqn:pdt:infty}
 \Vert fg \Vert_{H^{s,k}} \leq C | f |_{W^{k,\infty}} \Vert g \Vert_{H^{s,k}}.
 \end{equation}
\end{itemize}
\end{lemma}
\begin{proof}
See for instance~\cite[Lemma A.4]{Fradin2024} and references therein.
\end{proof}
We now state commutator estimate in $\R^d$. Note that in the case $s' = 0$,~\eqref{apdx:eqn:commutator:Rd:sprime} boils down to the classical commutator estimate that can be found, for instance, in~\cite[Proposition B.8]{Lannes2013}.
\begin{lemma}
Let $d \geq 1$, $t_0 > d/2$, $s \geq 0$.
\begin{itemize}
\item If $s_1 \geq s$ and $s_2 \geq s-1$ with $s_1+s_2 \geq s+t_0$, then there exists $C > 0$ such that
\begin{equation}
\label{apdx:eqn:commutator:Rd:classic}
|[\Lambda^s,f]g|_{L^2} \leq C | f |_{H^{s_1}} |g|_{H^{s_2}}.
\end{equation}
\item Let $s' \geq 0$. If moreover $s \geq t_0+1$, Then
\begin{equation}
\label{apdx:eqn:commutator:Rd:sprime}
|[\Lambda^s,f]g|_{H^{s'}} \leq C | f |_{H^{s+s'}} |g|_{H^{s+s'-1}}.
\end{equation}
\end{itemize}
\end{lemma}
\begin{proof}
The first inequality is classical. For the second one, we write
$$ \Lambda^{s'}( [\Lambda^s,f]g) = [\Lambda^{s+s'},f] g + [\Lambda^{s'},f] \Lambda^s g,$$
by definition of the commutator (see Subsection~\ref{subsection:notations}.
Using~\eqref{apdx:eqn:commutator:Rd:classic} with $s_1 = s+s'$ and $s_2 = s+s' -1$ to estimate the first term, and with $s_1 = s' + t_0+1$, $s_2 = s' -1$ for the second term, we get the result. 
\end{proof}
We now state commutator estimates in $H^{s,k}$ spaces.
\begin{lemma}[Commutator estimates] 
\label{apdx:commutator}
Let $d\in \N^*$, $t_0 > d/2$, $s >0$ and $k \in \N$ with $0 \leq k \leq s$.
 \begin{itemize}[label=\textbullet]
\item If $l\in \N$ with $0 \leq l \leq k$, then there exists $C > 0$ such that for any $f \in H^{s\vee(t_0+\frac32),k\vee2}$ and $g \in H^{(s-1)\vee(t_0+\frac12),2\vee(k\wedge(s-1))}$, we can write 
\begin{equation}
\label{apdx:eqn:commutator:gen}
\Vert [\Lambda^{s-l} \partial_r^l,f]g \Vert_{L^2} \leq C \Vert f \Vert_{H^{s\vee(t_0+\frac32),k\vee2}} \Vert g \Vert_{H^{(s-1)\vee(t_0+\frac12),2\vee(k\wedge(s-1))}}.
\end{equation}

\item If $s \geq  1$, then there exists $C > 0$ such that for any $f \in H^{s,0} \cap H^{t_0+\frac32,1}$, $g \in H^{s-1,0}$, we can write 
\begin{equation}
\label{apdx:eqn:commutator:horizontal}
\Vert [\dot{\Lambda}^{s},f]g \Vert_{L^2} \leq C \Vert \nabla_x f \Vert_{H^{t_0+\frac12,1}} \Vert g \Vert_{H^{(s-1)\vee(t_0+\frac12),0}}  + C \langle \Vert \nabla_x f \Vert_{H^{(s-1)\vee(t_0+\frac12),0}} \Vert g \Vert_{H^{t_0+\frac12,1}}\rangle_{s > t_0+1}.
\end{equation}

\item Let $k \geq 1$. There exists $C > 0$ such that if $f \in W^{k,\infty}_r$ and $g \in H^{s,k}$, then
\begin{equation}
\label{apdx:commutator:infty}
 \Vert [\Lambda^{s-k} \partial_r^k,f]g \Vert_{L^2} \leq C | f |_{W^{k,\infty}} \Vert g \Vert_{H^{s-1,k-1}}. 
\end{equation}
\end{itemize}
\end{lemma}
\begin{proof}
See~\cite[Lemma A.5]{Fradin2024} and references therein.
\end{proof}

We now state a similar lemma on symmetric commutators.
\begin{lemma}[Symmetric commutator estimates]
\label{apdx:commutator_sym}
Let $d\in \N^*$, $t_0 > d/2$, $s \in \R$  with $s \geq t_0+\frac52$ and $k \in \N$ with $2 \leq k \leq s$.

\begin{itemize}[label=\textbullet]
\item 
If $l\in \N$ with $0 \leq l \leq k$. Then there exists $C > 0$ such that for any $f \in H^{s,k}$, $g \in H^{s-1,k\wedge(s-1)}$, we can write 
\begin{equation}
\label{apdx:eqn:commutator_sym:gen}
\Vert [\Lambda^{s-l} \partial_r^l;f,g] \Vert_{L^2} \leq C \Vert f \Vert_{H^{s-1,k\wedge(s-1)}} \Vert g \Vert_{H^{s-1,k\wedge(s-1)}}.
\end{equation}
\item 
For $s \in \R$ and $ s \geq 1$, $f \in H^{s-1,0} \cap H^{t_0+\frac32,1}$, $g \in H^{s-1,0} \cap H^{t_0+\frac32,1}$ we can write 
\begin{equation}
\label{apdx:eqn:commutator_sym:horizontal}
\Vert [\dot{\Lambda}^{s};f,g] \Vert_{L^2} \leq C \Vert f \Vert_{H^{s-1,0}} \Vert g \Vert_{H^{t_0+\frac32,1}} + C \Vert g \Vert_{H^{s-1,0}} \Vert f \Vert_{H^{t_0+\frac32,1}}.
\end{equation}
\end{itemize}
\end{lemma}
We conclude this appendix with composition estimates.
\begin{lemma}[Composition estimates in $H^{s,k}(S_r)$]
\label{apdx:composition}
Let $d \in \N^*$, $t_0 > d/2$. Let $s \in \R$ with $s \geq t_0+\frac12$, $k\in \N$ with $1 \leq k \leq s$, $M> 0$ . There exists $C > 0$ such that for any $F \in W^{s,k,\infty}(\R\times(0,1))$ with $F(0,\cdot) = 0$, and any $u \in H^{s,k}(S_r)$ such that $\Vert u \Vert_{H^{s,k}} \leq M$, the following holds. Define $F \circ u : (x,r) \mapsto F(u(x,r),r)$. Then
\begin{equation}
\label{apdx:eqn:composition:gen}
\Vert F \circ u \Vert_{H^{s,k}} \leq C | F |_{W^{s,k,\infty}(\R\times(0,1))} \Vert u \Vert_{H^{s,k}}.
\end{equation}
In the particular case $u = \epsilon h$ and $F(y) = \frac{y}{1+y}$, under assumption~\eqref{hyp:H}, we get
\begin{equation}
\label{apdx:eqn:composition:H}
\Vert \frac{\epsilon h}{\hb+\epsilon h} \Vert_{H^{s,k}} \leq \epsilon C \Vert h \Vert_{H^{s,k}}
\end{equation}
\end{lemma}
\begin{proof}
See for instance~\cite[Lemma A.5]{Duchene2022}.
\end{proof}

\section*{Acknowledgments}
The author thanks his PhD supervisors Vincent Duchêne and David Lannes for their valuable advice. \\
This work was supported by the BOURGEONS project, grant ANR-23-CE40-0014-01 of the French National Research Agency (ANR).\\
This work is licensed under \href{https://creativecommons.org/licenses/by/4.0/}{ CC BY 4.0}.
\bibliographystyle{amsalpha}
\bibliography{../Seafile/refs/refs}
\end{document}